\documentclass[10pt,oneside,a4paper]{amsart}
\usepackage{amsmath,amsfonts,amsthm, amssymb}
\usepackage{array}
\usepackage[all,textures]{xy}
\usepackage{portland}

\theoremstyle{plain}
\newtheorem{theorem}{Theorem}[section]
\newtheorem{proposition}[theorem]{Proposition}
\newtheorem{lemma}[theorem]{Lemma}
\newtheorem{corollary}[theorem]{Corollary}

\theoremstyle{definition}
\newtheorem{definition}[theorem]{Definition}

\theoremstyle{remark}

\newtheorem{remark}[theorem]{Remark}
\newtheorem{example}[theorem]{Example}

\newcommand{\Kern}{\mathrm{Ker}}

\newcommand{\Beeld}{\mathrm{Im}}

\newcommand{\Mor}{\mathrm{Mor}}
\renewcommand{\lim}{\mathrm{lim}}

\newcommand{\Ext}{\mathrm{Ext}}

\newcommand{\Hom}{\mathrm{Hom}}

\newcommand{\RHom}{\mathrm{RHom}}

\newcommand{\op}{^{\mathrm{op}}}
\newcommand{\Ob}{\mathrm{Ob}}

\newcommand{\N}{\mathbb{N}}

\newcommand{\AAA}{\mathfrak{a}}
\newcommand{\BBB}{\mathfrak{b}}

\newcommand{\CCC}{\mathfrak{c}}

\newcommand{\EEE}{\mathfrak{e}}

\newcommand{\XXX}{\mathfrak{x}}
\newcommand{\YYY}{\mathfrak{y}}

\newcommand{\ZZZ}{\mathfrak{z}}

\newcommand{\CC}{\mathbf{C}}
\newcommand{\Set}{\ensuremath{\mathsf{Set}} }
\newcommand{\Alg}{\ensuremath{\mathsf{Alg}} }

\newcommand{\Mod}{\ensuremath{\mathsf{Mod}} }
\newcommand{\Bimod}{\ensuremath{\mathsf{Bimod}} }

\newcommand{\Map}{\ensuremath{\mathsf{Map}} }
\newcommand{\Mas}{\ensuremath{\mathsf{Mas}} }

\newcommand{\Cat}{\ensuremath{\mathsf{Cat}} }

\newcommand{\Inj}{\ensuremath{\mathsf{Inj}}}

\newcommand{\Des}{\ensuremath{\mathrm{Des}}}

\newcommand{\lra}{\longrightarrow}
\newcommand{\ra}{\rightarrow}

\newcommand{\aaa}{\ensuremath{\mathcal{A}}}
\newcommand{\bbb}{\ensuremath{\mathcal{B}}}
\newcommand{\ccc}{\ensuremath{\mathcal{C}}}
\newcommand{\ddd}{\ensuremath{\mathcal{D}}}
\newcommand{\eee}{\ensuremath{\mathcal{E}}}

\newcommand{\nnn}{\ensuremath{\mathcal{N}}}
\newcommand{\ooo}{\ensuremath{\mathcal{O}}}

\newcommand{\rrr}{\ensuremath{\mathcal{R}}}
\newcommand{\sss}{\ensuremath{\mathcal{S}}}
\newcommand{\ttt}{\ensuremath{\mathcal{T}}}
\newcommand{\uuu}{\ensuremath{\mathcal{U}}}
\newcommand{\vvv}{\ensuremath{\mathcal{V}}}
\newcommand{\www}{\ensuremath{\mathcal{W}}}
\newcommand{\xxx}{\ensuremath{\mathcal{X}}}

\newcommand{\zzz}{\ensuremath{\mathcal{Z}}}
\CompileMatrices
\SelectTips{cm}{11}

\title{Hochschild cohomology with support}
\author{Wendy Lowen} 
\address[Wendy Lowen]{Universiteit Antwerpen, Departement Wiskunde-Informatica, Middelheimcampus,
Middelheimlaan 1,
2020 Antwerp, Belgium}
\email{wendy.lowen@ua.ac.be}
\thanks{The author acknowledges the support of the European Union for the ERC grant No 257004-HHNcdMir and the support of the Research Foundation Flanders (FWO) under Grant No G.0112.13N}

\begin{document}
\maketitle

\begin{abstract}
In this paper we investigate the functoriality properties of map-graded Hochschild complexes. We show that the category $\Map$ of map-graded categories is naturally a stack over the category of small categories endowed with a certain Grothendieck topology of $3$-covers. For a related topology of $\infty$-covers on the cartesian morphisms in $\Map$, we prove that taking map-graded Hochschild complexes defines a sheaf. From the functoriality related to ``injections'' between map-graded categories, we obtain Hochschild complexes ``with support''. We revisit Keller's arrow category argument from this perspective, and introduce and investigate a general Grothendieck construction which encompasses both the map-graded categories associated to presheaves of algebras and certain generalized arrow categories, which together constitute a pair of complementary tools for deconstructing Hochschild complexes.
\end{abstract}



\section{Introduction}

Hochschild cohomology originated as a cohomology theory for algebras $A$. While the degree $0$ cohomology is the center of $A$ and the degree $1$ cohomology corresponds to the derivations of $A$, the degree $2$ cohomology group parametrizes first order deformations of $A$. The relation between Hochschild cohomology and deformation theory is in fact more profound, and can be understood in terms of the structured Hochschild complex $\CC(A)$. In the mean time, Hochschild cohomology and Hochschild complexes have been defined for a wide range of objects of algebro-geometric nature, from schemes \cite{swan} and presheaves of algebras \cite{gerstenhaberschack2} to differential graded, exact \cite{kellerdih} and abelian categories \cite{lowenvandenberghhoch}, and various links with deformation theory have been established in these contexts \cite{gerstenhaberschack1} \cite{lowenvandenberghab}, \cite{lowenmap}.

An important shortcoming of Hochschild cohomology when compared, for instance, to Hochschild homology, is its lack of functoriality. Even for algebras, a morphism $f: A \lra B$ does not naturally give rise to a map between the Hochschild cohomologies $HH^{\ast}(A)$ and $HH^{\ast}(B)$ in either direction. When we turn from algebras to linear categories, i.e. algebras with several objects in the sense of \cite{mitchell}, the situation becomes somewhat better. More precisely, the inclusion of a full subcategory $\BBB \subseteq \AAA$ naturally gives rise to a ``restriction'' map $\CC(\AAA) \lra \CC(\BBB)$ between the Hochschild complexes. Furthermore, this observation of \emph{limited functoriality} provides a way to relate the Hochschild complexes of categories related by a bimodule, through an intermediate \emph{arrow category}. In \cite{kellerdih}, the \emph{arrow category argument} for structured ($B_{\infty}$-algebra) Hochschild complexes is developed by Keller in the context of dg categories. The main applications of this argument can be divided into two types:
\begin{enumerate}
\item Relating the Hochschild complex of a more involved object (the arrow category) to the Hochschild complexes of it's easier building blocks (two smaller categories and a bimodule relating them).
\item Proving that two objects have isomorphic Hochschild complexes in the homotopy category of $B_{\infty}$-algebras, by relating them via a suitable bimodule.
\end{enumerate}
Applications of type (1) can be seen as generalizations of results on the Hochschild cohomology of triangular matrix algebras, a topic which, since the work of Happel \cite{happel}, has received a lot of attention \cite{platzeck, guccione, green, snashall, cibils, guin, saorin}. An application of type (2) is the relation between the Hochschild complexes of Koszul dual algebras obtained in \cite{kellerdih}. The argument is also extensively used in \cite{lowenvandenberghhoch} to compare various candiate Hochschild complexes of abelian categories and ringed spaces. The paper \cite{lowenvandenberghhoch} also contains some results more in the spirit of (1), like the existence of Mayer-Vietoris sequences for Hochschild cohomology of ringed spaces.

The main aim of the current work is to provide a comprehensive treatment of the natural tools for ``breaking down'' complicated Hochschild complexes into easier pieces. 
In subsequent work, we will apply our results both to improve our understanding of deformation theory, for instance of schemes, and to the computation of Hochschild cohomology groups of various origins, for instance for singular schemes in the absence of the classical HKR decomposition.

The framework we choose for our exposition is that of map-graded categories \cite{lowenmap}. Although a more powerful theory can be obtained in the combined context of ``map-graded differential graded categories'', for simplicity we present our work in the context of linear map-graded categories. A map-graded category $\AAA$ can be viewed as a group-graded algebra with several objects. It has an underlying grading category $\uuu$, object sets $\AAA_U$ for $U \in \uuu$, and morphism modules $\AAA_u(A,A')$ for $u: U \lra U'$ in $\uuu$, $A \in \AAA_U$ and $A' \in \AAA_{U'}$. Thus, the grading category $\uuu$ can be seen as prescribing a certain shape for $\AAA$. The following are examples of naturally map-graded categories:
\begin{enumerate}
\item[(i)] For a presheaf of $k$-algebras $\aaa: \uuu \lra \Alg(k)$, there is an associated $\uuu$-graded category $\AAA$ obtained as a kind of Grothendieck construction from $\aaa$. The structured Hochschild complex $\CC(\AAA)$ controls the deformation theory of $\aaa$ as a \emph{twisted} presheaf of algebras \cite{lowenmap} and calculates the Hochschild cohomology of $\aaa$ from \cite{gerstenhaberschack} as shown in \cite{lowenvandenberghCCT}.
\item[(ii)] For two linear categories $\AAA$ and $\BBB$ and an $\AAA$-$\BBB$-bimodule $M$, the arrow category $(\BBB \rightarrow_M \AAA)$ is naturally graded over the path category of $\bullet \rightarrow \bullet$.\end{enumerate}
Our point of view is that limited functoriality is determined by grading categories in a fundamental way. In \S \ref{parmapgraded}, we first endow the category $\Cat$ of small categories with the Grothendieck pretopology of $n$-covers (for $n \in \N$) for which a collection of functors $(\varphi_i: \vvv_i \lra \uuu)_{i \in I}$ is an $n$-cover provided that it induces a jointly surjective collection of maps $(\nnn(\varphi_i): \nnn_n(\vvv_i) \lra \nnn(\uuu))_{i \in I}$ between $n$-simplices of the simplicial nerves. We show that the category $\Map$ of map-graded categories is naturally fibered over $\Cat$ (Proposition \ref{propcart}) and constitutes a stack for $n \geq 3$ (Corollary \ref{corstack}). Let $\Map_c \subseteq \Map$ denote the full subcategory of cartesian morphisms  with respect to the fibered category $\Map \lra \Cat$.  In \S \ref{parfunct}, we show that taking Hochschild complexes defines a functor (Proposition \ref{propfunct})
$$\CC: \Map_c \lra B_{\infty}: (\uuu, \AAA) \longmapsto \CC_{\uuu}(\AAA).$$
Now endow $\Map_c$ with the pretopology of $\infty$-covers for which the collection $((\varphi_i, F_i): (\vvv_i, \BBB_i) \lra (\uuu, \AAA))_{i \in I}$ is an $\infty$-cover provided that the collection $(\varphi_i: \vvv_i \lra \uuu)_{i \in I}$ is an $\infty$-cover in $\Cat$, i.e an $n$-cover for every $n \geq 0$. Our theorem \ref{mainsheaf} implies:

\begin{theorem}
The functor $\CC: \Map_c \lra B_{\infty}$ is a sheaf for the pretopology of $\infty$-covers on $\Map$.
\end{theorem}

As an application of the theorem, in \S \ref{parMV} we obtain a Mayer-Vietoris sequence of Hochschild complexes for a $\uuu$-graded category $\AAA$ and two subcategories $\varphi_i: \vvv_i \subseteq \uuu$ for $i \in \{1,2\}$ that constitute an $\infty$-cover of $\uuu$:
$$0 \lra \CC_{\uuu}(\AAA) \lra \CC_{\vvv_1}(\AAA^{\varphi_1}) \oplus \CC_{\vvv_2}(\AAA^{\varphi_2}) \lra \CC_{\vvv_1 \cap \vvv_2}(\AAA^{\varphi}) \lra 0.$$
Here $(\vvv_i, \AAA^{\varphi_i}) \lra (\uuu, \AAA)$ and $(\vvv_1 \cap \vvv_2, \AAA^{\varphi}) \lra (\uuu, \AAA)$ are chosen to be cartesian.

In \S \ref{parhochsup}, we start from a single subcategory $\varphi: \vvv \subseteq \uuu$, and a cartesian functor $(\vvv, \BBB) \lra (\uuu, \AAA)$.  This gives rise to a surjective morphism between Hochschild complexes $\CC_{\uuu}(\AAA) \lra \CC_{\vvv}(\BBB)$ of which we investigate the kernel $\CC_{\uuu \setminus \vvv}(\AAA)$. The results we obtain depend upon the assumption that the $\uuu$-morphisms outside of $\vvv$ constitute an ideal $\zzz$ in $\uuu$. In this case we show in Proposition \ref{propidsubcat} that 
$$\CC_{\uuu \setminus \vvv}(\AAA) \cong \CC_{\uuu}(\AAA, (1_{\AAA})_{\zzz})$$
where $(1_{\AAA})_{\zzz}$ is the natural restriction of $1_{\AAA}$ to an $\AAA$-bimodule supported on $\zzz$ (i.e. with zero values outside of $\zzz$).
An example where our setup applies is the situation where $\uuu$ is the category associated to a collection of open subsets of a topological space $X$ ordered by inclusion, $\vvv$ is the full subcategory of subsets $U \subseteq V$ for a fixed subset $V$, and $\zzz$ contains the inclusions $U' \subseteq U$ with $U \nsubseteq V$.  

In \S \ref{pararrow}, we revisit the arrow category construction from \cite{kellerdih} in the map-graded context. For a $(\uuu, \AAA)$-$(\vvv, \BBB)$-bimodule $(S, M)$, we take the natural inclusion $\vvv \coprod \uuu \lra (\vvv \rightarrow_S \uuu)$ and corresponding cartesian functor 
$$(\vvv \coprod \uuu, \BBB \coprod \AAA) \lra (\vvv \rightarrow_S \uuu, \BBB \rightarrow_M \AAA)$$
as starting point for obtaining map-graded analogues of some of the main results from \cite{kellerdih}. Sections \S \ref{parconn} and \S \ref{parderived} are entirely modelled upon the treatment in \cite{kellerdih}, and mainly formulate results from \cite{kellerdih} in the map-graded context, making use of the natural Hom and tensor functors from \S \ref{parbimod}. Further, in \S \ref{pararrowthin}, we give an intrinsic characterization of arrow categories based upon the \emph{thin ideals} introduced in \S \ref{parthin}. 

In \S \ref{pargroth}, we present a unified framework for constructing map-graded categories and deconstructing their Hochschild complexes. Our main observation is that both examples (i) and (ii) that we gave of map-graded categories can be viewed as special cases of a generalized Grothendieck construction for map-graded categories. The classical Grothendieck construction from \cite{SGA1} takes a pseudofunctor $\uuu \lra \Cat$ as input and turns it into a category fibered over $\uuu$. The construction from \cite{lowenmap} of which we gave an example in (i) is clearly a $k$-linearized version of this construction, using the category $\Cat(k)$ of $k$-linear categories instead of $\Cat$. If we relax $\Cat(k)$ to the bicategory $\underline{\Cat}(k)$ of $k$-linear categories and bimodules, we can in fact describe \emph{any} $\uuu$-graded category as a kind of Grothendieck construction of a naturally associated pseudofunctor 
$$\AAA: \uuu \lra \underline{\Cat}(k): U \longmapsto \AAA_U.$$
In the paper, we go yet another step further and start from a pseudofunctor
$$(\uuu, \AAA): \ccc \lra \underline{\Map}: C \longmapsto (\uuu_C, \AAA_C)$$
where $\ccc$ is a small category and $\underline{\Map}$ is the bicategory of map-graded categories and bimodules described in \S \ref{parbimod}. Allowing arbitrary bimodules rather than functors between map-graded categories allows us to capture the arrow category with respect to a bimodule from (ii). In general, we can now deconstruct the Hochschild complex of the Grothendieck construction $(\tilde{\uuu}, \tilde{\AAA})$ of $(\uuu, \AAA)$ based upon the internal structure of $\ccc$. Here, the strategy is to cover $(\tilde{\uuu}, \tilde{\AAA})$ by other Grothendieck constructions, based upon base change for pseudofunctors from \S \ref{parbase}. For instance, in Proposition \ref{propcstar}, we prove the following:

\begin{proposition}\label{propintro}
Suppose $\ccc$ has finite products. Let $\ccc^{\ast}$ be the category $\ccc$ with terminal object $\ast$ adjoined. Put $(\tilde{\uuu}|_{\ast}, \tilde{\AAA}|_{\ast}) = (\tilde{\uuu}, \tilde{\AAA})$ and let $(\tilde{\uuu}|_{C}, \tilde{\AAA}|_{C})$ be the Grothendieck construction of the restriction of $(\uuu, \AAA)$ to $\ccc/C$.
Let $(C_i)_{i \in I}$ be a collection of objects in $\ccc$ such that for every $C \in \ccc$ there exists a map $C \lra C_i$ for some $i$. There is a natural functor
$$\CC: \ccc^{\ast} \lra B_{\infty}: C \longmapsto \CC_{\tilde{\uuu}|_C}(\tilde{\AAA}|_C)$$
which satisfies the sheaf property with respect to the collection of maps $(C_i \lra \ast)_{i \in I}$ in $\ccc^{\ast}$.
\end{proposition}

Let us now restrict our attention to the case where $\ccc$ is the category associated to a poset. Then it is not hard to see that $\ccc$ can be covered by path categories of:
$$\bullet \rightarrow \bullet \rightarrow \dots \rightarrow \dots \rightarrow \bullet \rightarrow \bullet$$
In \S \ref{pargenarrow}, we consider the Grothendieck constructions of such categories as ``generalized arrow categories'', and deconstruct them using iterated arrow category constructions.
As such, we see in \S \ref{parcoverarrow} how the sheaf property for Hochschild complexes on the one hand, and the arrow category construction on the other hand, can be seen as complementary tools for deconstructing Hochschild cohomology. 

In the final section \S \ref{parcomp}, we consider a pseudofunctor
$(\uuu, \AAA): \ccc \lra \Map_c$ with some extension $(\uuu^{\star}, \ccc^{\star}): \ccc^{\ast} \lra \Map_c$, which we compare to the natural pseudofunctor of Grothendieck constructions
$$(\uuu^{\ast}, \AAA^{\ast}): \ccc^{\ast} \lra \Map_c: C \lra (\tilde{\uuu}|_{C}, \tilde{\AAA}|_{C}).$$
Our main Theorem \ref{thmmaincomp} encompasses the following comparison result:

\begin{theorem}\label{thmintro}
There is a morphism of pseudofunctors landing in the homotopy category of $B_{\infty}$-algebras
$$\CC_{\uuu^{\ast}}(\AAA^{\ast}) \lra \CC_{\uuu^{\star}}(\AAA^{\star})$$
in which every component morphism is a quasi-isomorphism except possibly the top one.
The situation where the top morphism is a quasi-isomorphism too can be characterized by the usual bimodule criterion from \cite{kellerdih}.
\end{theorem}

The proof of the theorem is heavily based upon Keller's arrow category argument in the case of a fully faithful functor $\BBB \lra \AAA$, which is in fact a special case of our theorem.
A combination of Proposition \ref{propintro} and Theorem \ref{thmintro} allows us to pull a deconstruction of Hochschild complexes of Grothendieck constructions back to a deconstruction of the Hochschild complexes of the categories $(\uuu^{\star}_{\ast}, \AAA^{\star}_{\ast})$ and $(\uuu_C, \AAA_C)$ in which we are primarily interested - at least in the homotopy category of $B_{\infty}$-algebras. In particular, exact sequences of Mayer-Vietoris type now naturally give rise to Mayer-Vietoris exact triangles, inducing the desired long exact cohomology sequences.
As an application, we recover the Mayer-Vietoris triangles for ringed spaces from \cite[\S 7.9]{lowenvandenberghhoch}.

\vspace{0,5cm}

\noindent \emph{Acknowledgement.} The author is very grateful to Michel Van den Bergh for the original idea of introducing Hochschild complexes of map-graded categories, as a tool for obtaining a Hochschild cohomology local-to-global spectral sequence for arbitrary ringed spaces, dating back to one of her research stays at the Mittag-Leffler institute in 2004. A write up of the envisaged approach, which involves the theory of hypercoverings, remains a joint work in progress \cite{lowenvandenberghlocglob}.

\section{Map-graded categories} \label{parmapgraded}

Let $k$ be a commutative ground ring.
In this section we introduce the category $\Map$ of small $k$-linear map-graded categories and functors, which is naturally fibered over the category $\Cat$ of small categories and functors. We also introduce the intermediate category $\Mas$ of map-graded sets. A map-graded set $(\uuu, \AAA)$ consists of a small category $\uuu$ and for every object $U \in \uuu$, a set $\AAA_U$. A map-graded category $(\uuu, \AAA)$ is a map-graded set with additionally, for every morphism $u: U \lra U'$ in $\uuu$, $A \in \AAA_U$ and $A' \in \AAA_{U'}$, a $k$-module $\AAA_u(A, A')$. These modules are endowed with category-like composition and identity morphism \cite{lowenmap}. Every map-graded set or category $(\uuu, \AAA)$ has an associated object $(\uuu^{\sharp}, \AAA^{\sharp})$ of the same kind, with ``ungrouped object sets'', i.e. with $\Ob(\uuu^{\sharp}) = \coprod_{U \in \uuu}\AAA_U$ and $(\AAA^{\sharp})_{A} = \{A\}$. We define the \emph{nerve} $\nnn(\AAA)$ of $(\uuu, \AAA)$ to be the simplicial nerve $\nnn(\uuu^{\sharp})$. For small categories, map-graded sets, and map-graded categories, we define pretopologies of $n$-covers by declaring a collection of functors to be an $n$-cover provided that the induced collection of maps between $n$-simplices of the nerves is jointly surjective (Definition \ref{defncover}).
We prove that for $n \geq 0$, $\Mas$ is a stack over $\Cat$ (Theorem \ref{thmmasstack}) and for $n \geq 3$, $\Map$ is a stack over $\Mas$ (Theorem \ref{thmmapstack}) and over $\Cat$ (Corollary \ref{corstack}).

\subsection{Fibered categories and stacks}\label{parfibstack}

In this section we recall the basic concepts concerning fibered categories and stacks \cite{SGA1} \cite{vistoli}, and give some results we will use later on.

Let $\uuu$ be a category. A \emph{category over $\uuu$} is a functor $F: \aaa \lra \uuu$. The functor is percieved from the point of view of its fibers. That is, for every $U \in \uuu$ we consider the fiber of objects over $U$:
$$\aaa_U = \{ A \in \aaa \,\, |\,\, F(A) = U\}$$
and for $u: V \lra U$, $A \in \aaa_U$, $B \in \aaa_{V}$,  we consider the fiber of morphisms over $u$:
$$\aaa_{u}(B,A) = \{ a \in \aaa(B,A) \,\, |\,\, F(a) = u \}.$$

A morphisms $a: B \lra A$ in $\aaa$ with $F(a) = u: V \lra U$ is \emph{cartesian (with respect to $F$)} provided that for every $v: W \lra V$ in $\uuu$ and $C \in \aaa_W$, the composition map
$$a-: \aaa_v(C,B) \lra \aaa_{uv}(C,A)$$
is an isomorphism.

The category $\aaa$ is called \emph{fibered over $\uuu$} provided that for every $u: V \lra U$ in $\uuu$ and $A \in \aaa_U$, there is an object $u^{\ast}A \in \aaa_V$ and a cartesian morphism $\delta^{u, A} \in \aaa_u(u^{\ast}A, A)$. The choice, for every $u$ and $A$, of such a cartesian morphism $\delta^{u,A}$ is called a \emph{choice of cartesian morphisms}. For a fibered category with a choice of cartesian morphisms, there is an associated pseudofunctor
$$\aaa: \uuu^{\op} \lra \Cat: U \longmapsto \aaa(U)$$
where $\Cat$ denotes the category of small categories. Here $\aaa(U)$ is the category with object set given by $\aaa_U$ and $\aaa(U)(B,A) = \aaa_{1_U}(B,A)$. For a map $u: V \lra U$ in $\uuu$, the corresponding functor $u^{\ast}: \aaa(U) \lra \aaa(V)$ is naturally determined by the chosen cartesian morphisms.

\begin{proposition}\label{funchoice}
Let $F: \aaa \lra \uuu$ be a fibered category over $\uuu$ with a choice of cartesian morphisms. Suppose for $v: W \lra V$, $u: V \lra U$ in $\uuu$ and $A \in \aaa_U$, and for the cartesian morphisms $\delta^{u,A}: u^{\ast}A \lra A$, $\delta^{v, u^{\ast}A}: v^{\ast} u^{\ast} A \lra u^{\ast} A$ and $\delta^{uv, A}: (uv)^{\ast} A \lra A$, we have $v^{\ast} u^{\ast} A = (uv)^{\ast} A$ and $\delta^{u,A} \delta^{v, u^{\ast} A} = \delta^{uv, A}$. Then the associated pseudofunctor $\aaa: \uuu^{\op} \lra \Cat$ is actually a functor.
\end{proposition}

Next we recall the definition of a pretopology. Let $\uuu$ be a category with pullbacks. A \emph{pretopology} on $\uuu$ consists of the notion of a \emph{covering collection} of maps $(U_i \lra U)_{i \in I}$, also called a \emph{cover} of the object $U$, satisfying the following axioms:
\begin{enumerate}
\item The collection consisting of $1_{U}: U \lra U$ is a cover of $U$.
\item If $(U_i \lra U)_{i \in I}$ is a cover of $U$ and $u: V \lra U$ an arbitrary map, the the collection of pullbacks $(V \times_{U} U_i \lra V)_{i \in I}$ is a cover of $V$.
\item If $(U_i \lra U)_{i \in I}$ is a cover of $U$ and for every $i \in I$, $(U_{ij} \lra U_i)_{j \in J_i}$ is a cover of $U_i$, then the collection of compositions $(U_{ij} \lra U_i \lra U)_{i\in I, j \in J_i}$ is a cover of $U$.
\end{enumerate}

Consider a fibered category $\aaa$ over $\uuu$ with a choice of cartesian morphisms and associated pseudofunctor. Let there be given a pretopology on $\uuu$ and let $S = (U_i \lra U)_{i \in I}$ be a cover of $U \in \uuu$. The \emph{descent category} $\Des(S, \aaa)$ is defined in the following way. An object, called a \emph{descent datum}, consists of a collection $(A_i)_{i \in I}$ of objects with $A_i \in \aaa(U_i)$, together with for every $i, j \in I$ an isomorphism $\alpha_1^{\ast}(A_i) \cong \alpha_2^{\ast}(A_j)$ in $\aaa(U_{ij})$ for the pullback
$$\xymatrix{ {U_i} \ar[r] & {U} \\ {U_{ij}} \ar[u]^{\alpha_1} \ar[r]_{\alpha_2} & {U_j.} \ar[u] }$$
These isomorphisms have to satisfy the natural compatibility requirement on triple pullbacks.
A morphism between descent data $(A_i)_i \lra (B_i)_i$ consists of compatible morphisms $A_i \lra B_i$ in $\aaa(U_i)$.
The fibered category $\aaa$ is called a \emph{stack} (resp. a \emph{prestack}) provided that the natural comparison functor
$$\aaa(U) \lra \Des(S, \aaa): A \longmapsto (u_i^{\ast}A)_i$$
is an equivalence of categories (resp. fully faithful) for every cover $S = (u_i: U_i \lra U)_{i \in I}$.

Next we collect some useful facts concerning composable functors  $G: \aaa \lra \xxx$ and $F: \xxx \lra \uuu$. We suppose all three categories have pullbacks and the functors $F$ and $G$ preserve pullbacks.

\begin{proposition}\label{propfibtrans}
\begin{enumerate}
\item If a morphism $a: A \lra A'$ in $\aaa$ is cartesian with respect to $G$ and $G(a)$ is cartesian with respect to $F$, then $a$ is cartesian with respect to $FG$.
\item If both $F$ and $G$ are fibered, then so is $FG$.
\end{enumerate}
\end{proposition}

\begin{proposition}
Suppose $\uuu$ is endowed with a pretopology. There is a pretopology on $\xxx$ for which $(x_i: X_i \lra X)_i$ is a cover of $X$ if and only if $(F(x_i): F(X_i) \lra F(X))_i$ is a cover of $F(X)$.
\end{proposition}

\begin{proposition}
Consider morphisms in $\xxx$:
$$\xymatrix{ {X_2} \ar[r]_{x_2} & {X_1} \ar[r]_{x_1} & {X}}$$
\begin{enumerate}
\item If $x_2$ and $x_1$ are cartesian, then so is $x_1x_2$.
\item If $x_1$ and $x_1 x_2$ are cartesian, then so is $x_2$.
\end{enumerate}
\end{proposition}

\begin{proposition} Consider a commutative square in $\xxx$:
$$\xymatrix{ {X_1} \ar[r]^{x_1} & {X} \\ {X_{12}} \ar[u]^{y_1} \ar[r]_{y_2} & {X_2} \ar[u]_{x_2} }$$
\begin{enumerate}
\item If the square is a pullback and $x_2$ is cartesian, then so is $y_1$.
\item If the image of the square under $F$ is a pullback, and the morphisms $x_1$, $x_2$, $y_1$, $y_2$ are cartesian, then the square is a pullback.
\end{enumerate}
\end{proposition}

\begin{proposition}
There is a pretopology on $\xxx$ for which $(x_i: X_i \lra X)_i$ is a cover of $X$ if and only if the morphisms $x_i$ are cartesian with repect to $F$.
\end{proposition}

\begin{proposition}\label{proppretopx}
Suppose $\uuu$ is endowed with a pretopology. There is a pretopology on $\xxx$ for which $(x_i: X_i \lra X)_i$ is a cover of $X$ if and only if the following two conditions hold:
\begin{enumerate} 
\item $(F(x_i): F(X_i) \lra F(X))_i$ is a cover of $F(X)$;
\item every $x_i$ is cartesian with respect to $F$.
\end{enumerate}
\end{proposition}

Now suppose both $F$ and $G$ are fibered, and consider a choice of cartesian morphisms for both functors. For a morphism $u: V \lra U$ in $\uuu$ and $A \in \aaa_U$ with respect to $FG$, we obtain the cartesian morphism $x = \delta^{u, G(A)}: u^{\ast}G(A) \lra G(A)$ in $\xxx$. Next we obtain the cartesian morphism $\delta^{x, A}: x^{\ast}A \lra A$. We make a choice of cartesian morphisms for $FG$ by putting $u^{\ast}A = x^{\ast}A$ and $\delta^{u,A} = \delta^{x, A}$.
Let $\uuu$ be endowed with a pretopology and endow $\xxx$ with the pretopology described in Proposition \ref{proppretopx}.
Let $\xxx^F$ be the pseudofunctor associated to $F: \xxx \lra \uuu$, $\aaa^{G}$ the one associated to $G: \aaa \lra \xxx$ and $\aaa^{FG}$ the one associated to $FG: \aaa \lra \uuu$.

\begin{proposition}\label{stacktrans}
If $\xxx^F$ and $\aaa^G$ are stacks (resp. prestacks), then so is $\aaa^{FG}$.
\end{proposition}

\subsection{Map-graded categories and functors}\label{mapgraded}

Let $\uuu$ be a base category.
A \emph{$\uuu$-graded set} $\XXX$ consists of the datum, for every $U \in \uuu$, of a set $\XXX_U$.
A \emph{($k$-linear) $\uuu$-graded category} (see \cite{lowenmap}) $\AAA$ consists of the following data:

\begin{itemize}
\item For every $U \in \uuu$, a set $\AAA_U$ of \emph{objects over $U$}.
\item For every $u: V \lra U$ in $\uuu$, $B \in \AAA_V$, $A \in \AAA_U$, a $k$-module $\AAA_u(B,A)$ of \emph{morphisms over $u$}.
\item For a further $v: W \lra V$ in $\uuu$ and $C \in \AAA_W$, a \emph{composition} map
$$\AAA_u(B,A) \otimes_k \AAA_v(C,B) \lra \AAA_{uv}(C,A).$$
\item For $A \in \AAA_U$, an \emph{identity element}
$$1_A \in \AAA_{1_U}(A,A).$$
\end{itemize}

These data should satisfy the obvious category-type axioms, i.e. the composition has to be associative and the identity elements have to act identically under composition.
There is an associated $k$-linear category $\tilde{\AAA}$ with $\Ob(\tilde{\AAA}) = \coprod_{U \in \uuu} \AAA_U$ and $\AAA(A_V, A_U) = \oplus_{u \in \uuu(V,U)} \AAA_u(A_V, A_U)$.
Clearly, a $\uuu$-graded category has an underlying $\uuu$-graded set given by the object sets.

\begin{remark}\label{remlin}\label{remgraded}
The definition of a $k$-linear $\uuu$-graded category is a $k$-linearized version of the notion of a category over $\uuu$ from \S \ref{parfibstack}. Indeed, if we drop $k$-linearity from the definition, a $\uuu$-graded category $\aaa$ now has a natural associated category $\tilde{\aaa}$ with $\tilde{\aaa}(A_V, A_U) = \coprod_{u \in \uuu(V,U)} \AAA_u(A_V, A_U)$. Hence, there is a natural functor $\tilde{\aaa} \lra \uuu$, and the datum of this functor is equivalent to the datum of $\aaa$.
To avoid confusion, for us a \emph{$\uuu$-graded category} will impicitely mean a $k$-linear $\uuu$-graded category over some fixed $k$, whereas we will refer to the non-linear variant explicitely as a \emph{non-linear $\uuu$-graded category}.
\end{remark}

\begin{example}\label{exmonoid}
Suppose $\uuu$ has a single object $\ast$ and $\AAA_{\ast} = \{\ast\}$. Then $G = \uuu(\ast, \ast)$ is a monoid, and $\AAA$ corresponds to a $G$-graded algebra $A$ with $A_g = \AAA_g(\ast, \ast)$. Thus, in the spirit of \cite{mitchell}, we can view map-graded categories as monoid-graded algebras with several objects.
\end{example}

\begin{example}\label{exstandgr}
A $k$-linear category $\AAA$ can be made into a graded category in several natural ways:
\begin{enumerate} 
\item Let $\uuu_{\AAA}$ be the category with $\Ob(\uuu_{\AAA}) = \Ob(\AAA)$ and $\uuu_{\AAA}(B,A) = \{\ast \}$ for all $B, A \in \AAA$. Then we can make $\AAA$ into a $\uuu_{\AAA}$-graded category $\AAA_{st}$ with $(\AAA_{st})_A = \{ A \}$ and $(\AAA_{st})_{\ast}(B,A) = \AAA(B,A)$. We refer to this grading as the \emph{standard grading} on $\AAA$.
\item Let $e$ be the category with one object $\ast$ and one morphism $1_{\ast}$. Then we can make $\AAA$ into an $e$-graded category $\AAA_{tr}$ with $(\AAA_{tr})_{\ast} = \Ob(\AAA)$ and $(\AAA_{tr})_{1_{\ast}}(B,A) = \AAA(B,A)$. We refer to this grading as the \emph{trivial grading} on $\AAA$.
\item Clearly, every partition of $\Ob(\AAA)$ will give rise to a graded incarnation of $\AAA$ ``in between'' the two extremes described in (1) and (2).
\end{enumerate}
\end{example}

\begin{example}\label{exfreegr}
Let $\uuu$ be a category. The \emph{free $\uuu$-graded category} $k\uuu$ is defined by putting $k\uuu_U = \{ U \}$ and $k\uuu_u(V,U) = k$ for all $u: V \lra U \in \uuu$. Compositions are defined by means of the multiplication of $k$, and identity elements are provided by the unit of $k$. Instead of $k$, we can actually use an arbitrary $k$-algebra $A$ and perform a similar construction.
\end{example}

\begin{example}\label{exsharp}
Let $\XXX$ be a $\uuu$-graded set. There is an associated small category $\uuu^{\sharp}$ and a $\uuu^{\sharp}$-graded set $\XXX^{\sharp}$ involving only singleton sets above the objects of $\uuu^{\sharp}$. Precisely, we put $\Ob(\uuu^{\sharp}) = \coprod_{U \in \uuu} \XXX_U$ and $\uuu^{\sharp}(X_U, X_{U'}) = \uuu(U, U')$, and $\XXX^{\sharp}_{X_U} = \{ X_U \}$.  If $\AAA$ is a $\uuu$-graded category, we further obtain the $\uuu^{\sharp}$-graded category $\AAA^{\sharp}$ with in addition $\AAA^{\sharp}_u(A_U, A_{U'}) = \AAA_u(A_U, A_{U'})$.
\end{example}

\begin{example}\label{extrst}
With the notations of Example \ref{exstandgr}, for a $k$-linear category, we have $(\uuu_{\AAA}, \AAA_{st}) = (e, \AAA_{tr})^{\sharp}$.
\end{example}

Let $\varphi: \vvv \lra \uuu$ be a functor, $\XXX$ a $\uuu$-graded set and $\YYY$ a $\vvv$-graded set. A \emph{$\varphi$-graded map} $F$ consists of maps $F_V: \YYY_V \lra \XXX_{\varphi(V)}$ for every $V \in \vvv$.

Let $\AAA$ be a $\uuu$-graded category and $\BBB$ a $\vvv$-graded category.
A \emph{$\varphi$-graded functor} $F$ consists of the following data:

\begin{itemize}
\item For every $V \in \vvv$, a map $F_V: \BBB_V \lra \AAA_{\varphi(V)}$.
\item For every $v: V \lra V'$ in $\vvv$, $B \in \BBB_V$ and $B' \in \BBB_{V'}$, a $k$-linear map $$\BBB_v(B,B') \lra \AAA_{\varphi(v)}(F(B), F(B')).$$
\end{itemize}

These data should satisfy the obvious functoriality-type axioms, i.e. $F$ respects compositions and identity elements.
Clearly, a $\varphi$-graded functor has an underlying $\varphi$-graded map between object sets.
Let an underlying ground ring $k$ be fixed. Map-graded categories (resp. sets) and functors (resp. maps) constitute a category $\Map$ (resp. $\Mas$) in the following way. An object of $\Map$ (resp. $\Mas$) is given by a small category $\uuu$ and a $\uuu$-graded category (resp. set) $\AAA$. A morphism $(\uuu, \AAA) \longmapsto (\vvv, \BBB)$ is given by a functor $\varphi: \uuu \lra \vvv$ and a $\varphi$-graded functor (resp. map) $F: \AAA \lra \BBB$. 

\begin{example}
Let $\AAA$ be a linear category. Let $(\AAA_{st}, \uuu_{\AAA})$ be the graded category resulting from the standard grading on $\AAA$ as in Example \ref{exstandgr} (1) and let $(\AAA_{tr}, e)$ be the graded category resulting from the trivial grading on $\AAA$ as in Example \ref{exstandgr} (2). Let $\varphi: \uuu_{\AAA} \lra e$ be the unique functor. The map $(\AAA_{st})_A = \{ A \} \lra \Ob(\AAA) = (\AAA_{tr})_{\ast}: A \longmapsto A$ and the maps $1: (\AAA_{st})_{\ast}(A,A') = \AAA(A,A') \lra \AAA(A,A') = (\AAA_{tr})_{\ast}(A,A')$ define a $\varphi$-graded functor.
\end{example}

\begin{example}\label{exadj}
Let $\AAA$ be a $\uuu$-graded category and $\BBB$ a linear category with associated $e$-graded category $\BBB_{tr}$ with trivial grading as in Example \ref{exstandgr} (2) and let $\varphi: \uuu \lra e$ be the unique functor. There is a one-one correspondence between $\varphi$-graded functors $\AAA \lra \BBB_{tr}$ and $k$-linear functors $\tilde{\AAA} \lra \BBB$.
\end{example}

\begin{example}\label{exfunsharp}
With the notations of Example \ref{exsharp}, let $\AAA$ be a $\uuu$-graded set resp. category and let $\varphi_{\AAA}: \uuu^{\sharp} \lra \uuu$ be the forgetful functor. There is a $\varphi_{\AAA}$-graded map resp. functor $F_{\AAA}: \AAA^{\sharp} \lra \AAA$ given by $F(A_U) = A_U \in \AAA_U$ (and $1: \AAA^{\sharp}_u(A,A') \lra \AAA_u(A,A')$ in the category case).
A graded map resp. functor $(\varphi, F): (\vvv, \BBB) \lra (\uuu, \AAA)$ gives rise to a commutative square
$$\xymatrix{ {(\vvv, \BBB)} \ar[r]^{(\varphi, F)} & {(\uuu, \AAA)} \\ {(\vvv^{\sharp}, \BBB^{\sharp})} \ar[u]^{(\varphi_{\BBB}, F_{\BBB})} \ar[r]_{(\varphi^{\sharp}, F^{\sharp})} & {(\uuu^{\sharp}, \AAA^{\sharp}).} \ar[u]_{(\varphi_{\AAA}, F_{\AAA})} }$$
We thus obtain a natural functor 
$$(-)^{\sharp}: \Map \lra \Map: (\uuu, \AAA) \longmapsto (\uuu^{\sharp}, \AAA^{\sharp})$$
and similarly in the case of $\Mas$.
\end{example}

\begin{example}
Denote the category of small $k$-linear categories by $\Cat(k)$. There are natural functors $\Map \lra \Cat(k): \AAA \longmapsto \tilde{\AAA}$ and $\Cat(k) \lra \Map: \AAA \longmapsto \AAA_{tr}$. By Example \ref{exadj}, the first functor is left adjoint to the second functor.
\end{example}

\subsection{Two fibered categories of map-graded categories}\label{parfibmap}

Let $k$ be a fixed commutative ground ring. We consider the categories $\Map$ and $\Mas$ as defined in \S \ref{mapgraded}, as well as the category $\Cat$ of small categories.
There are natural forgetful functors
$$\Psi_1: \Map \lra \Mas: (\uuu, \AAA) \longmapsto (\uuu, \XXX_{\AAA})$$
where $\XXX_{\AAA}$ is the underlying $\uuu$-graded set of objects of $\AAA$ and
$$\Psi_0: \Mas \lra \Cat: (\uuu, \AAA) \longmapsto \uuu.$$
Further, all three categories $\Map$, $\Mas$ and $\Cat$ have pullbacks and the two functors $\Psi_1$ and $\Psi_0$ preserve them. We first look at $\Cat$. For functors $\varphi_1: \vvv_1 \lra \uuu$ and $\varphi_2: \vvv_2 \lra \uuu$, the pullback
\begin{equation}\label{pb}
\xymatrix{{\vvv_1} \ar[r]^{\varphi_1} & {\uuu} \\ {\vvv_1 \times_{\uuu} \vvv_2} \ar[u]^{\alpha_1} \ar[r]_{\alpha_2} & {\vvv_2} \ar[u]_{\varphi_2}}
\end{equation}
is given by the category $\vvv_1 \times_{\uuu} \vvv_2$ where
$$\Ob(\vvv_1 \times_{\uuu} \vvv_2) = \Ob(\vvv_1) \times_{\Ob(\uuu)} \Ob(\vvv_2)$$ and
$$(\vvv_1 \times_{\uuu} \vvv_2)((V_1, V_2), (V_1', V_2')) = \vvv_1(V_1, V_1') \times_{\uuu(\varphi_1(V_1), \varphi_2(V_2))} \vvv_2(V_2, V_2')$$
are given by the pullbacks in the category of sets.
For a collection of functors $\varphi_i: \vvv_i \lra \uuu$, we similarly obtain a limit category $\prod_{i, \uuu} \vvv_i$. 

\begin{example}\label{exint}
Consider subcategories $\vvv_1 \subseteq \uuu$ and $\vvv_2 \subseteq \uuu$. We define the category $\vvv_1 \cap \vvv_2$ with $\Ob(\vvv_1 \cap \vvv_2) = \Ob(\vvv_1) \cap \Ob(\vvv_2)$ and $(\vvv_1 \cap \vvv_2)(V,V') = \vvv_1(V,V') \cap \vvv_2(V,V')$. The canonical inclusion functors $\vvv_1 \cap \vvv_2 \subseteq \vvv_1$ and $\vvv_2 \subseteq \vvv_1 \cap \vvv_2$ induce an isomorphism of categories $\vvv_1 \cap \vvv_2 \cong \vvv_1 \times_{\uuu} \vvv_2$.
\end{example}

Pullbacks in $\Mas$ and $\Map$ are described in a similar fashion. For instance, the pullback of graded categories
is described by
$$\xymatrix{ {(\vvv_1, \BBB_1)} \ar[r]^{(\varphi_1, F_1)} & {(\uuu, \AAA)}\\   {(\vvv_1 \times_{\uuu} \vvv_2, \BBB_1 \times_{\AAA} \BBB_2)} \ar[u]^{(\alpha_1, G_1)} \ar[r]_-{(\alpha_2, G_2)} & {(\vvv_2, \BBB_2)} \ar[u]_-{(\varphi_2, F_2)} }$$
with underlying pullback of categories described by \eqref{pb} and with
$$(\BBB_1 \times_{\AAA} \BBB_2)_{(V_1, V_2)} = (\BBB_1)_{V_1} \times_{\AAA_{\varphi_1(V_1)}} (\BBB_2)_{V_2}$$
and
$$(\BBB_1 \times_{\AAA} \BBB_2)_{(v_1, v_2)}((B_1, B_2), (B_1', B_2')) = (\BBB_1)_{v_1}(B_1, B_1') \times_{\AAA_{\varphi_1(v_1)}(F_1(B_1), F_1(B_1'))} (\BBB_2)_{v_2}(B_2, B_2').$$

\begin{example}\label{expullsharp}
The functors $(-)^{\sharp}: \Map \lra \Map$ and $(-)^{\sharp}: \Mas \lra \Mas$ from Example \ref{exfunsharp} preserves pullbacks.
\end{example}

Thus, the results of \S \ref{parfibstack} apply to the composable functors $\Psi_0$ and $\Psi_1$.
In the remainder of this section we show that both these functors are fibered.

We start with $\Psi_0$. Let $\varphi: \vvv \lra \uuu$ be a functor and $\XXX$ a $\uuu$-graded set. We define the $\vvv$-graded set $\XXX^{\varphi}$ with $$\XXX^{\varphi}_V = \XXX_{\varphi(V)}$$
and the morphism of graded sets $(\varphi, \delta^{\varphi, \XXX}): (\vvv, \XXX^{\varphi}) \lra (\uuu, \XXX)$ with $\delta^{\varphi, \XXX}_V: \XXX^{\varphi}_V = \XXX_{\varphi(V)} \lra \XXX_{\varphi(V)}$ given by the identity morphism.

\begin{proposition}
A morphism $(\varphi, F): (\vvv, \YYY) \lra (\uuu, \XXX)$ of graded sets is cartesian with respect to $\Psi_0$ if and only if for every $V \in \vvv$ the map $F_V: \YYY_V \lra \XXX_{\varphi(V)}$ is an isomorphism of sets.
\end{proposition}

\begin{proposition}\label{propfib0}
\begin{enumerate}
\item The morphisms $(\varphi, \delta^{\varphi, \XXX})$ are cartesian with respect to $\Psi_0$.
\item The category $\Mas$ is fibered over $\Cat$ through $\Psi_0$.
\end{enumerate}
\end{proposition}

Next we look at $\Psi_1$. 
Let $(\varphi, F): (\vvv, \YYY) \lra (\uuu, \XXX)$ be a morphism of graded sets, and let $\AAA$ be a $\uuu$-graded category with underlying graded set $\XXX$. We define the $\vvv$-graded category $\AAA^{\varphi, F}$ with
$$\AAA^{\varphi, F}_V = \YYY_V$$
and with, for $v: V \lra V'$ in $\vvv$, $Y \in \YYY_V$, $Y' \in \YYY_{V'}$:
$$\AAA^{\varphi, F}_v(Y, Y') = \AAA_{\varphi(v)}(F(Y), F(Y'))$$
and with, for another $v': V' \lra V''$ in $\vvv$ and $Y'' \in \YYY_{V''}$, the composition in $\AAA^{\varphi, F}$
$$\AAA_{\varphi(v')}(F(Y'), F(Y'')) \otimes \AAA_{\varphi(v)}(F(Y), F(Y')) \lra \AAA_{\varphi(v' v)}(F(Y), F(Y''))$$
defined by the composition in $\AAA$, and similarly for the identity elements in $\AAA^{\varphi, F}_{1_V}(Y,Y) = \AAA_{1_{\varphi(V)}}(F(Y), F(Y))$.
We define the morphism of graded categories
$(\varphi, \delta = \delta^{\varphi, F, \AAA}): (\vvv, \AAA^{\varphi, F}) \lra (\uuu, \AAA)$ by the morphisms $\delta_V = F_V: \YYY_V \lra \AAA_{\varphi(V)}$ and the identity morphisms
$$\AAA^{\varphi, F}_v(Y, Y') = \AAA_{\varphi(v)}(F(Y), F(Y')) \lra \AAA_{\varphi(v)}(F(Y), F(Y')).$$

\begin{proposition}\label{lemcart1}
A morphism $(\varphi, F): (\vvv, \BBB) \lra (\uuu, \AAA)$ of graded categories is cartesian with respect to $\Psi_1$ if and only if for every $v: V \lra V'$ in $\vvv$, $B \in \BBB_V$, $B' \in \BBB_{V'}$, the map $\BBB_v(B,B') \lra \AAA_{\varphi(v)}(F(B), F(B'))$ is an isomorphism. 
\end{proposition}

\begin{proposition}\label{propfib1}
\begin{enumerate}
\item The morphisms $(\varphi, \delta^{\varphi, F, \AAA})$ are cartesian with respect to $\Psi_1$.
\item The category $\Map$ is fibered over $\Mas$ through $\Psi_1$.
\end{enumerate}
\end{proposition}

By Proposition \ref{propfibtrans}, as a consequence of Propositions \ref{propfib0} and \ref{propfib1}, the category $\Map$ is fibered over $\Cat$ through $\Psi$. We end this section by describing the choice of cartesian morphisms which follows from the higher choices for $\Psi_0$ and $\Psi_1$.

Let $\AAA$ be a $\uuu$ graded category and let $\varphi: \vvv \lra \uuu$ be a functor. We define the $\vvv$-graded category $\AAA^{\varphi}$ with 
$$\AAA^{\varphi}_V = \AAA_{\varphi(V)}$$ and with, for $v: V \lra V'$ in $\vvv$, $A \in \AAA^{\varphi}_V$ and $A' \in \AAA^{\varphi}_{V'}$
$$\AAA^{\varphi}_v(A, A') = \AAA_{\varphi(v)}(A,A').$$
For another $v': V' \lra V''$ in $\vvv$ and $A'' \in \AAA^{\varphi}_{V''}$, the composition in $\AAA^{\varphi}$
$$\AAA_{\varphi(v')}(A', A'') \otimes \AAA_{\varphi(v)}(A,A') \lra \AAA_{\varphi(v'v)}(A,A'')$$
is defined by the composition in $\AAA$, and similarly for the identity elements in $\AAA_{\varphi(1_V)}(A,A) = \AAA_{1_{\varphi(V)}}(A,A)$.
We further define the morphism of graded categories $(\varphi,\delta^{\varphi, \AAA}): (\vvv, \AAA^{\varphi}) \lra (\uuu, \AAA)$
for which the maps
$$\AAA^{\varphi}_V \lra \AAA_{\varphi(V)}$$ are given by identities and the maps
$$\AAA^{\varphi}_v(A, A') \lra \AAA_{\varphi(v)}(A,A')$$
as well. 

\begin{proposition}\label{propcart}
\begin{enumerate}
\item The morphisms $(\varphi, \delta^{\varphi, \AAA})$ are cartesian with respect to $\Psi_1$.
\item The category $\Map$ is fibered over $\Cat$ through $\Psi$.
\end{enumerate}
\end{proposition}

\begin{proposition}\label{lemcart}
A morphism $(\varphi, F): (\vvv, \BBB) \lra (\uuu, \AAA)$ of graded categories is cartesian with respect to $\Psi$ if and only if the following hold:
\begin{enumerate}
\item For $V \in \vvv$, the map $F_V: \BBB_V \lra \AAA_{\varphi(V)}$ is an isomorphism.
\item For $v: V \lra V'$ in $\vvv$, $B \in \BBB_V$, $B' \in \BBB_{V'}$, the map $\BBB_v(B,B') \lra \AAA_{\varphi(v)}(F(B), F(B'))$ is an isomorphism. 
\end{enumerate}
\end{proposition}

To distinguish between the notions arrising in Propositions \ref{lemcart1} and \ref{lemcart}, we make the following

\begin{definition}\label{defcart}
Consider a morphism $(\varphi, F): (\vvv, \BBB) \lra (\uuu, \AAA)$ between graded categories.
\begin{enumerate}
\item $(\varphi, F)$ is called \emph{cartesian} if it is cartesian with respect to $\Psi$.
\item $(\varphi, F)$ is called \emph{subcartesian} if it is cartesian with respect to $\Psi_1$.
\end{enumerate}
\end{definition}

\begin{remark}
\begin{enumerate}
\item If $(\varphi, F)$ is cartesian, then it is subcartesian.
\item $(\varphi,F)$ is subcartesian if and only if $(\varphi^{\sharp}, F^{\sharp})$ as defined in Example \ref{exfunsharp} is cartesian.
\end{enumerate}
\end{remark}

\begin{example}\label{exfff}
Let $\BBB \lra \AAA$ be a fully faithful $k$-linear functor between $k$-linear categories. With the notations from Example \ref{exstandgr} the natural graded functor $(e, \BBB_{tr}) \lra (e, \AAA_{tr})$ is subcartesian and the functor
$(\uuu_{\BBB}, \BBB_{st}) \lra (\uuu_{\AAA}, \AAA_{st})$ is cartesian (see also Example \ref{extrst}).
\end{example}

\subsection{Sites of categories, graded sets and graded categories}\label{parsite}

In this section, we will introduce pretopologies on the categories $\Cat$, $\Mas$ and $\Map$.
We start by introducing appropriate nerves. For a small category $\uuu$, we denote by $\nnn(\uuu)$ the simplicial nerve of $\uuu$. Concretely, for $n \geq 1$, $\nnn_n(\uuu)$ consists of the $n$-simplices
$$\xymatrix{ {U_0} \ar[r]_{u_0} & {U_1} \ar[r]_{u_1} & {\dots} \ar[r]_{u_{n-1}} & {U_n}}$$
with $U_i \in \Ob(\uuu)$ and $u_i \in \Mor(\uuu)$, and $\nnn_0(\uuu) = \Ob(\uuu)$. Note that $\nnn_1(\uuu) = \Mor(\uuu)$.
For a functor $\varphi: \vvv \lra \uuu$, we obtain an induced map $\nnn(\varphi): \nnn(\vvv) \lra \nnn(\uuu)$ between the  nerves, with components $\nnn_n(\varphi): \nnn_n(\vvv) \lra \nnn_n(\uuu)$.

For a $\uuu$-graded category or set $\AAA$, we consider the associated category $\uuu^{\sharp}$ from Example \ref{exsharp} with $\Ob(\uuu^{\sharp}) = \coprod_{U \in \uuu} \AAA_U$ and $\uuu^{\sharp}(A_U, A'_{U'}) = \uuu(U, U')$. We define the \emph{nerve} of $\AAA$ to be $\nnn(\AAA) = \nnn(\uuu, \AAA) = \nnn(\uuu^{\sharp})$. 
By Example \ref{exfunsharp}, a graded map or functor $(\varphi, F): (\vvv, \BBB) \lra (\uuu, \AAA)$ induces an associated map $\nnn(F): \nnn(\BBB) \lra \nnn(\AAA)$ given by $\nnn(\varphi^{\sharp}): \nnn(\vvv^{\sharp}) \lra \nnn(\uuu^{\sharp})$.

\begin{definition}\label{defncover}
Let $n \in \N \cup \{\infty\}$.  A collection of functors (resp. graded maps, resp. graded functors) $(F_i: \BBB_i \lra \AAA)_{i \in I}$ is an \emph{$n$-cover} of $\AAA$ in $\Cat$ (resp. in $\Mas$, resp. in $\Map$) if for all $k \in \N$ with $k \leq n$, the collection of maps $(\nnn_k(F_i))_{i \in I}$ is jointly surjective.
\end{definition}

Thus, a collection of graded maps (resp. graded functors) $(\varphi_i, F_i): (\vvv_i, \BBB_i) \lra (\uuu, \AAA)$ is an $n$-cover of $\AAA$ in $\Mas$ (resp. in $\Map$) if and only if the collection of functors $(\varphi^{\sharp}: \vvv_i^{\sharp} \lra \uuu^{\sharp})_{i \in I}$ is an $n$-cover in $\Cat$.

\begin{lemma}\label{lemcover}
Let $n \in \N$. 
Consider a collection of functors (resp. graded maps, resp. graded functors) $\sss = (F_i: \BBB_i \lra \AAA)_{i \in I}$. If the collection of maps $(\nnn_n(F_i))_{i \in I}$ is jointly surjective, then $\sss$ is an $n$-cover in $\Cat$ (resp. in $\Mas$, resp. in $\Map$).
\end{lemma}

\begin{proof}
Let $k < n$. Let us look at the graded cases, and let $\AAA$ be graded over $\uuu$ and $\BBB$ over $\vvv_i$. Every $k$-simplex $u$ for $\uuu^{\sharp}$ gives rise to an $n$-simplex $u'$ by adding $n - k$ identity maps $1_{U_0}: U_0 \lra U_0$. If $u' = \varphi_i(v')$ for some $i$ and $v' \in \vvv_i^{\sharp}$, then removing the first $n -k$ maps from $v'$ yields a $k$-simplex $v$ with $\varphi(v) = u$.
\end{proof}

\begin{lemma}\label{lemnervepb}
Let $n \in \N$. Consider a pullback in $\Cat$, $\Mas$ or $\Map$:
$$\xymatrix{ {\BBB_1} \ar[r]^{F_1} & {\AAA} \\ {\BBB_1 \times_{\AAA} \BBB_2} \ar[u]^{G_1} \ar[r]_-{G_2} & {\BBB_2} \ar[u]_{F_2} }$$
We have $\nnn_n(\BBB_1 \times_{\AAA} \BBB_2) \cong \nnn_n(\BBB_1) \times_{\nnn(\AAA)} \nnn(\BBB_2).$
\end{lemma}


\begin{proposition}\label{pretop}
Let $n \in \N \cup \{ \infty \}$ be fixed. The $n$-covers from Definition \ref{defncover} define pretopologies on $\Cat$, $\Mas$ and $\Map$. 
\end{proposition}

\begin{proof}
The identity and glueing properties are immediate. Concerning the pullback property, we note that in $\Set$, the pullback of a jointly surjective collection of maps is jointly surjective. Hence, the result follows from Example \ref{expullsharp}.
\end{proof}

\begin{example}\label{excover}
\begin{enumerate}
\item A functor $\varphi: \vvv \lra \uuu$ which is full and surjective on objects constitutes an $\infty$-cover of $\uuu$ in $\Cat$.
\item Consider a collection of objects $(U_i)_{i \in I}$ in the category $\uuu$ and the accociated collection of functors $(\varphi_i: \uuu/U_i \lra \uuu)_{i \in I}$. The collection $(\varphi_i)_i$ is a $1$-cover in $\Cat$ if for every object $U \in \uuu$, there exists a morphism $U \lra U_i$ for some $i$. In this case, the collection $(\varphi_i)_i$ is automatically an $\infty$-cover of $\uuu$. 
\end{enumerate}
\end{example}

\subsection{The stack of map-graded sets}\label{parmasstack}

Consider the functor $\Psi_0: \Mas \lra \Cat$ and consider the pretopology $\ttt_n$ of $n$-covers from Proposition \ref{pretop} on $\Cat$ for some $n \in \N \cup \{ \infty \}$.

\begin{proposition}
The pseudofunctor $\Mas$ associated to $\Psi_0$ is a functor.
\end{proposition}

Let $\sss = (\varphi_i: \vvv_i \lra \uuu)_i$ be an $n$-cover of $\uuu$ in $\Cat$. We introduce some notation to be able to describe the descent category $\Des(\sss, \Mas)$. 
We put $\vvv_{ij} = \vvv_{i} \times_{\uuu} \vvv_{j}$ and $\vvv_{ijk} = \vvv_i \times_{\uuu} \vvv_j \times_{\uuu} \vvv_k$. 
We denote the canonical maps from a $k$-fold pullback by $\alpha_1, \dots, \alpha_k$ to avoid confusion when some of the indices $i, j, \dots$ coincide. For example, we have $\alpha_1: \vvv_{ij} \lra \vvv_i$ and $\alpha_2: \vvv_{ij} \lra \vvv_j$.
For $\BBB_i \in \Mas(\vvv_i)$, we use notations like ${\BBB_i}|^1_{ij} = \BBB_i^{\alpha_1} \in \Map(\vvv_{ij})$ and ${\BBB_j}|^2_{ijk} = \BBB_j^{\alpha_2} \in \Map(\vvv_{ijk})$.
The descent category $\Des(\sss, \Mas)$ has objects given by $(\BBB_{i})_{i}$ with $\BBB_{i} \in \Mas(\vvv_{i})$ along with compatible isomorphisms 
$$\rho_{ij}: {\BBB_i}|^1_{ij} \cong {\BBB_j}|^2_{ij}.$$
Precisely, we require that the cocycle condition
\begin{equation} \label{cocyc}
(\rho_{jk}|^{23}_{ijk})(\rho_{ij}|^{12}_{ijk}) = \rho_{ik}|^{13}_{ijk}
\end{equation}
on triple pullbacks holds. Consider $(V_1, V_2, V_3) \in \vvv_{ijk}$. We have $(\rho_{jk}|^{23}_{ijk})_{(V_1, V_2, V_3)} = (\rho_{jk})_{(V_2, V_3)}: (\BBB_j)_{V_2} \lra (\BBB_k)_{V_3}$, $(\rho_{ij}|^{12}_{ijk})_{(V_1, V_2, V_3)} = (\rho_{ij})_{(V_1, V_2)}: (\BBB_i)_{V_1} \lra (\BBB_j)_{V_2}$, $(\rho_{ik}|^{13}_{ijk})_{(V_1, V_2, V_3)} = (\rho_{ik})_{(V_1, V_3)}: (\BBB_i)_{V_1} \lra (\BBB_k)_{V_3}$. hence, from \eqref{cocyc} we obtain
\begin{equation} \label{cocyc2}
(\rho_{jk})_{(V_2, V_3)} (\rho_{ij})_{(V_1, V_2)} = (\rho_{ik})_{(V_1, V_3)}.
\end{equation}
In particular, we have
\begin{equation} \label{cocyc3}
(\rho_{ii})_{(V,V)} = 1: (\BBB_i)_V \lra (\BBB_i)_V.
\end{equation}

%

\begin{theorem}\label{thmmasstack}
The functor $\Mas$ is a stack.
\end{theorem}

\begin{proof}
We may take $n = 0$. Let $\sss = (\varphi_i: \vvv_i \lra \uuu)_i$ be a $0$-cover of $\uuu$ in $\Cat$.

First, consider $\AAA, \BBB \in \Mas(\uuu)$ and a compatible collection of morphism $F_i: \BBB^{\varphi_i} \lra \AAA^{\varphi_i}$. For the pullback
$$\xymatrix{ {\vvv_i} \ar[r]^{\varphi_i} & {\uuu} \\ {\vvv_{ij}} \ar[u]^{\alpha_1} \ar[r]_{\alpha_2} & {\vvv_j} \ar[u]_{\varphi_j}}$$
we have $(\AAA^{\varphi_i})^{\alpha_1} = (\AAA^{\varphi_j})^{\alpha_2}$ and similarly for $\BBB$, and compatibility amounts to the fact that $(F_i)^{\alpha_1} = (F_j)^{\alpha_2}$. We are to define a unique glueing $F: \BBB \lra \AAA$. This consists of maps $F_U: \BBB_U \lra \AAA_U$ for every $U \in \uuu$ such that for every $i$ and $V \in \vvv_i$ we have $(F_i)_V = (F^{\varphi_i})_V = F_{\varphi_i(V)}: \BBB_{\varphi_i(V)} \lra \AAA_{\varphi_i(V)}$. Since we have a $0$-cover, for $U \in \uuu$ there is some $i$ and some $V \in \vvv_i$ with $\varphi_i(V) = U$, so we put
$F_U = (F_i)_V$. If for another $j$ and $V' \in \vvv_j$ we also have $\varphi_j(V') = U$, then $(V, V') \in \vvv_{ij}$ and $(F_i)_V = (F_i)^{\alpha_1}_{(V,V')} = (F_j)^{\alpha_2}_{(V,V')} = (F_j)_{V'}$.

Next we consider a descent datum $(\BBB_i)_i$ with $\BBB_i \in \Mas(\vvv_i)$ with compatible isomorphisms $\rho_{ij}: {\BBB_i}|^1_{ij} \cong {\BBB_j}|^2_{ij}$. We define the $\uuu$-graded set $\BBB$ with 
$$\BBB_U = \coprod_{\varphi_i(V) = U} (\BBB_i)_{V} / \sim$$
where for $X_1 \in (\BBB_i)_{V_1}$ and $X_2 \in (\BBB_j)_{V_2}$ we have $X_1 \sim X_2$ if and only if the two elements correspond through the isomorphism $(\rho_{ij})_{(V_1, V_2)}: (\BBB_i)_{V_1} \lra (\BBB_j)_{V_2}$. 
The relation $\sim$ is obviously symmetric, it is transitive by \eqref{cocyc2} and reflexive by \eqref{cocyc3}.  For every $i$ and $V \in \vvv_i$ with $\varphi_i(V) = U$, the canonical map
$(\BBB_i)_V \lra \BBB_U$ is an isomorphism, giving rise to a compatible collection of isomorphisms $F_i: \BBB_i \lra \BBB^{\varphi_i}$ constituting an isomorphism $F: (\BBB_i)_i \lra (\BBB^{\varphi_i})_i$ in $\Des(\sss, \Mas)$.
\end{proof}

\subsection{The stack of map-graded categories}\label{parstackmap}

Consider the functor $\Psi_1: \Map \lra \Mas$ and consider the pretopology $\ttt_n$ of $n$-covers from Proposition \ref{pretop} on $\Mas$ for some $n \in \N \cup \{ \infty \}$.

\begin{proposition}
The pseudofunctor $\Map$ associated to $\Psi_1$ is a functor.
\end{proposition}

Let $\sss = ((\varphi_i,F_i): (\vvv_i,\YYY_i) \lra (\uuu, \XXX))_i$ be an $n$-cover of $\XXX$ in $\Mas$. We use notation that is similar to the notation introduced in \S \ref{parmasstack}.
Pullbacks are denoted $\YYY_{ij} = \YYY_i \times_{\XXX} \YYY_j$ and $\YYY_{ijk} = \YYY_i \times_{\XXX} \YYY_j \times_{\XXX} \YYY_k$ and for $\BBB_i \in \Map(\YYY_i)$, restrictions are denoted $\BBB_i|^1_{ij} \in \Map(\YYY_{ij})$ etc.. The descent category $\Des(\sss, \Map)$ has objects given by $(\BBB_i)_i$ with $\BBB_i \in \Map(\YYY_i)$ along with compatible isomorphisms
$$\rho_{ij}: \BBB_i|^1_{ij} \cong \BBB_j|^2_{ij}$$
for which the cocycle condition \eqref{cocyc} holds. To unravel this condition, we now consider maps $v_{\kappa}: (V_{\kappa}, Y_{\kappa}) \lra (V'_{\kappa}, Y'_{\kappa})$ in $\vvv_{\kappa}^{\sharp}$ for $\kappa \in \{ 1, 2, 3 \}$ Suppose $(v_1, v_2, v_3) \in \YYY_{ijk}$. We have $(\rho_{ij})_{v_1, v_2}: (\BBB_i)_{v_1}(Y_1, Y_1') \lra (\BBB_j)_{v_2}(Y_2, Y_2')$, $(\rho_{jk})_{v_2, v_3}: (\BBB_i)_{v_2}(Y_2, Y_2') \lra (\BBB_j)_{v_3}(Y_3, Y_3')$, $(\rho_{ik})_{v_1, v_3}: (\BBB_i)_{v_1}(Y_1, Y_1') \lra (\BBB_k)_{v_3}(Y_3, Y_3')$ and by \eqref{cocyc} 
\begin{equation} \label{cocycbis2}
(\rho_{jk})_{v_2, v_3} (\rho_{ij})_{v_1, v_2} = (\rho_{ik})_{v_1, v_3}.
\end{equation}
In particular, we have
\begin{equation}\label{cocycbis3}
(\rho_{ii})_{v_1,v_1} = 1: (\BBB_i)_{v_1}(Y_1,Y_1') \lra (\BBB_i)_{v_1}(Y_1,Y_1').
\end{equation}

\begin{theorem}\label{thmmapstack}
\begin{enumerate}
\item If $n \geq 2$, then $\Map$ is a prestack.
\item If $n \geq 3$, then $\Map$ is a stack.
\end{enumerate}\end{theorem}

\begin{proof}
Let $\sss = ((\varphi_i,F_i): (\vvv_i,\YYY_i) \lra (\uuu, \XXX))_i$ be a $2$-cover of $\XXX$ in $\Mas$. 

First, consider $\AAA, \BBB \in \Map(\XXX)$ and a compatible family of morphisms $G_i: \BBB|_i \lra \AAA|_i$ in $\Map(\YYY_i)$. Thus, for every $Y \in (\YYY_i)_V$, $Y' \in (\YYY_i)_{V'}$ and $v \in \vvv_i(V,V')$, we have a map
$(G_i)_{v, Y, Y'}: (\BBB|_i)_v(Y, Y') \lra (\AAA|_i)_v(Y,Y')$, i.e a map $(G_i)_{v, Y, Y'}: \BBB_{\varphi_i(v)}(F_i(Y), F_i(Y')) \lra \AAA_{\varphi_i(v)}(F_i(Y), F_i(Y'))$. To define a unique glueing $G: \BBB \lra \AAA$ in $\Map(\XXX)$, we consider $u: U \lra U'$ in $\uuu$, $X \in \XXX_U$, $X' \in \XXX_{U'}$. Since $\sss$ is a $1$-cover, there is an $i$ and $v: V \lra V$ in $\vvv_i$, $Y \in (\YYY_i)_V$, $Y' \in (\YYY_i)$ with $\varphi_i(v) = u$, $F_i(Y) = X$, $F_i(Y') = X'$. We put
$$G_{u, X, X'} = (G_i)_{v, Y, Y'}: \BBB_u(X,X') \lra \AAA_u(X,X').$$
Compatibility of the $G_i$ ensures that $G$ is well defined and the fact that $\sss$ is a $2$-cover can be used to show that $G$ is a $\uuu$-graded functor.

Next we consider a descent datum $(\BBB_i)_i$ with $\BBB_i \in \Map(\YYY_i)$ with compatible isomorphisms $\rho_{ij}: {\BBB_i}|^1_{ij} \cong {\BBB_j}|^2_{ij}$ in $\Map(\YYY_{ij})$. We are to define a $\uuu$-graded category $\BBB$ with underlying graded set $\XXX$. 
For $u: (U, X) \lra (U', X')$ in $\uuu^{\sharp}$, we first define the set
$$\XXX_u(X,X') = \coprod (\BBB_i)_v(B,B')$$
where the coproduct is taken over all $v: (V, B) \lra (V', B')$ in $\vvv_i^{\sharp}$ with $\varphi_i(v) = u$, $X = F_i(B)$ and $X' = F_i(B')$.
For $w: (W, C) \lra (W', C')$ in $\vvv_j^{\sharp}$ with $\varphi_j(w) = u$, $X = F_j(C)$ and $X' = F_j(C')$, we use the isomorphism
$$(\rho_{ij})_{v,w}: (\BBB_i)_v(B,B') \cong (\BBB_j)_w(C,C')$$
to declare when morphisms are equivalent. We thus obtain an equivalence relation on $\XXX_u(X,X')$ such that the quotient $\BBB_u(X,X')$ has a natural $k$-module structure, and the canonical morphisms
$$(\BBB_i)_v(B,B') \lra \BBB_u(X,X')$$
are $k$-linear isomorphisms. To define the composition on $\BBB$, we use the fact that $\sss$ is a $2$-cover to choose, for $u: (U, X) \lra (U', X')$ in $\uuu^{\sharp}$ and $u': (U', X') \lra (U'', X'')$ a single $(\varphi_i, F_i): (\vvv_i, \YYY_i) \lra (\uuu, \XXX)$ and $v: (V, B) \lra (V', B')$, $v': (V', B') \lra (V'', B'')$ with $\varphi^{\sharp}_i(v) = u$, $\varphi^{\sharp}_i(v') = u'$ and to use 
the composition
$$(\BBB_i)_{v'}(B', B'') \otimes (\BBB_i)_v(B,B') \lra (\BBB_i)_{v'v}(B,B'').$$
Finally, to show that the composition is associative, we use the fact that $\sss$ is a $3$-cover.
\end{proof}

\begin{lemma}\label{lemtopocomp}
Consider the functor $\Psi_0: \Mas \lra \Cat$. On $\Cat$, consider the pretopology of $n$-covers for some $n \in \N \cup \{ \infty \}$. On $\Mas$, consider the pretopology of collections consisting of cartesian morphisms with respect to $\Psi_0$, that are mapped to an $n$-cover under $\Psi_0$ (as described in Proposition \ref{proppretopx}).
Every cover for this pretopology is an $n$-cover in the sense of Proposition \ref{pretop}.
\end{lemma}

\begin{corollary}\label{corstack}
Consider the functor $\Psi = \Psi_0 \Psi_1: \Map \lra \Cat$ and consider the pretopology of $n$-covers from Proposition \ref{pretop} on $\Cat$ for some $n \in \N \cup \{ \infty \}$.
The pseudofunctor $\Map'$ associated to $\Psi$ is a functor.
\begin{enumerate}
\item If $n \geq 2$, then $\Map'$ is a prestack.
\item If $n \geq 3$, then $\Map'$ is a stack.
\end{enumerate}
\end{corollary}

\begin{proof}
This follows from Theorems \ref{thmmasstack}, \ref{thmmapstack}, Lemma \ref{lemtopocomp} and Proposition \ref{stacktrans}.
\end{proof}

\section{Bimodules}\label{parparbimod}

In this section we introduce the bicategory $\underline{\Map}$ of map-graded categories and bimodules between them. We develop the usual machinary of tensor (\S \ref{partensor}) and Hom (\S \ref{parhom}) functors in the map-graded context. Some attention is given to the fact that in this context, the notion of bimodule is very natural, while there seems to be no natural notion of module available which does not implicitly or explicitly use bimodules.

\subsection{Bifunctors}\label{parbifun}
Let $\uuu$ and $\vvv$ be small categories. A \emph{$\uuu$-$\vvv$-bifunctor} $S$ is by definition a functor
$$\vvv^{\op} \times \uuu \lra \Set: (V,U) \longmapsto S(V,U).$$
Functoriality translates into the existence of action maps
$$\uuu(U, U') \times S(V,U) \times \vvv(V',V) \lra S(V',U')$$
satisfying the natural associativity and identity axioms. For $\uuu$, we have the \emph{identity $\uuu$-bifunctor}
$$1_{\uuu}: \uuu^{\op} \times \uuu \lra \Set: (V,U) \longmapsto \uuu(V,U).$$
A \emph{morphism} between $\uuu$-$\vvv$-bifunctors $S$ and $S'$ is a natural transformation $S \lra S'$.
There is a natural category $\mathsf{Bifun}(\uuu, \vvv)$ of $\uuu$-$\vvv$-bifunctors and their morphisms.

\begin{example}\label{exbifun}
\begin{enumerate}
\item Let $\varphi: \vvv \lra \uuu$ be a functor between small categories. There is an associated $\uuu$-$\vvv$-bifunctor $S_{\varphi}$ with $S_{\varphi}(V,U) = \uuu(\varphi(V), U)$.
\item Let $\varphi: \uuu \lra \vvv$ be a functor between small categories. There is an associated $\uuu$-$\vvv$-bifunctor $S^{\varphi}$ with $S^{\varphi}(V,U) = \uuu(V, \varphi(U))$.
\end{enumerate}
\end{example}

Bifunctors can be composed in the following way. Consider an additional small category $\www$ and a $\vvv$-$\www$ bifunctor $T$. We define $S \circ T$ to be the $\uuu$-$\www$-bifunctor with 
$$S \circ T(W,U) = \coprod_{V \in \vvv} S(V,U) \times T(W,V)/ \sim$$
where for $s \in S(V,U)$, $v \in \vvv(V',V)$ and $t \in T(W, V')$ we have $(sv, t) \sim (s, vt)$.
There are canonical isomorphisms $(R \circ S) \circ T \cong R \circ (S \circ T)$, $1_{\uuu} \circ S \cong S$ and $S \circ 1_{\vvv} \cong S$. These give rise to a bicategory $\underline{\Cat}$ of categories, bifunctors and natural transformations.

\subsection{Bimodules and tensor functors}\label{parbimod}\label{partensor}

Consider a $\uuu$-graded category $\AAA$, a $\vvv$-graded category $\BBB$ and a $\uuu$-$\vvv$-bifunctor $S$. An \emph{$\AAA$-$S$-$\BBB$-bimodule} consists of $k$-modules $M_s(B,A)$ for $s \in S(V,U)$, $B \in \BBB_V$, $A \in \AAA_U$ with actions
$$\AAA_u(A, A') \otimes M_s(B,A) \otimes \BBB_v(B',B) \lra M_{usv}(B',A')$$
satisfying the natural associativity and identity axioms.

\begin{example}\label{exbimod}
\begin{enumerate}
\item Let $(\varphi, F): (\vvv, \BBB) \lra (\uuu, \AAA)$ be a graded functor. Let $S_{\varphi}$ be the $\uuu$-$\vvv$-bifunctor from Example \ref{exbifun}(1). For $s \in S_{\varphi}(V,U) = \uuu(\varphi(V), U)$ and  $B \in \BBB_V$, $A \in \AAA_U$, we put $(M_F)_s(B,A) = \AAA_{s}(F(B), A)$. This defines an $\AAA$-$S_{\varphi}$-$\BBB$-bimodule $M_F$.

\item Let $(\varphi, F): (\uuu, \AAA) \lra (\vvv, \BBB)$ be a graded functor. Let $S^{\varphi}$ be the $\uuu$-$\vvv$-bifunctor from Example \ref{exbifun}(2). For $s \in S^{\varphi}(V,U) = \vvv(V,\varphi(U))$ and $B \in \BBB_V$, $A \in \AAA_U$, we put $(M^F)_s(B,A) = \BBB_{s}(B, F(A))$. This defines an $\AAA$-$S^{\varphi}$-$\BBB$-bimodule $M^F$.
 
 \end{enumerate}
\end{example}

The $\AAA$-$S$-$\BBB$-bimodules form an abelian category $\Bimod_{S}(\AAA, \BBB)$ with the natural choice of morphisms. If $\uuu = \vvv$, we can take $S = 1_{\uuu}$ and  $\AAA$-$1_{\uuu}$-$\BBB$-bimodules are simply called $\AAA$-$\BBB$-bimodules. The corresponding category is denoted $\Bimod_{\uuu}(\AAA, \BBB)$. As usual, $\AAA$-$\AAA$-bimodules are called $\AAA$-bimodules and the corresponding category is denoted $\Bimod_{\uuu}(\AAA)$. In $\Bimod_{\uuu}(\AAA)$, we have the \emph{identity $\AAA$-bimodule} $1_{\AAA}$ with
${(1_{\AAA})}_u(A,A') = \AAA_u(A,A')$.

Similar to \cite[Proposition 2.11]{lowenmap}, the category $\Bimod_S(\AAA, \BBB)$ can be described as a module category over a linear category. To do so we define the linear category $\AAA^{\op} \otimes_S \BBB$ with
$$\Ob(\AAA^{\op} \otimes_S \BBB) = \coprod_{s \in S(V,U)} \AAA_V \times \BBB_U$$
and
$$\Hom((s, A, B), (s', A', B')) = \oplus_{s' = usv} \AAA_u(A,A') \otimes \BBB_v(B',B).$$

\begin{proposition}
There is an isomorphism of linear categories
$$\Bimod_S(\AAA,\BBB) \cong \Mod(\AAA^{\op} \otimes_S \BBB).$$
\end{proposition}

Consider an additional $\www$-graded category $\CCC$ and a $\vvv$-$\www$-bifunctor $T$. There is a natural tensor product 
$$\Bimod_S(\AAA, \BBB) \times \Bimod_T(\BBB, \CCC) \lra \Bimod_{S \circ T}(\AAA, \CCC): (M, N) \longmapsto M \otimes_{\BBB} N$$
with
$$(M \otimes_{\BBB} N)_{r}(C,A) = \oplus_{[(s,t)] = r,B} M_s(B,A) \otimes_k N_t(C,B)/\sim$$
where for $s \in S(V,U)$, $v \in \vvv(V',V)$ and $t \in T(W, V')$, $B\in \BBB_V$, $B' \in \BBB_{V'}$, $m \in M_s(B,A)$, $b \in \BBB_v(B,B')$ and $n \in N_t(C,B')$ we have $(mb, n) \sim (m, bn)$.

As an application of the tensor product, we obtain tensor actions
$$\otimes_{\BBB}: \Bimod_S(\AAA, \BBB) \times \Bimod_{\vvv}(\BBB) \lra \Bimod_S(\AAA, \BBB): (M, X) \longmapsto M \otimes_{\BBB} X$$
and
$$\otimes_{\AAA}: \Bimod_{\uuu}(\AAA) \times \Bimod_S(\AAA, \BBB) \lra \Bimod_S(\AAA, \BBB): (X, M) \longmapsto X \otimes_{\AAA} M.$$

Let $(\varphi, F): (\vvv, \BBB) \lra (\uuu, \AAA)$ be a graded functor. There is an induced functor
$$F^{\ast}: \Bimod_{\uuu}(\AAA) \lra \Bimod_{\vvv}(\BBB): M \longmapsto F^{\ast}M$$
with $(F^{\ast}M)_v(B,B') = M_{\varphi(v)}(F(B), F(B'))$.

\begin{example}
For a functor $\varphi: \vvv \lra \uuu$ and a $\uuu$-graded $\AAA$, consider the cartesian $\varphi$-graded functor $\delta^{\varphi, \AAA}: \AAA^{\varphi} \lra \AAA$. We obtain the induced functor
$$(-)^{\varphi} = (\delta^{\varphi, \AAA})^{\ast}: \Bimod_{\uuu}(\AAA) \lra \Bimod_{\vvv}(\AAA^{\varphi}): M \longmapsto M^{\varphi}$$
with $M^{\varphi}_v(A,A') = M_{\varphi(v)}(A,A')$.
\end{example}

Letting $S$ vary, we obtain a category $\Bimod(\AAA, \BBB)$ of $\AAA$-$\BBB$ bimodules in the following way. An $\AAA$-$\BBB$ bimodule consist of a $\uuu$-$\vvv$-bifunctor $S$ and an $\AAA$-$S$-$\BBB$-bimodule $N$. A morphism of bimodules $(\phi, F): (S, M) \lra (T, N)$ consists of a natural transformation $\phi: S \lra T$ of bifunctors and, for every $A \in \AAA_U$, $B \in \BBB_V$ and $s \in S(V,U)$, a morphism $$F_{s,B,A}: M_s(B,A) \lra N_{\phi(s)}(B,A)$$ such that the morphism $F_{s,B,A}$ are compatible with the actions of $\AAA$ and $\BBB$.

Combining the composition of bifunctors and the tensor product of bimodules, we obtain a tensor product
$$\Bimod(\AAA, \BBB) \times \Bimod(\BBB, \CCC) \lra \Bimod(\AAA, \CCC): ((S,M), (T,N)) \longmapsto (S \circ T, M \otimes_{\BBB} N).$$
If we consider a furter $\zzz$-graded category $\ZZZ$ and $(R,P) \in \Bimod(\ZZZ, \AAA)$, then there are natural isomorphisms
$$((R \circ S) \circ T, (P \otimes_{\AAA} M) \otimes_{\BBB} N) \cong (R \circ (S \circ T), P \otimes_{\AAA} (M \otimes_{\BBB} N))$$
and $(1_{\uuu} \circ S, 1_{\AAA} \otimes_{\AAA} M) \cong (S,M) \cong (S \circ 1_{\vvv}, M \otimes_{\BBB} 1_{\BBB})$.
These give rise to a bicategory $\underline{\Map}$ of map-graded categories, bimodules and bimodule morphisms.

\subsection{One sided bimodules and Hom functors}\label{parhom}
Bimodules are the natural notion when working with graded categories, but some bimodules can be considered to be more ``one-sided'' than others. Let $\uuu$ and $\vvv$ be categories with an $\uuu$-$\vvv$-bimodule $S$. Consider a $\uuu$-graded category $\AAA$ and a $\vvv$-graded category $\BBB$. Furthermore, consider the free $\uuu$-graded category $k\uuu$ and the free $\vvv$-graded category $k\vvv$ as in example \ref{exfreegr}. Then there are natural bimodule categories
$$\Mod_S(\AAA) = \Bimod_S(\AAA, k\vvv)$$ and
$$\Mod_S(\BBB) = \Bimod_S(k\uuu, \BBB).$$

\begin{example}\label{exmodmod}
If we take, in the first case, $\vvv = e$, then there is a unique $\uuu$-$e$-bimodule $S$ with $S(\ast, U) = \{ \ast \}$ for every $U \in \uuu$. Thus we obtain the category of left $\AAA$-modules
$$\Mod^l(\AAA) = \Mod_S(\AAA).$$
Similarly, taking $\uuu = e$ there  is a unique $e$-$\vvv$-bifunctor $S$ with $S(V, \ast) = \{ \ast \}$ for every $V \in \vvv$, yielding the category of right $\BBB$-modules
$$\Mod^r(\BBB) = \Mod_S(\BBB).$$
\end{example}

Now we return to the general situation of an underlying $\uuu$-$\vvv$-bimodule $S$. Then fixing one argument in an $\AAA$-$\BBB$-bimodule $M$ yields one-sided bimodules in the following sense.
Fix $U \in \uuu$ and $A \in \AAA_U$. Then $S$ yields an $e$-$\vvv$-bimodule $S_U$ with
$$S_U(V, \ast) = S(V,U)$$
and $M$ yields a $ke$-$\BBB$-bimodule $M_A \in \Bimod_{S_U}(ke, \BBB)$ with
$$(M_A)_s(B, \ast) = M_s(B,A).$$
Furthemore, the categories $\Bimod_{S_U}(ke, \BBB)$ are connected in the following way. Consider a morphism $u: U \lra U'$ in $\uuu$. Then there is an associated functor
$$u^{\ast}: \Bimod_{S_{U'}}(ke, \BBB) \lra \Bimod_{S_{U}}(ke, \BBB): M \longmapsto u^{\ast}M$$
with $$(u^{\ast}M)_s(B, \ast) = M_{us}(B, \ast).$$
Now we are ready to define the Hom functor
$$\Hom_{\BBB}: \Bimod_S(\AAA, \BBB) \times \Bimod_S(\AAA, \BBB) \lra \Bimod_{\uuu}(\AAA): (M, N) \longmapsto \Hom_{\BBB}(M,N)$$
where for $u: U \lra U'$ in $\uuu$ and $A \in \AAA_U$, $A' \in \AAA_{U'}$
$$\Hom_{\BBB}(M,N)_u(A,A') = \Bimod_{S_U}(ke, \BBB)(M_A, u^{\ast} N_{A'}).$$
There are natural morphisms
$$1_{\AAA} \lra \Hom_{\BBB}(M,M)$$
given by the natural action maps
$$\AAA_{u}(A,A') \otimes M_s(B,A) \lra M_{us}(B, A').$$

In a similar way we define
$$\Hom_{\AAA^{\op}}: \Bimod_S(\AAA, \BBB) \times \Bimod_S(\AAA, \BBB) \lra \Bimod_{\vvv}(\BBB): (M,N) \longmapsto \Hom_{\AAA^{\op}}(M,N).$$

\begin{proposition}\label{propadj}
\begin{enumerate}
\item The functor $- \otimes_{\AAA} M: \Bimod_{\uuu}(\AAA) \lra \Bimod_S(\AAA, \BBB)$ is left adjoint to $\Hom_{\BBB}(M,-): \Bimod_S(\AAA, \BBB) \lra  \Bimod_{\uuu}(\AAA)$.
\item The functor $M \otimes_{\BBB} -: \Bimod_{\vvv}(\BBB) \lra \Bimod_S(\AAA, \BBB)$ is left adjoint to $\Hom_{\AAA^{\op}}(M, -): \Bimod_S(\AAA, \BBB) \lra  \Bimod_{\vvv}(\BBB)$.
\end{enumerate}
\end{proposition}

\section{Functoriality of map-graded Hochschild complexes} \label{parfunct}

The Hochschild complex $\CC_{\uuu}(\AAA)$ of a map-graded category $(\uuu, \AAA)$ is defined in analogy with the Hochschild complex of an algebra \cite{lowenmap}, naturally making use of the simplicial structure of the nerve $\nnn(\AAA) = \nnn(\uuu^{\sharp})$. In this section we investigate the functoriality properties of map-graded Hochschild complexes. A morphism of map-graded categories is called \emph{subcartesian} if it is cartesian with respect to $\Map \lra \Mas$. This amounts to the fact that all maps occuring between $\Hom$-modules are isomorphisms. For instance, if we consider a fully faithful functor $\BBB \lra \AAA$ between $k$-linear categories as a graded functor beteen trivially graded categories, it is subcartesian (Example \ref{exfff}).
In Proposition  \ref{propfunct}, we show that taking Hochschild complexes is functorial with respect to subcartesian functors. This constitutes a natural generalization of the \emph{limited functoriality} of Hochschild complexes of linear categories (see \cite{kellerdih} in the differential graded context). Let $\Map_{sc} \subseteq \Map$ denote the full subcategory of subcartesian morphisms. We endow $\Map_{sc}$ with the pretopology of $n$-covers from Definition \ref{defncover}. In Theorem \ref{mainsheaf}, we show that the functor
$$\CC^n: \Map_{sc} \lra \Mod(k): (\uuu, \AAA) \longmapsto \CC^n_{\uuu}(\AAA)$$
is a sheaf. As an application of the theorem, in \S \ref{parMV} we obtain a Mayer-Vietoris sequence of Hochschild complexes
$$0 \lra \CC_{\uuu}(\AAA) \lra \CC_{\vvv_1}(\AAA^{\varphi_1}) \oplus \CC_{\vvv_2}(\AAA^{\varphi_2}) \lra \CC_{\vvv_1 \cap \vvv_2}(\AAA^{\varphi}) \lra 0$$
for a map-graded category $(\uuu, \AAA)$ and two cartesian morphisms $(\vvv_i, \AAA^{\varphi_i}) \lra (\uuu, \AAA)$ and $(\vvv_1 \cap \vvv_2, \AAA^{\varphi}) \lra (\uuu, \AAA)$ associated to subcategories $\varphi_i: \vvv_i \subseteq \uuu$ constituting an $n$-cover of $\uuu$ for all $n \geq 0$. Finally, in \S \ref{parcensor}, we discuss censoring subcategories as a natural generalization of the censoring relations from \cite[\S 4.3]{lowenvandenberghhoch}.

\subsection{The map-graded Hochschild complex}
Let $\AAA$ be a $\uuu$-graded category. Let $\uuu^{\sharp}$ be the category defined in Example \ref{exsharp}. 
Recall from \S \ref{parsite} that the nerve of $\AAA$ is defined to be the simplicial set $\nnn(\AAA) = \nnn(\uuu^{\sharp})$ with $n$-simplices $\sigma = (u, A)$ given by data
$$\xymatrix{ {A_0} & {A_1} & {\dots} & {A_n}\\ {U_0} \ar[r]_{u_0} & {U_1} \ar[r]_{u_1} & {\dots} \ar[r]_{u_{n-1}} & {U_n}}$$
with $u_i \in \uuu, A_i \in \AAA_{U_i}$.
For $u \in \nnn(\uuu)_n$, we will use the notation
$$|u| = u_{n-1}\cdots u_1 u_0.$$
A graded functor $(\varphi, F): (\vvv, \BBB) \lra (\uuu, \AAA)$ induces a functor $\vvv^{\sharp} \lra \uuu^{\sharp}$ and hence a map $\nnn(F): \nnn(\BBB) \lra \nnn(\AAA)$.

Let $M$ be an $\AAA$-bimodule.
The \emph{Hochschild complex of $\AAA$ with values in $M$} naturally arises from this simplicial structure as the complex $\CC_{\uuu}(\AAA,M)$ with
$$\CC_{\uuu}^n(\AAA) = \prod_{(u, A) \in \nnn(\AAA)} \Hom_k(\AAA_{u_{n-1}}(A_{n-1}, A_n) \otimes \dots \otimes \AAA_{u_0}(A_0, A_1), M_{|u|}(A_0, A_n))$$
with the simplicial Hochschild differential.
We put $\CC_{\uuu}(\AAA) = \CC_{\uuu}(\AAA, 1_{\AAA})$. This complex is in fact a $B_{\infty}$-algebra \cite{lowenmap}.

\begin{example}
Let $\AAA$ be a linear category. For all the $\uuu$-gradings on $\AAA$ of Example \ref{exstandgr}, the corresponding Hochschild complexes $\CC_{\uuu}(\AAA)$ are canonically isomorphic to $\CC(\AAA)$. This results from the fact that all the nerves of these graded categories are canonically isomorphic to $\nnn(\AAA)$.
\end{example}

\subsection{Limited functoriality}\label{parlimfun}
It is well known that the Hochschild complex of linear categories satisfies so called ``limited functoriality'' with respect to inclusions of full subcategories, see \cite{kellerdih} for the more general statement for differential graded categories. In this section we discuss a limited functoriality property for map-graded categories.

Recall that by Proposition \ref{lemcart1}, a graded functor $(\varphi, F): (\vvv, \BBB) \lra (\uuu, \AAA)$ is subcartesian in the sense of Definition \ref{defcart} provided that for every $v: V \lra V'$ in $\vvv$, $B \in \BBB_V$, $B' \in \BBB_{V'}$, the map
$$\BBB_v(B,B') \lra \AAA_{\varphi(v)}(F(B), F(B'))$$
is an isomorphism.

\begin{proposition}\label{propfunct}
Consider a graded functor $(\varphi, F): (\vvv, \BBB) \lra (\uuu, \AAA)$. Let $M$ be an $\AAA$-bimodule with induced $\BBB$-bimodule $F^{\ast}M$. 
\begin{enumerate}
\item There is a canonical map
$${\prod_{(u,A) \in \nnn(\AAA)_n} \Hom_k(\AAA_{u_{n-1}}(A_{n-1}, A_n) \otimes \dots \otimes \AAA_{u_0}(A_0, A_1), M_{|u|}(A_0, A_n))}$$
$$\xymatrix{ {} \ar[r]_-{(F^{\ast}_M)^n} & {}}$$
$${\prod_{(v,B) \in \nnn(\BBB)_n} \Hom_k(\BBB_{v_{n-1}}(B_{n-1}, B_n) \otimes \dots \otimes \BBB_{v_0}(B_0, B_1), (F^{\ast} M)_{|v|}(B_0, B_n))}$$
given by
$$(F^{\ast}_M)^n((\phi_{(u,A)})_{(u,A)}) = (\phi_{(\varphi(v), F(B))} \circ F^{\otimes n})_{(v,B)}.$$
\item The maps $(F^{\ast}_M)^n$ determine a morphism of complexes
$$F^{\ast}_M: \CC_{\uuu}(\AAA, M) \lra \CC_{\vvv}(\BBB, F^{\ast}M).$$
\item If $(\varphi, F)$ is subcartesian, we have $F^{\ast}1_{\AAA} \cong 1_{\BBB}$ and the maps $(F^{\ast}_{1_{\AAA}})^n$ determine a morphism of $B_{\infty}$-algebras
$$F^{\ast}: \CC_{\uuu}(\AAA) \lra \CC_{\vvv}(\BBB)$$
with
$$(F^{\ast})^n((\phi_{(u,A)})_{(u,A)}) = (F^{-1} \circ \phi_{(\varphi(v), F(B))} \circ F^{\otimes n})_{(v,B)}.$$
\end{enumerate}
\end{proposition}

Clearly, graded categories with subcartesian graded functors consitute a subcategory $\Map_{sc} \subseteq \Map$. Let $B_{\infty}$ denote the category of $B_{\infty}$-algebras and morphisms. By Proposition \ref{propfunct} (3), we obtain a contravariant functor
\begin{equation} \label{CC}
\CC: \Map_{sc} \lra B_{\infty}: (\uuu, \AAA) \longmapsto \CC_{\uuu}(\AAA).
\end{equation}

\begin{definition}\label{defsurjinj}
Let $(\varphi, F): (\vvv, \BBB) \lra (\uuu, \AAA)$ be a graded functor between graded categories (resp. a graded map between graded sets). Let $n \in \N \cup \{ \infty \}$.
\begin{enumerate}
\item $(\varphi, F)$ is \emph{$n$-surjective} if the canonical
$\nnn_k(F): \nnn_k(\BBB) \lra \nnn_k(\AAA)$ is surjective for $k \leq n$.
\item $(\varphi, F)$ is \emph{$n$-injective} if $\nnn_k(F): \nnn_k(\BBB) \lra \nnn_k(\AAA)$ is injective for $k \leq n$. 
\item $(\varphi, F)$ is \emph{injective} if $\nnn_1(\varphi): \nnn_1(\vvv) \lra \nnn_1(\uuu)$ is injective and every $F: \BBB_V \lra \AAA_{\varphi(V)}$ is injective.
\end{enumerate}
\end{definition}

\begin{remark}\label{remsurjinj}
\begin{enumerate}
\item $(\varphi, F)$ is $n$-surjective if and only if the collection containing $(\varphi, F)$ as single element is an $n$-cover in $\Map$ (resp. $\Mas$) in the sense of Definition \ref{defncover}.
\item If $(\varphi, F)$ is $1$-injective, then it is $\infty$-injective.
\item If $(\varphi, F)$ is injective, then it is $1$-injective.
\item $(\varphi,F)$ is $1$-injective if and only if $(\varphi^{\sharp}, F^{\sharp})$ defined as in Example \ref{exfunsharp} is injective.
\end{enumerate}
\end{remark}

\begin{proposition}\label{propsurjinj}
Let $(\varphi, F): (\vvv, \BBB) \lra (\uuu, \AAA)$ be a subcartesian graded functor and let $M$ be an $\AAA$-bimodule. Let $n \in \N$.
\begin{enumerate}
\item If $(\varphi, F)$ is $n$-injective then $(F^{\ast})^n: \CC^n_{\uuu}(\AAA, M) \lra \CC^n_{\vvv}(\BBB, F^{\ast}M)$ is surjective.
\item If $(\varphi, F)$ is $n$-surjective then $(F^{\ast})^n: \CC^n_{\uuu}(\AAA, M) \lra \CC^n_{\vvv}(\BBB, F^{\ast}M)$ is injective.
\end{enumerate}
\end{proposition}

\begin{proof} 
(1) If $\nnn_n(F): \nnn_n(\BBB) \lra \nnn_n(\AAA)$ is injective, the map $(F^{\ast})^n$ is isomorphic to a projection on a subproduct, whence surjective.
(2) Looking at the prescription for $(F^{\ast})^n$, if $\phi \in \CC^n_{\uuu}(\AAA, M)$ is such that $\phi_{(\varphi(v), F(B))} = 0$ for every $(v,B) \in \nnn_n(\BBB)$, then using $n$-surjectivity we can write every $(u,A) \in \nnn(\AAA)$ as $(u, A) = (\varphi(v), F(B))$ for some $(v,B)$ whence $\phi_{(u,A)} = 0$.
\end{proof}

\begin{example}\label{exisosharp}
Consider the graded functor $(\varphi, F): (\uuu^{\sharp}, \AAA^{\sharp}) \lra (\uuu, \AAA)$ as in Example \ref{exfunsharp}. Clearly, $(\varphi, F)$ is subcartesian and $\nnn(F): \nnn(\AAA^{\sharp}) = \nnn((\uuu^{\sharp})^{\sharp}) \cong \nnn(\uuu^{\sharp}) \lra \nnn(\uuu^{\sharp})$ is an isomorphism. Consequently,
$F^{\ast}: \CC_{\uuu}(\AAA) \cong \CC_{\uuu^{\sharp}}(\AAA^{\sharp})$ is an isomorphism.
\end{example}

\subsection{The sheaf of Hochschild complexes}\label{parsheafhoch}
Consider the presheaf
$$\CC: \Map_{sc} \lra B_{\infty}: (\uuu, \AAA) \longmapsto \CC_{\uuu}(\AAA)$$
from \eqref{CC}.
Let $n \in \N \cup \{ \infty \}$ be fixed. We endow $\Map_{sc} \subseteq \Map$ with the pretopology of $n$-covers, as described in Proposition \ref{pretop}.

\begin{theorem}\label{mainsheaf}
The presheaf $\CC^n: \Map_{sc} \lra \Mod(k): (\uuu, \AAA) \longmapsto \CC^n_{\uuu}(\AAA)$ is a sheaf.
If $n = \infty$, then $\CC: \Map_{sc} \lra B_{\infty}$ is a sheaf of $B_{\infty}$-algebras.
\end{theorem}

\begin{proof}
Let $((\varphi_i, F_i): (\vvv_i, \BBB_i) \lra (\uuu, \AAA))_i$ be an $n$-cover in $\Map_{sc}$. A compatible family of elements for this cover consists of $n$-cocycles
$$\phi_i \in \CC_{\vvv_i}(\BBB_i)$$
such that for every pullback diagram
$$\xymatrix{ {(\vvv_i, \BBB_i)} \ar[r]^{(\varphi_i, F_i)} & {(\uuu, \AAA)}\\   {(\vvv_i \times_{\uuu} \vvv_j, \BBB_i \times_{\AAA} \BBB_j)} \ar[u]^{(\alpha_1, G_1)} \ar[r]_-{(\alpha_2, G_2)} & {(\vvv_j, \BBB_j)} \ar[u]_-{(\varphi_j, F_j)} }$$
we have $$G_i^{\ast}(\phi_i) = G_j^{\ast}(\phi_j).$$
To define the unique glueing of this family on $\AAA$, we must define for every $(u, A) \in \nnn_n(\AAA)$ a corresponding cocycle 
$$\phi_{(u,A)} \in \Hom_k(\AAA_{u_{n-1}}(A_{n-1}, A_n) \otimes \dots \otimes \AAA_{u_0}(A_0, A_1), \AAA_{|u|}(A_0, A_n)).$$
Since the collection $(\nnn_n(F_i))_i$ is jointly surjective, there in an $i$ and $(v, B) \in \nnn_n(\BBB_i)$ for which $\varphi_i(v) = u$, $F_i(B) = A$. We thus have isomorphisms
$F_i: (\BBB_i)_{v_i}(B_i, B_{i+1}) \lra \AAA_{u_i}(A_i, A_{i+1})$
and
$F_i: (\BBB_i)_{|v|}(B_0, B_{n}) \lra \AAA_{|u|}(A_0, A_{n})$.
We put
$$\phi_{(u,A)} = F_i \circ {\phi_i}_{(v,B)} \circ (F_i^{-1})^{\otimes n}.$$
It remains to show that this is well defined. Suppose there is another $j$ and $(w, C) \in \nnn_n(\BBB_j)$ for which $\varphi_j(w) = u$, $F_j(C) = A$. We are to show that 
$F_i \circ {\phi_i}_{(v,B)} \circ (F_i^{-1})^{\otimes n} = F_j \circ {\phi_j}_{(w,C)} \circ (F_j^{-1})^{\otimes n}$.
Now $G_1^{\ast}$ maps the collection $\phi_i$ to a collection $\psi_i$ with
$${\psi_i}_{((v, w), (B, C))} = (G_1)^{-1} \circ {\phi_i}_{(v,B)} \circ (G_1)^{\otimes n}$$
and $G_2^{\ast}$ maps the collection $\phi_j$ to a collection $\psi_j$ with
$${\psi_j}_{((v, w), (B,C))} = (G_2)^{-1} \circ {\phi_j}_{(w,C)} \circ (G_2)^{\otimes n}.$$
We thus have
$$F_i \circ {\phi_i}_{(v,B)} \circ (F_i^{-1})^{\otimes n} = (F_i G_1) \circ {\psi_i}_{((v, w), (B, C))} \circ ((F_i G_1)^{-1})^{\otimes n}$$
and
$$F_j \circ {\phi_j}_{(w,C)} \circ (F_j^{-1})^{\otimes n} = (F_j G_2) \circ {\psi_j}_{((v, w), (B, C))} \circ ((F_j G_2)^{-1})^{\otimes n}.$$
By the commutativity of the pullback square and the compatibility assumption $\psi_i = \psi_j$, these two expressions are equal as desired.  
\end{proof}

Let $\AAA$ be a $\uuu$-graded category with underlying graded set $\XXX$. There are natural functors
$$\Phi_1: \Mas/\XXX \lra \Map: ((\varphi, F): (\vvv, \YYY) \rightarrow (\uuu, \XXX)) \longmapsto \AAA^{(\varphi, F)}$$
and
$$\Phi_2: \Cat/\uuu \lra \Map: (\varphi: \vvv \rightarrow \uuu) \longmapsto \AAA^{\varphi}.$$

Let $n \in \N \cup {\infty}$. Endow $\Mas/ \XXX$ and $\Cat/ \uuu$ with the pretopologies of $n$-covers induced from the ones on $\Mas$ and $\Cat$, and also endow $\Map$ with the pretopology of $n$-covers.

\begin{proposition}
The functors $\Phi_1$ and $\Phi_2$ both map covers to covers. 
\end{proposition}

\begin{corollary}\label{corsheafhoch}
The presheaves
$$\CC^n_1 = \CC^n \circ \Phi_1: \Mas/ \XXX \lra \Mod(k): ((\varphi, F): (\vvv, \YYY) \rightarrow (\uuu, \XXX)) \longmapsto \CC^n_{\vvv}(\AAA^{(\varphi, F)})$$
and
$$\CC^n_2 = \CC^n \circ \Phi_2:  \Cat/\uuu \lra \Mod(k): (\varphi: \vvv \rightarrow \uuu) \longmapsto \CC^n_{\vvv}(\AAA^{\varphi})$$
are sheaves.
\end{corollary}

We end this section by noting that for a fixed $\AAA$-bimodule $M$, one similarly has:

\begin{proposition}\label{sheafbimod}
There are natural sheaves
$$\CC^n_{1,M}: \Mas/\XXX \lra \Mod(k): ((\varphi, F): (\vvv, \YYY) \rightarrow (\uuu, \XXX)) \longmapsto \CC^n_{\vvv}(\AAA^{(\varphi, F)}, F^{\ast}M)$$
and
$$\CC^n_{2, M} = \CC^n \circ \Phi_2:  \Cat/\uuu \lra \Mod(k): (\varphi: \vvv \rightarrow \uuu) \longmapsto \CC^n_{\vvv}(\AAA^{\varphi}, M^{\varphi}).$$
\end{proposition}

\begin{remark}
It is possible to formulate a version of Theorem \ref{mainsheaf} taking map graded categories endowed with a bimodule as input data for $\CC$, such that Proposition \ref{sheafbimod} is obtained as a corollary. The details are left to the reader.
\end{remark}

\subsection{$\infty$-surjections}

Let $(\varphi, F): (\vvv, \BBB) \lra (\uuu, \AAA)$ be an $\infty$-surjective morphism in $\Map$. To formulate the sheaf property of $\CC^n$ for the corresponding $\infty$-cover, we look at the pullback
$$\xymatrix{ {(\vvv, \BBB)} \ar[r]^{(\varphi, F)} & {(\uuu, \AAA)}\\   {(\vvv \times_{\uuu} \vvv, \BBB \times_{\AAA} \BBB)} \ar[u]^{(\alpha_1, G_1)} \ar[r]_-{(\alpha_2, G_2)} & {(\vvv, \BBB)} \ar[u]_-{(\varphi, F)} }$$
According to Theorem \ref{mainsheaf}, we obtain an exact sequence of $B_{\infty}$-algebras
$$\xymatrix{ 0 \ar[r] & {\CC_{\uuu}(\AAA)} \ar[r]_{F^{\ast}} & {\CC_{\vvv}(\BBB)} \ar[r]_-{G_2^{\ast} - G_1^{\ast}} & {\CC_{\vvv \times_{\uuu} \vvv}(\BBB \times_{\AAA} \BBB).} }$$
Define the complex $$\CC_{\vvv/\uuu}(\BBB/\AAA) = \Beeld(G_2^{\ast} - G_1^{\ast}).$$
Then we obtain a long exact cohomology sequence
$$\xymatrix{ {\dots} \ar[r] & {HH^i_{\uuu}(\AAA)} \ar[r]_-{F^{\ast}} & {HH^i_{\vvv}(\BBB)} \ar[r]_-{G_2^{\ast} - G_1^{\ast}} & {HH^i_{\vvv/\uuu}(\BBB/\AAA)} \ar[r] & {\dots}}$$


\subsection{Mayer-Vietoris sequences}\label{parMV}
Let $\AAA$ be a $\uuu$-graded category. Consider two subcartesian $1$-injections
$(\varphi_1, F_1): (\vvv_1, \BBB_1) \lra (\uuu, \AAA)$ and $(\varphi_2, F_2): (\vvv_2, \BBB_2) \lra (\uuu, \AAA)$ that together constitute an $\infty$-cover of $\AAA$. It is our aim to formulate the sheaf property, so we have to look at all possible pullbacks between $(\varphi_1, F_1)$ and $(\varphi_2, F_2)$. For $i \in \{1,2\}$, we first look at the pullback
$$\xymatrix{ {(\vvv_i, \BBB_i)} \ar[r]^{(\varphi_i, F_i)} & {(\uuu, \AAA)}\\   {(\vvv_i \times_{\uuu} \vvv_i, \BBB_i \times_{\AAA} \BBB_i)} \ar[u]^{(\alpha_1, G_1)} \ar[r]_-{(\alpha_2, G_2)} & {(\vvv_i, \BBB_i).} \ar[u]_-{(\varphi_i, F_i)} }$$
The functor $(-)^{\sharp}$ from Example \ref{exfunsharp} maps this pullback to the pullback of $(\varphi_i^{\sharp}, F_i^{\sharp}): (\vvv_i^{\sharp}, \BBB_i^{\sharp}) \lra (\uuu^{\sharp}, \AAA^{\sharp})$ with itself. Now $(\varphi_i^{\sharp}, F_i^{\sharp})$ is an injection by Remark \ref{remsurjinj} (4), whence it is easily seen to be a monomorphism in $\Map$. Thus, the pullback of this morphism with itself is given by identity morphisms, inducing identity morphisms between Hochschild complexes. Using Example \ref{exisosharp}, we conclude that
an element $\phi \in \CC_{\vvv_i}(\BBB_i)$ satisfies
$$G^{\ast}_1(\phi) = G^{\ast}_2(\phi) \in \CC_{\vvv_i \times_{\uuu} \vvv_i}(\BBB_i \times_{\AAA} \BBB_i).$$
It thus remains to look at the pullback
$$\xymatrix{ {(\vvv_1, \BBB_1)} \ar[r]^{(\varphi_1, F_1)} & {(\uuu, \AAA)}\\   {(\vvv_1 \times_{\uuu} \vvv_2, \BBB_1 \times_{\AAA} \BBB_2)} \ar[u]^{(\alpha_1, G_1)} \ar[r]_-{(\alpha_2, G_2)} & {(\vvv_2, \BBB_2)} \ar[u]_-{(\varphi_2, F_2)} }$$
from which we obtain an exact sequence of complexes:
$$\xymatrix{ 0 \ar[r] & {\CC_{\uuu}(\AAA)} \ar[r]_-{\begin{pmatrix} F_1^{\ast} \\ F_2^{\ast} \end{pmatrix}} & {\CC_{\vvv_1}(\BBB_1) \oplus \CC_{\vvv_2}(\BBB_2)} \ar[r]_-{\begin{pmatrix} - G_1^{\ast} & G_2^{\ast} \end{pmatrix}} & {\CC_{\vvv_1 \times_{\uuu} \vvv_2}(\BBB_1 \times_{\AAA} \BBB_2)} \ar[r] & 0.}$$
Here, the sequence is left exact by the sheaf property Theorem \ref{mainsheaf} and moreover right exact by Proposition \ref{propsurjinj} (2).
We thus obtain a long exact cohomology sequence
$$\xymatrix{ {\dots} \ar[r] & {HH^i_{\uuu}(\AAA)} \ar[r]_-{\begin{pmatrix} F_1^{\ast} \\ F_2^{\ast} \end{pmatrix}} & {HH^i_{\vvv_1}(\BBB_1) \oplus HH^i_{\vvv_2}(\BBB_2)} \ar[r]_-{\begin{pmatrix} - G_1^{\ast} & G_2^{\ast} \end{pmatrix}} & {HH^i_{\vvv_1 \times_{\uuu} \vvv_2}(\BBB_1 \times_{\AAA} \BBB_2)} \ar[r] & {\dots}}$$

\begin{example}
Let $\AAA$ be a $\uuu$-graded category and let $\varphi_1: \vvv_1 \subseteq \uuu$ and $\varphi_2: \vvv_2 \subseteq \uuu$ be subcategories that together constitute an $\infty$-cover of $\uuu$ in $\Cat$. Then the induced cartesian morphisms $(\vvv_1, \AAA^{\varphi_1}) \lra (\uuu, \AAA)$ and $(\vvv_2, \AAA^{\varphi_2}) \lra (\uuu, \AAA)$ in $\Map$ constitute an $\infty$-cover by $1$-injections. Put $\varphi: \vvv_1 \cap \vvv_2 \subseteq \uuu$ the inclusion. By Example \ref{exint} we obtain an exact sequence of complexes
$$\xymatrix{ 0 \ar[r] & {\CC_{\uuu}(\AAA)} \ar[r]_-{\begin{pmatrix} \varphi_1^{\ast} \\ \varphi_2^{\ast} \end{pmatrix}} & {\CC_{\vvv_1}(\AAA^{\varphi_1}) \oplus \CC_{\vvv_2}(\AAA^{\varphi_2})}
 \ar[r]_-{\begin{pmatrix} - \alpha_1^{\ast} & \alpha_2^{\ast} \end{pmatrix}} & {\CC_{\vvv_1 \cap \vvv_2}(\AAA^{\varphi})} \ar[r] & 0.}$$
Based upon Proposition \ref{sheafbimod}, one obtains a version of this sequence involving bimodules.
\end{example}

\begin{example}
Let $(X, \ooo)$ be a ringed space with an acyclic basis $\bbb$ of open sets in the sense of \cite{lowenvandenberghhoch}. Consider open sets $U_0, U_1$ and $U_2$ with $U_0 = U_1 \cup U_2$ and put $U_{12} = U_1 \cap U_2$. Put 
$$\bbb_{\star} = \{ B \in \bbb \,\, |\,\, B \subseteq U_{\star} \}$$
for $\star = 1, 2, 12$. Then $\bbb_{\star}$ is an acyclic basis for $U_{\star}$ and $\bbb_{12} = \bbb_1 \cap \bbb_2$. Futher, $\bbb_0 = \bbb_1 \cup \bbb_2$ is an acyclic basis for $U_0$ (which is smaller that the basis $U_0$ immediately inherits form $\bbb$). 
Now consider all the posets $(\bbb_{\star}, \subseteq)$ as small categories. Then we have a cover of $\bbb_0$ consisting of the injections $\varphi_i: \bbb_i \lra \bbb_0$ for $i = 1,2$ and we have $\bbb_{12} = \bbb_1 \cap \bbb_2$ as categories. 
On each of these categories, we consider the corresponding $\bbb_{\star}$-graded category $\BBB_{\star}$ with $(\BBB_{\star})_V = \{V\}$ and
$$\BBB_{\star}(W,V) = \ooo(W)$$
for $W \subseteq V$, i.e. $\BBB_{\star}$ is the map-graded category associated to the restricted structure sheaf on $\bbb_{\star}$.
Clearly, $\BBB_1 = (\BBB_0)^{\varphi_1}$ and similarly for the other injections. According to \cite{lowenvandenberghhoch}, $\CC_{\bbb_{\star}}(\BBB_{\star})$ computes the Hochschild cohomology of the ringed space $(U_{\star}, \ooo|_{U_{\star}})$.
From the above, we obtain an exact sequence
$$\xymatrix{ 0 \ar[r] & {\CC_{\bbb_0}(\BBB_0)} \ar[r] & {\CC_{\bbb_1}(\BBB_1) \oplus \CC_{\bbb_2}(\BBB_2)} \ar[r] & {\CC_{\bbb_{12}}(\BBB_{12})} \ar[r] & 0}$$
and an induced long exact cohomology sequence
$$\xymatrix{ {\dots} \ar[r] & {HH^i(U_0)} \ar[r] & {HH^i(U_1) \oplus HH^i(U_2)} \ar[r] & {HH^i(U_{12})} \ar[r] & {\dots}}$$
The more subtle problem of defining a sheaf of Hochschild complexes on quasi-compact opens of a quasi-compact separated scheme was solved in \cite{lowensheafhoch}. The existence of Mayer-Vietoris sequences for arbitrary ringed spaces was shown in \cite[\S 7.9]{lowenvandenberghhoch} making use of the definition of the Hochschild complex of $X$ as the Hochschild complex of the linear category of injectives in the category of sheaves $\Mod(X)$. We come back to this approach in \S \ref{parcomp}.
\end{example}

\subsection{$1$-injections}
Let $(\varphi, F): (\vvv, \BBB) \lra (\uuu, \AAA)$ be a $1$-injective subcartesian graded functor and let $M$ be an $\AAA$-bimodule. We thus have injections
$\nnn_n(\varphi^{\sharp}): \nnn_n(\vvv^{\sharp}) \lra \nnn_n(\uuu^{\sharp})$ for all $n \in \N$. By Proposition \ref{propsurjinj}, we obtain a surjective morphism
$$F^{\ast}: \CC_{\uuu}(\AAA, M) \lra \CC_{\vvv}(\BBB, F^{\ast}M).$$
Define the complex
$$\CC_{\uuu \setminus \vvv^{\sharp}}(\AAA, M) = \Kern(F^{\ast}).$$
Then $\CC_{\uuu \setminus \vvv^{\sharp}}^n(\AAA, M)$ is isomorphic to
\begin{equation}\label{eqsupport}
\prod_{(u,A) \notin \Beeld(\nnn_n(\varphi^{\sharp}))} \Hom_k(\AAA_{u_{n-1}}(A_{n-1}, A_n) \otimes \dots \otimes \AAA_{u_0}(A_0, A_1), M_{|u|}(A_0, A_n)).
\end{equation}

\begin{proposition}
The following are equivalent:
\begin{enumerate}
\item $(u, A) \in \Beeld(\nnn_n(\varphi^{\sharp}))$.
\item For every $u_i: A_i \lra A_{i+1}$ in $\nnn_1(\uuu^{\sharp})$ occuring in $(u, A)$, we have $u_i \in \Beeld(\nnn_1(\varphi^{\sharp}))$.
\end{enumerate}
\end{proposition}

\begin{proof}
Suppose (2) holds. Let $u_i = \varphi^{\sharp}(v): A_i \lra A_{i+1}$ for $v: B \lra B'$ in $\nnn_1(\vvv^{\sharp})$ and let  $u_{i+1} = \varphi^{\sharp}(w): A_{i+1} \lra A_{i+2}$ for $w: C \lra C'$ in $\nnn_1(\vvv^{\sharp})$. Then $\varphi^{\sharp}(1_{B'}) = 1_{A_{i+1}} = \varphi^{\sharp}(1_{C})$ whence $B' = C$.
\end{proof}

We call $\CC_{\uuu \setminus \vvv^{\sharp}}(\AAA, M)$ the \emph{Hochschild complex of $\AAA$ with support outside $\vvv^{\sharp}$ (and values in $M$)}.
We thus obtain an exact sequence of complexes
\begin{equation}\label{eqkern}
0 \lra \CC_{\uuu \setminus \vvv^{\sharp}}(\AAA, M) \lra \CC_{\uuu}(\AAA, M) \lra \CC_{\vvv}(\BBB, F^{\ast}M) \lra 0.
\end{equation}
and a long exact cohomology sequence
$$\xymatrix{{\dots} \ar[r] & {HH^i_{\uuu \setminus \vvv^{\sharp}}(\AAA, M)} \ar[r] & {HH^i_{\uuu}(\AAA, M)} \ar[r]_-{F^{\ast}} & {HH^i(\BBB, F^{\ast} M)} \ar[r] & {\dots} }$$

\begin{example}\label{exuminv}
Let $\AAA$ be a $\uuu$-graded category and $\varphi: \vvv \subseteq \uuu$ a subcategory. Consider the cartesian morphism $\delta^{\varphi, \AAA}: \AAA^{\varphi} \lra \AAA$. Then $(\varphi, \delta^{\varphi, \AAA}): (\vvv, \AAA^{\varphi}) \lra (\uuu, \AAA)$ is injective whence $1$-injective. In this case the following are equivalent:
\begin{enumerate}
\item $(u, A) \in \Beeld(\nnn_n(\varphi^{\sharp}))$.
\item For every $u_i: U_i \lra U_{i+1}$ in $\nnn_1(\uuu)$ occuring in $u$, we have $u \in \Beeld(\nnn_1(\varphi))$.
\end{enumerate}
Consequently, we will denote $\CC_{\uuu \setminus \vvv}(\AAA, M) = \CC_{\uuu \setminus \vvv^{\sharp}}(\AAA, M)$ and call this complex the \emph{Hochschild complex of $\AAA$ with support outside $\vvv$}.
\end{example}

\subsection{Censoring subcategories}\label{parcensor}
Let $\AAA$ be a $\uuu$-graded category. A subcategory $\vvv \subseteq \uuu$ is called \emph{censoring} if for all $u: U \lra U'$, $A \in \AAA_U$, $A' \in \AAA_{U'}$ with $u \notin \vvv$, we have
$\AAA_u(A,A') = 0$. 
The terminology is taken from \cite[\S 4.3]{lowenvandenberghhoch}, cfr. the following example:

\begin{example}\label{extrans}
Let $\AAA$ be a linear category and let $\rrr$ be a transitive relation on $\Ob(\AAA)$. In \cite{lowenvandenberghhoch}, the relation $\rrr$ is called \emph{censoring} if $\AAA(B, A) = 0$ for $(B,A) \notin \rrr$. Clearly, this yields a special case of a censoring subcategory $\uuu_{\rrr}$ of the standard grading category $\uuu$ of $\AAA$. Precisely, we let $\uuu_{\rrr}$ be the category with $\Ob(\uuu_{\rrr}) = \Ob(\AAA)$ and 
$$\uuu_{\rrr}(B,A) = \begin{cases} \ast & \text{if} \,\, (B,A) \in \rrr \\ \varnothing & \text{otherwise.} \end{cases}$$
\end{example}
A $1$-injection $(\varphi, F) : (\vvv, \YYY) \lra (\uuu, \AAA)$ of graded sets is called \emph{censoring} if for all $u: A \lra A'$ in $ \uuu^{\sharp}$ which is not in the image of $\nnn_1(F): \nnn_1(\YYY) \lra \nnn_1(\AAA)$, we have $\AAA_u(A,A') = 0$. 

\begin{remark}
\begin{enumerate}
\item A subcategory $\varphi: \vvv \subseteq \uuu$ is censoring if and only if the cartesian map $(\varphi, \delta^{\varphi, \AAA}): (\vvv, \AAA^{\varphi}) \lra (\uuu, \AAA)$ is a censoring $1$-injection.
\item A $1$-injection $(\varphi, F) : (\vvv, \YYY) \lra (\uuu, \AAA)$ is censoring if and only if $\varphi^{\sharp}: \vvv^{\sharp} \lra \uuu^{\sharp}$ defines a censoring subcategory for the graded category $(\uuu^{\sharp}, \AAA^{\sharp})$.
\end{enumerate}
\end{remark}

\begin{proposition}\label{propcens}
Let $\AAA$ be a $\uuu$-graded category and let $M$ be an $\AAA$-bimodule.
\begin{enumerate}
\item Let $\varphi: \vvv \subseteq \uuu$ be a censoring subcategory. Then the map
$$\varphi^{\ast}: \CC_{\uuu}(\AAA, M) \lra \CC_{\vvv}(\AAA^{\varphi}, M^{\varphi})$$
is an isomorphism of complexes.
\item Let $(\varphi, F): (\vvv, \BBB) \lra (\uuu, \AAA)$ be a subcartesian, censoring $1$-injection. Then the map
$$F^{\ast}: \CC_{\uuu}(\AAA, M) \lra \CC_{\vvv}(\BBB, F^{\ast}M)$$ is an isomorphism of complexes. 
\end{enumerate}
\end{proposition}

\begin{proof}
By the higher remarks, it suffices to prove either one of (1), (2). Let us prove (2). 
By the assumption, every component 
$$\Hom_k(\AAA_{u_{n-1}}(A_{n-1}, A_n) \otimes \dots \otimes \AAA_{u_0}(A_0, A_1), M_{|u|}(A_0, A_n))$$
in \eqref{eqsupport} contains some $u_i: A_i \lra A_{i+1}$ not contained in $\Beeld(\nnn_1(\varphi^{\sharp}))$, which thus has $\AAA_{u_i}(A_i, A_{i+1}) = 0$. Consequently $\CC_{\uuu \setminus \vvv^{\sharp}}(\AAA, M) = 0$.
\end{proof}

Proposition \ref{propcens} has a useful corollary, which says that a censoring subcategory essentially censors possible uncontrollable parts of bimodules. This was precisely the original intuition behind the terminology in \cite{lowenvandenberghhoch}. 

\begin{corollary}
Let $\AAA$ be a $\uuu$-graded category and let $\vvv \subseteq \uuu$ be a censoring subcategory. Let 
$$f: M \lra N$$
be a morphism of $\AAA$-bimodules. If 
$$f^{\varphi}: M^{\varphi} \lra N^{\varphi}$$
is a quasi-isomorphism, then
so is
$$\CC_{\uuu}(\AAA, f): \CC_{\uuu}(\AAA,M) \lra \CC_{\uuu}(\AAA, N).$$
\end{corollary}

\section{Hochschild cohomology with support}\label{parhochsup}

Let $\AAA$ be a $\uuu$-graded category, $\varphi: \vvv \subseteq \uuu$ a subcategory, and $(\vvv, \BBB) \lra (\uuu, \AAA)$ a cartesian functor.
In this section we investigate some cases where the cohomology of the complex $\CC_{\uuu \setminus \vvv}(\AAA, M)$ (see Example \ref{exuminv}) has a nice cohomological interpretation. The main point is that we need some control over the ``complement'' of $\vvv$ in $\uuu$. Precisely, we assume that this complement (the $\uuu$-morphisms not in $\vvv$) constitutes an ideal $\zzz$ in $\uuu$. In this case we show in Proposition \ref{propidsubcat} that 
$$\CC_{\uuu \setminus \vvv}(\AAA, M) \cong \CC_{\uuu}(\AAA, M_{\zzz})$$
where $M_{\zzz}$ is the natural restriction of $M$ to an $\AAA$-bimodule supported on $\zzz$ (i.e. with zero values outside of $\zzz$).
Our setup applies in the situation where $\uuu$ is the category associated to a collection of open subsets of a topological space $X$ ordered by inclusion, $\vvv$ is the full subcategory of subsets $U \subseteq V$ for a fixed subset $V$, and $\zzz$ contains the inclusions $U' \subseteq U$ with $U \nsubseteq V$.  

In \S \ref{pararrow}, we revisit the arrow category construction from \cite{kellerdih} in the map-graded context. For an $(\uuu, \AAA)$-$(\vvv, \BBB)$-bimodule $(S, M)$, we take the natural inclusion $\vvv \coprod \uuu \lra (\vvv \rightarrow_S \uuu)$ and corresponding cartesian functor 
$$(\vvv \coprod \uuu, \BBB \coprod \AAA) \lra (\vvv \rightarrow_S \uuu, \BBB \rightarrow_M \AAA)$$
as starting point for obtaining map-graded analogues of some of the main results from \cite{kellerdih}. Sections \S \ref{parconn} and \S \ref{parderived} are entirely modelled upon the treatment in \cite{kellerdih}, and mainly formulate results from \cite{kellerdih} in the map-graded context, making use of the natural Hom and tensor functors from \S \ref{parbimod}. Further, in \S \ref{pararrowthin}, we give an intrinsic characterization of arrow categories based upon the \emph{thin ideals} introduced in \S \ref{parthin}. 

\subsection{Ideals in categories}
Let $\uuu$ be an arbitrary category. Recall from \S \ref{parbifun} that a $\uuu$-bifunctor $S$ consists of sets $S(V,U)$ for $U, V \in \uuu$ and actions
$$\uuu(U,U') \times S(V,U) \times \uuu(V',V) \lra S(V', U').$$
Denote the category of $\uuu$-bimodules by $\Bimod(\uuu)$.
A \emph{(two sided) ideal} in $\uuu$ is a subfunctor
$$\zzz \subseteq 1_{\uuu}$$
of the identity bifunctor in the category $\mathsf{Bifun}(\uuu)$ of $\uuu$-bifunctors. More precisely, it consists of subsets $\zzz(V,U) \subseteq \uuu(V,U)$ for all $U, V \in \uuu$ such that the composition of $\uuu$ restricts to
$$\uuu(U,U') \times \zzz(V,U) \times \uuu(V',V) \lra \zzz(V', U').$$
We put the \emph{morphisms in $\zzz$} equal to $\Mor(\zzz) = \coprod_{V, U} \zzz(V,U)$.

\subsection{Bimodules over ideals}\label{parbimideal}
Let $\AAA$ be a $\uuu$-graded category and let $\zzz$ be an ideal in $\uuu$. An \emph{$\AAA$-bimodule $M$ on $\zzz$} consists of:
\begin{itemize}
\item $k$-modules $M_z(B, A)$ for all $z: V \lra U$ in $\zzz$ and $A \in \AAA_U$, $B \in \AAA_{V}$,
\item actions $\AAA_u(A,A') \otimes M_z(B, A) \otimes \AAA_v(B',B) \lra M_{uzv}(B',A')$ for all additional $u: U \lra U'$, $v: V' \lra V$ in $\uuu$ and $A' \in \AAA_{U'}$, $B' \in \AAA_{V'}$
\end{itemize}
with the natural axioms. The $\AAA$-bimodules on $\zzz$ form an abelian category $\Bimod_{\zzz}(\AAA)$. There are obvious exact funcors
$$(-)|^{\uuu}: \Bimod_{\zzz}(\AAA) \lra \Bimod_{\uuu}(\AAA): M \longmapsto M|^{\uuu}$$
where $M|^{\uuu}$ is the extension of $M$ by zero values outside of $\zzz$, and
$$(-)|_{\zzz}: \Bimod_{\uuu}(\AAA) \lra \Bimod_{\zzz}(\AAA): M \longmapsto M|_{\zzz}$$
where $M|_{\zzz}$ is the restriction of $M$ to $\zzz$. Clearly, $(-)|_{\zzz}$ is right adjoint to $(-)|^{\uuu}$, $$(-)|_{\zzz}(-)|^{\uuu} = 1_{\Bimod_{\zzz}(\AAA)}$$
and the canonical 
$$(-)_{\zzz} = (-)|^{\uuu}(-)|_{\zzz} \lra 1_{\Bimod_{\uuu}(\AAA)}$$
corresponds to inclusions of $\AAA$-modules
$$M_{\zzz} \subseteq M$$
where $M_{\zzz}$ is obtained by changing the values of $M$ to zero outside of $\zzz$. 
We can thus identify $\Bimod_{\zzz}(\AAA)$ with the full subcategory of $\Bimod_{\uuu}(\AAA)$ of bimodules \emph{supported on $\zzz$}.

\subsection{Ideal-subcategory decomposition}
We now turn to our main situation of interest, which is summarized as follows:
\begin{itemize}
\item $\uuu$ is an arbitrary category;
\item $\vvv \subseteq \uuu$ is a subcategory;
\item $\zzz \subseteq 1_{\uuu}$ is an ideal;
\item $\Mor(\uuu) = \Mor(\vvv) \coprod \Mor(\zzz)$.
\end{itemize}
In this case, we call $(\zzz, \vvv)$ an \emph{ideal-subcategory decomposition} of $\uuu$.

\begin{example}
Let $X$ be a topological space, $V$ an open subset and $Z = X \setminus V$ its closed complement. Let $\uuu \subseteq \mathrm{open}(X)$ be a sub-poset of the open sets of $X$. Let $\vvv \subseteq \uuu$ be the full subcategory with
$$\Ob(\vvv) = \{ U \in \uuu \,\, |\,\, U \subseteq V \}$$
and $\zzz \subseteq 1_{\uuu}$ the ideal with $u: U \lra U'$ in $\zzz$ if and only if $U' \nsubseteq V$.
Clearly, $(\zzz, \vvv)$ is an ideal-subcategory decomposition of $\uuu$.
\end{example}


\subsection{The localization sequence of bimodule categories}\label{parlocbimod} 
Let $\AAA$ be a $\uuu$-graded category and $(\zzz, \vvv)$ an ideal-subcategory decomposition of $\uuu$. Put $\BBB = \AAA|_{\vvv}$.
In this section we take a closer look at the sequence
$$\xymatrix{{\Bimod_{\zzz}(\AAA)} \ar[r]_-{(-)|^{\uuu}_{(\zzz)}} & {\Bimod_{\uuu}(\AAA)} \ar[r]_{(-)|_{\vvv}} & {\Bimod_{\vvv}(\BBB).}}$$ 
Clearly $\Bimod_{\zzz}(\AAA)$ is the kernel of $(-)|_{\vvv}$, and we already know from \S \ref{parbimideal} that $(-)|^{\uuu}_{(\zzz)}$ has a right adjoint $(-)|_{\zzz}$. Since $(-)|_{\vvv}$ is induced by an underlying $k$-linear functor
$\BBB \otimes_{\vvv} \BBB \lra \AAA \otimes_{\uuu} \AAA$, both adjoints of $(-)|_{\vvv}$ exist and can be described explicitely. Moreover, since $\zzz$ is an ideal, the right adjoint has a particularly easy description. Define the functor
$$(-)|_{(\vvv)}^{\uuu}: \Bimod_{\vvv}(\BBB) \lra \Bimod_{\uuu}(\AAA): M \longmapsto M|^{\uuu}$$
in the following way. Put
$$(M|^{\uuu})_u(B,A) = \begin{cases} M_u(B,A) &\text{if}\,\, u \in \vvv \\ 0 & \text{otherwise}\end{cases}.$$
Then there is a unique way to let $\AAA$ act on $M|^{\uuu}$ so that 
$$\AAA_{u'}(A,A') \otimes (M|^{\uuu})_u(B,A) \otimes \AAA_{u''}(B',B) \lra (M|^{\uuu})_{u'uu''}(B',A')$$
is given by the action of $\BBB$ on $M$ if $u$, $u'$, $u''$ are in $\vvv$, and is zero otherwise. This is a well defined action thanks to the fact that $\zzz$ is an ideal. The functor $(-)|_{(\vvv)}^{\uuu}$ is right adjoint to $(-)|_{\vvv}$ and we have
$$(-)|_{\vvv}(-)|_{(\vvv)}^{\uuu} = 1_{\Bimod_{\vvv}(\BBB)}.$$
The canonical
$$1_{\Bimod_{\uuu}(\AAA)} \lra (-)|_{(\vvv)}^{\uuu}(-)|_{\vvv} = (-)_{\vvv}$$
corresponds to quotients of $\AAA$-bimodules
$$M \lra M_{\vvv}$$
where $M_{\vvv}$ is obtained from $M$ by changing the values of $M$ outside $\vvv$ to zero.

Summarizing, we have an exact sequence of functors
$$0 \lra (-)_{\zzz} \lra 1_{\Bimod_{\uuu}(\AAA)} \lra (-)_{\vvv} \lra 0$$
corresponding to exact sequences of bimodules
\begin{equation} \label{eqbimod}
0 \lra M_{\zzz} \lra M \lra M_{\vvv} \lra 0.
\end{equation}

\subsection{The localization sequence for Hochschild cohomology}\label{parlochoch}
We can now use \eqref{eqbimod} to obtain the following short exact sequence of Hochschild complexes:
\begin{equation}\label{eqhoch2}
0 \lra \CC_{\uuu}(\AAA, M_{\zzz}) \lra \CC_{\uuu}(\AAA, M) \lra \CC_{\uuu}(\AAA, M_{\vvv}) \lra 0.
\end{equation}

\begin{proposition}\label{propidsubcat}
The sequences \eqref{eqkern} and \eqref{eqhoch2} are canonically isomorphic.
\end{proposition}

\begin{proof}
This follows from the fact that $(\zzz, \vvv)$ is an ideal-subcategory decomposition of $\uuu$. Indeed, let us first compare $\CC_{\uuu}(\AAA, M_{\vvv})$ to $\CC_{\vvv}(\BBB, M|_{\vvv})$. In every product
$$\prod_{(u, A) \in \nnn(\AAA)} \Hom_k(\AAA_{u_{n-1}}(A_{n-1}, A_n) \otimes \dots \otimes \AAA_{u_0}(A_0, A_1), (M_{\vvv})_{|u|}(A_0, A_n))$$
turning up in $\CC^n_{\uuu}(\AAA, M_{\vvv})$ we may clearly remove the pieces with $|u| \in \zzz$ since then $(M_{\vvv})_{|u|}(A_0, A_n) = 0$. But since $\zzz$ is an ideal and $\vvv$ a subcategory, the remaining pieces are precisely the ones with all the $u_i \in \vvv$, and we recover $\CC_{\vvv}(\BBB, M|_{\vvv})$. Similarly, consider a product
$$\prod_{(u, A) \in \nnn(\AAA)} \Hom_k(\AAA_{u_{n-1}}(A_{n-1}, A_n) \otimes \dots \otimes \AAA_{u_0}(A_0, A_1), (M_{\zzz})_{|u|}(A_0, A_n))$$
in $\CC^n_{\uuu}(\AAA, M_{\zzz})$.
This time, all the pieces with $|u| \in \vvv$ can be removed. What remains are the pieces in which at least one $u_i$ belongs to $\zzz$, which corresponds precisely to $\CC_{\uuu \setminus \vvv}(\AAA, M)$.
\end{proof}

In fact, the isomorphism $\CC_{\uuu}(\AAA, M_{\vvv}) \cong \CC_{\vvv}(\BBB, M|_{\vvv})$ is easily understood on the derived level. Indeed, the localization between $\Bimod_{\uuu}(\AAA)$ and $\Bimod_{\vvv}(\BBB)$ of \S \ref{parlocbimod} yields:
$$\RHom_{\Bimod_{\uuu}(\AAA)}(1_{\AAA}, (M|_{\vvv})|^{\uuu}) \cong \RHom_{\Bimod_{\vvv}(\BBB)}(1_{\BBB}, M|_{\vvv}).$$
The more mysterious part in the isomorphic sequences \eqref{eqkern} and \eqref{eqhoch2} remains the Hochschild complex with support
$$\CC_{\uuu \setminus \vvv}(\AAA, M) \cong \CC_{\uuu}(\AAA, M_{\zzz}),$$
but at least it now has an interpretation as an ordinary Hochschild complex of $\AAA$ with values in the bimodule $M_{\zzz}$ supported on $\zzz$. 
So far, the most meaningful incarnation of the sequences \eqref{eqkern} and \eqref{eqhoch2} is perhaps
\begin{equation}\label{eqhoch3}
0 \lra \CC_{\uuu}(\AAA, M_{\zzz}) \lra \CC_{\uuu}(\AAA, M) \lra \CC_{\vvv}(\BBB, M|_{\vvv}) \lra 0.
\end{equation}

\subsection{Thin ideals}\label{parthin}
Consider $\varphi: \vvv \subseteq \uuu$ with $\Ob(\vvv) = \Ob(\uuu)$ and let $(\zzz, \vvv)$ be an ideal-subcategory decomposition of $\uuu$. Let $\BBB$ be a $\vvv$-graded category. We define a \emph{$\BBB$-bimodule on $\zzz$} to consist of the following data:

\begin{itemize}
\item for all $z: V \lra U$ in $\zzz$ and $B \in \BBB_V$, $A \in \BBB_U$, a $k$-module $M_z(B,A)$;
\item actions $\BBB_v(A, A') \otimes M_z(B,A) \otimes \BBB_{v'}(B' B) \lra M_{vzv'}(B',A')$ for all additional $v: U \lra U'$, $v': V' \lra V$ in $\vvv$ and $A' \in \BBB_{U'}$, $B' \in \BBB_{V'}$.
\end{itemize}
We thus obtain the abelian category $\Bimod_{\zzz}(\BBB)$ of \emph{$\BBB$-bimodules on $\zzz$}.

Now consider a $\uuu$-graded category $\AAA$ and suppose $\BBB = \AAA|_{\vvv}$.
There is a functor
$$(-)|^{\AAA}: \Bimod_{\zzz}(\BBB) \lra \Bimod_{\zzz}(\AAA): M \longmapsto M|^{\AAA}$$
where the actions $\AAA_{z'}(A,A') \otimes M_z(B,A) \lra M_{z'z}(B,A')$ and $M_z(B,A) \otimes \AAA_{z''}(B',B) \lra M_{zz''}(B',A)$ are zero for $z$, $z'$ and $z''$ in $\zzz$. This yields a well defined action since $\zzz$ is an ideal.
Obviously, there is also a functor
$$(-)|_{\BBB}: \Bimod_{\zzz}(\AAA) \lra \Bimod_{\zzz}(\BBB): M \longmapsto M|_{\BBB}$$ which restricts the action to $\BBB$. Clearly,
$$(-)|_{\BBB}(-)|^{\AAA} = 1_{\Bimod_{\zzz}(\BBB)}.$$

\begin{definition}
\begin{enumerate}
\item An ideal $\zzz \subseteq 1_{\uuu}$ is called \emph{thin} if it contains no consecutive morphisms, i.e. if we consider $z \in \zzz(V,U)$, $u' \in \uuu(V',V)$, $u \in \uuu(U,U')$, then it follows that $u' \notin \zzz$ and $u \notin \zzz$. 
\item An $\AAA$-bimodule $M$ on $\zzz$ is called \emph{$\zzz$-thin} if it receives no non trivial actions from $\zzz$, i.e if we have $z \in \zzz(V,U)$, $z' \in \zzz(V',V)$, $z'' \in \zzz(U,U')$, then it follows that both $\AAA_{z''}(A, A') \otimes M_z(B,A) \lra M_{z''z}(B,A')$ and $M_z(B,A) \otimes \AAA_{z'}(B',B) \lra M_{zz'}(B',A)$ are zero for all $A$, $A'$, $B$, $B'$.
\end{enumerate}
\end{definition}

\begin{example}\label{exthin}
Let $\uuu$ be a category and $\Ob_1 \subseteq \Ob(\uuu)$ and $\Ob_2 \subseteq \Ob(\uuu)$ be two classes of objects with no morphisms going from $\Ob_2$ to $\Ob_1$. Then the $\uuu$-morphisms starting in $\Ob_1$ and landing in $\Ob_2$ form a thin ideal in $\uuu$.
\end{example}

Obviously, if $\zzz$ is thin, then every $M$-bimodule on $\zzz$ is $\zzz$-thin. In general, we make the following easy observations:
\begin{proposition}
For an $\AAA$-bimodule $M$ on $\zzz$, the following are equivalent:
\begin{enumerate}
\item $M$ is $\zzz$-thin;
\item $M \cong (M|_{\BBB})|^{\AAA}$;
\item $M$ is in the image of $(-)|^{\AAA}$.
\end{enumerate}
\end{proposition}

\begin{corollary}\label{isobimod}
If $\zzz$ is thin, then $(-)|^{\AAA}$ and $(-)|_{\BBB}$ constitute inverse isomorphisms
$$\Bimod_{\zzz}(\BBB) \cong \Bimod_{\zzz}(\AAA).$$
\end{corollary}

\subsection{Arrow categories}\label{pararrow}
In this section, we introduce the arrow category construction from \cite{kellerdih} in the map-graded setting.
Let $\AAA$ be a $\uuu$-graded category, $\BBB$ a $\vvv$-graded category, $S$ a $\uuu$-$\vvv$-bifunctor and $M$ and $\AAA$-$S$-$\BBB$-bimodule. To these data we associate the \emph{arrow category} $\BBB \ra_M \AAA$ which is a $\vvv \ra_S \uuu$ graded category. Here, $\www = \vvv \ra_S \uuu$ is the underlying arrow category with
$$\Ob(\www) = \Ob(\vvv) \coprod \Ob(\uuu)$$
and 
$$\begin{aligned}
\www(V', V) = \vvv(V',V) &&& \www(U, U') = \uuu(U,U')\\
\www(V,U) = S(V,U) &&& \www(U,V) = \varnothing
\end{aligned}$$
for $U, U' \in \uuu$ and $V, V' \in \vvv$.
Similarly, $\CCC = \BBB \ra_M \AAA$ is the $\www$-graded category with
$$\Ob(\CCC) = \Ob(\AAA) \coprod \Ob(\BBB)$$
and
$$\begin{aligned}
\CCC_v(B', B) = \BBB_v(B',B) &&& \CCC_u(A, A') = \AAA_u(A,A')\\
\CCC_s(B,A) = M_s(B,A) &&& 
\end{aligned}$$
for $v: V' \lra V$ in $\vvv$, $B' \in \BBB_{V'}$, $B \in \BBB_V$, $u: U \lra U'$ in $\uuu$, $A \in \AAA_U$, $A' \in \AAA_{U'}$.

\begin{remark}
Even when $\AAA$ and $\BBB$ are linear categories with standard grading, and $M$ is an ordinary $\AAA$-$\BBB$-bimodule, the resulting arrow category $\BBB \ra_M \AAA$ is naturally graded in a non-standard way since there are no morphisms going from $\AAA$ to $\BBB$. 
\end{remark}

Consider the natural inclusions
$\varphi_{\uuu}: \uuu \lra \www$, $\varphi_{\vvv}: \vvv \lra \www$ and $\varphi_{\uuu \coprod \vvv}: \uuu \coprod \vvv \lra \www$. Clearly, we have
$\CCC^{\varphi_{\uuu}} = \AAA$, $\CCC^{\varphi_{\vvv}} = \BBB$ and $\CCC^{\varphi_{\uuu \coprod \vvv}} = \AAA \coprod \BBB$ and we obtain the induced surjections
$$\varphi_{\uuu}^{\ast}: \CC_{\www}(\CCC) \lra \CC_{\uuu}(\AAA),$$
$$\varphi_{\vvv}^{\ast}: \CC_{\www}(\CCC) \lra \CC_{\vvv}(\BBB)$$
and
$$\varphi_{\uuu \coprod \vvv}^{\ast} \cong \begin{pmatrix} \varphi_{\vvv}^{\ast} \\ \varphi_{\uuu}^{\ast} \end{pmatrix} : \CC_{\www}(\CCC) \lra \CC_{\uuu \coprod \vvv}(\AAA \coprod \BBB) \cong \CC_{\vvv}(\BBB) \oplus \CC_{\uuu}(\AAA).$$
Further, $S$ defines a thin ideal $\sss$ in $\www$ with
$$\sss(V,U) = \begin{cases} S(V,U) & \text{if}\,\, U \in \uuu, V \in \vvv \\ \varnothing & \text{else} \end{cases}$$
and $(\sss, \uuu \coprod \vvv)$ is an ideal-subcategory decomposition in $\www$.
For the category of $\BBB \coprod \AAA$-bimodules on $\sss$ in the sense of \S \ref{parthin}, we clearly have an isomorphism of categories
$$\Bimod_{\sss}(\BBB \coprod \AAA) \cong \Bimod_S(\BBB, \AAA)$$
and hence, by Corollary \ref{isobimod}, an isomorphism of categories
\begin{equation}\label{isobimod2}
\Bimod_S(\BBB, \AAA) \cong \Bimod_{\sss}(\CCC).
\end{equation}

\subsection{Arrow categories and thin ideals}\label{pararrowthin}
In this section, we characterize the situation that occurs from the arrow category construction.
Let $\CCC$ be a $\www$-graded category and let $\sss$ be a thin ideal in $\www$. We define the full subcategories $\varphi_{\uuu}: \uuu \subseteq \www$ and $\varphi_{\vvv}: \vvv \subseteq \www$ in the following way. An object $W \in \www$ belongs to $\vvv$ if there exists a path
$$W \lra W_1 \lra \dots \lra W_n \lra W_{n+1}$$
for which the last map $W_n \lra W_{n+1}$ belongs to $\sss$. Similarly, $W$ belongs to $\uuu$ if there exists a path
$$W_1 \lra W_2 \lra \dots \lra W$$
for which $W_1 \lra W_2$ belongs to $\sss$.
We put $\AAA = \CCC^{\varphi_{\uuu}}$ and $\BBB = \CCC^{\varphi_{\vvv}}$.
Let $S$ denote the restriction of $\sss$ to a $\uuu$-$\vvv$-bifunctor, and let $M$ be the restriction of $1_{\CCC}$ to an $\AAA$-$\BBB$-bimodule.

\begin{proposition}\label{propthinarrow}
For $U \in \uuu$ and $V \in \vvv$, we have $\www(U, V) = \varnothing$. In particular, the categories $\uuu$ and $\vvv$ are disjoint, and there is an injection $\varphi: (\vvv \rightarrow_{S} \uuu) \lra \www$ for which $$(\BBB \rightarrow_M \AAA) = \CCC^{\varphi}.$$
The following are equivalent:
\begin{enumerate}
\item $\varphi$ is an isomorphism;
\item We have $\Ob(\www) = \Ob(\vvv) \cup \Ob(\uuu)$ and $\sss(V,U) = \www(V,U)$ for $V \in \vvv$ and $U \in \uuu$;
\end{enumerate}
In this case, we have $(\BBB \rightarrow_M \AAA) = \CCC$.
\end{proposition}

\begin{proof}
Suppose there is a morphism $U \lra V$ with $U \in \uuu$ and $V \in \vvv$. Then there is a path $$W_1 \lra W_2 \lra \dots \lra U \lra V \lra \dots \lra W_n \lra W_{n+1}$$ with the two morphisms at the ends belonging to $\sss$. But then since $\sss$ is an ideal, also $W_2 \lra \dots \lra W_{n+1}$ belongs to $\sss$. With $W_1 \lra W_2$ in $\sss$, this contradicts the thinness of $\sss$.

The functor $\varphi$ is clearly injective on objects and on morphisms in $\uuu$ or $\vvv$, and since $\sss(V,U) \subseteq \www(V,U)$, the functor is indeed injective. The description of $\CCC^{\varphi}$ clearly follows.

 It remains to show the equivalence of (1) and (2). Condition (2) is clearly necessary for $\varphi$ to be surjective. Conversely, surjectivity readily follows from (2) taking into account that for $U \in \uuu$ and $V \in \vvv$, we have $\www(U, V) = \varnothing$.
\end{proof}

Recall that a delta is a category in which the arrows go only one way, i.e.

\begin{definition}\label{defdelta}
A \emph{delta} is a category $\www$ such that for all $W \neq W' \in \www$ we have $\www(W,W') = \varnothing$ or $\www(W', W) = \varnothing$. 
\end{definition}

\begin{example}\label{exdelta}
Let $\www$ be a delta with a terminal object $\ast \in \www$. For $W \in \www$, let $\ast_W: W \lra \ast$ denote the unique morphism.  Let $\sss$ consist of all morphsims $\ast_W : W \lra \ast$ for $W \neq \ast$. Since $\www$ is a delta and $\ast$ is terminal, for all $W \neq \ast$ we have $\www(\ast, W) = \varnothing$ and we have $\www(\ast,\ast) = \{1_{\ast} \} = \{ \ast_{\ast} \}$.
Then $\sss$ is a thin ideal by Example \ref{exthin}. In the above notations, $\uuu$ consists of $\ast$ and $\vvv$ consists of all other objects. Then condition (2) in Proposition \ref{propthinarrow} is fulfilled whence $\www \cong (\vvv \rightarrow_S \uuu)$.
\end{example}

\subsection{Connecting homomorphism}\label{parconn}

In the notations of \S \ref{pararrow}, we now investigate the complex
$$\CC_{\www}(\CCC, (1_{\CCC})_{\sss})$$
which fits into the following exact sequence from \eqref{eqhoch3}:
$$\xymatrix{ {0} \ar[r] & {\CC_{\www}(\CCC, (1_{\CCC})_{\sss})} \ar[r] & {\CC_{\www}(\CCC)} \ar[r]_-{\begin{pmatrix} \varphi_{\vvv}^{\ast} \\ \varphi_{\uuu}^{\ast} \end{pmatrix}} & {\CC_{\vvv}(\BBB) \oplus \CC_{\uuu}(\AAA)} \ar[r] & 0.}$$
Since we are interested in identifying when $\varphi_{\vvv}^{\ast}$ and $\varphi_{\uuu}^{\ast}$ are quasi-isomorphisms, we will describe the components of the connecting homomorphism explicitely. In fact, this morphism is determined by two maps on the chain level which we will describe next:
$$\alpha: \CC_{\uuu}(\AAA) \lra \CC_{\www}(\CCC, (1_{\CCC})_{\sss})[1]$$ and
$$\beta: \CC_{\vvv}(\BBB) \lra \CC_{\www}(\CCC, (1_{\CCC})_{\sss})[1].$$
To do this, we first make $\CC_{\www}(\CCC, (1_{\CCC})_{\sss})$ explicit. By definition, $\CC^n_{\www}(\CCC, (1_{\CCC})_{\sss})$ is given by
$$\prod_{(u,A) \in \nnn(\CCC) \exists u_i \in \sss}  \Hom_k( \BBB_{u_{n-1}}(A_{n-1},A_n) \otimes \dots \otimes M_{u_i}(A_i,A_{i+1})  \dots \otimes \AAA_{u_0}(A_0, A_1), M_{|u|}(A_0, A_n)).$$
Let $\phi \in \CC^n_{\uuu}(\AAA)$ be a Hochschild cocycle.
Following \cite{kellerdih} we define $\alpha(\phi)$ with non-zero components in 
$$\Hom_k(M_{u_{n}}(A_n,A_{n+1})  \dots \otimes \AAA_{u_0}(A_0, A_1), M_{|u|}(A_0, A_{n+1})).$$
given by 
$$\alpha(\phi)(x, a_{n-1}, \dots, a_0) = x\phi(a_{n-1}, \dots, a_0),$$
and similarly for $\beta$.

\begin{lemma}
The maps $\alpha$ and $\beta$ are chain maps for which
$$\begin{pmatrix} \beta & \alpha \end{pmatrix}: \CC_{\vvv}(\BBB) \oplus \CC_{\uuu}(\AAA) \lra \CC_{\www}(\CCC, (1_{\CCC})_{\sss})[1]$$
induces the connecting homomorphisms.
\end{lemma}

The triangle 
\begin{equation}\label{triangle}
\xymatrix{ {\CC_{\www}(\CCC)} \ar[r]_-{\begin{pmatrix} \varphi_{\vvv}^{\ast} \\ \varphi_{\uuu}^{\ast} \end{pmatrix}} & {\CC_{\vvv}(\BBB) \oplus \CC_{\uuu}(\AAA)} \ar[r]_-{\begin{pmatrix} \beta & \alpha \end{pmatrix}} & {\CC_{\www}(\CCC, (1_{\CCC})_{\sss})[1]} \ar[r] &}\end{equation}
corresponds to a homotopy bicartesian square
$$\xymatrix{ {\CC_{\www}(\CCC)} \ar[d]_{\varphi_{\vvv}^{\ast}} \ar[r]^{\varphi_{\uuu}^{\ast}} & {\CC_{\uuu}(\AAA)} \ar[d]^{\alpha} \\ {\CC_{\vvv}(\BBB)} \ar[r]_-{\beta} & {\CC_{\www}(\CCC, (1_{\CCC})_{\sss})[1].}}$$

We thus have:

\begin{proposition}
\begin{enumerate}
\item If $\alpha$ is a quasi-isomorphism, then so is $\varphi^{\ast}_{\vvv}$.
\item If $\beta$ is a quasi-isomorphism, then so is $\varphi^{\ast}_{\uuu}$.
\end{enumerate}
\end{proposition}

\subsection{Derived interpretation}\label{parderived}

Following \cite{kellerdih}, we now give a derived interpretation of the morphisms $\alpha$ and $\beta$. 
We start by giving various derived interpretations of $\CC_{\www}(\CCC, (1_{\CCC})_S$.

Let $B(\AAA) \lra 1_{\AAA}$ and $B(\BBB) \lra 1_{\BBB}$ be the Bar resolutions in $\Bimod_{\uuu}(\AAA)$ and $\Bimod_{\vvv}(\BBB)$ respectively, as defined in \cite[\S 3.2]{lowenmap}.
We have 
$$\CC_{\uuu}(\AAA) \cong \Hom_{\Bimod_{\uuu}(\AAA)}(B(\AAA), 1_{\AAA}) = \RHom_{\Bimod_{\uuu}(\AAA)}(1_{\AAA}, 1_{\AAA})$$ and 
$$\CC_{\vvv}(\BBB) \cong \Hom_{\Bimod_{\vvv}(\BBB)}(B(\BBB), 1_{\BBB}) = \RHom_{\Bimod_{\vvv}(\BBB)}(1_{\BBB}, 1_{\BBB}).$$
Similarly, we have
$$\CC_{\www}(\CCC, (1_{\CCC})_S)[1] \cong \Hom_{\Bimod_S(\AAA, \BBB)}(B(\AAA) \otimes_{\AAA} M \otimes_{\BBB} B(\BBB), M)$$
and $$\Hom_{\Bimod_S(\AAA, \BBB)}(B(\AAA) \otimes_{\AAA} M \otimes_{\BBB} B(\BBB), M) = \RHom_{\Bimod_{\sss}(\AAA, \BBB)}(M,M).$$
In particular, from the triangle \eqref{triangle} we obtain a long exact cohomology sequence
$$\dots \lra HH_{\www}^i(\CCC) \lra HH_{\vvv}^i(\BBB) \oplus HH_{\uuu}^i(\AAA) \lra \Ext^{i}_{\Bimod_S(\BBB, \AAA)}(M,M) \lra \dots$$

By Proposition \ref{propadj}, we further have
$$\CC_{\www}(\CCC, (1_{\CCC})_S)[1] \cong \Hom_{\Bimod_{\uuu}(\AAA)}(B(\AAA), \Hom_{\BBB}(M \otimes_{\BBB} B(\BBB), M))$$
with $\Hom_{\BBB}(M \otimes_{\BBB} B(\BBB), M) = \RHom_{\BBB}(M,M).$
Similarly, 
$$\CC_{\www}(\CCC, (1_{\CCC})_S)[1] \cong \Hom_{\Bimod_{\vvv}(\BBB)}(B(\BBB), \Hom_{\AAA^{\op}}(B(\AAA) \otimes_{\AAA} M, M))$$
with $\Hom_{\AAA^{\op}}(B(\AAA) \otimes_{\AAA} M, M) = \RHom_{\AAA^{\op}}(M,M)$.

Consider the natural morphisms 
$$\lambda: 1_{\AAA} \lra \Hom_{\BBB}(M,M) \lra \RHom_{\BBB}(M,M)$$
and 
$$\omega: 1_{\BBB} \lra \Hom_{\AAA^{\op}}(M,M) \lra \RHom_{\AAA^{\op}}(M,M)$$
of bimodules and the induced 
$$\CC_{\uuu}(\AAA, \lambda): \CC_{\uuu}(\AAA) \lra \CC_{\uuu}(\AAA, \RHom_{\BBB}(M,M))$$
and
$$\CC_{\vvv}(\BBB, \omega): \CC_{\vvv}(\BBB) \lra \CC_{\vvv}(\BBB, \RHom_{\AAA}(M,M)).$$

We have the following map-graded version Theorem \cite[Theorem 4.1.1]{lowenvandenberghhoch} extracted from \cite[\S 4]{kellerdih}:

\begin{theorem}\label{thmarrowqis}
\begin{enumerate}
\item We have $\alpha \cong \CC_{\uuu}(\AAA, \lambda)$ and $\beta \cong \CC_{\vvv}(\BBB, \omega)$.
\item If $\CC_{\uuu}(\AAA, \lambda)$ is a quasi-isomorphism (hence, in particular, if $\lambda$ is a quasi-isomorphism), then so is $\varphi_{\vvv}^{\ast}$.
\item If $\CC_{\vvv}(\BBB, \omega)$  is a quasi-isomorphism (hence, in particular, if $\omega$ is a quasi-isomorphism), then so is $\varphi_{\uuu}^{\ast}$.
\end{enumerate}
\end{theorem}

The analogue of \cite[\S 4.6, Theorem]{kellerdih} in the map-graded context can be formulated and proven in a similar fashion. In particular, we mention the following compatibility results explicitly.

Let $(\varphi, F): (\uuu, \AAA) \lra (\vvv, \BBB)$ be a subcartesian graded functor with associated $\uuu$-$\vvv$-bifunctor $S^{\varphi}$ and $\AAA$-$S^{\varphi}$-$\BBB$-bimodule $M^F$ as in Example \ref{exbimod}.
On the one hand, since $(\varphi, F)$ is subcartesian, we have the induced morphism of $B_{\infty}$-algebras
$$F^{\ast}: \CC_{\vvv}(\BBB) \lra \CC_{\uuu}(\AAA)$$
from Proposition \ref{propfunct}.

\begin{proposition}\label{propcompat}
\begin{enumerate}
\item $\lambda: 1_{\AAA} \lra \RHom_{\BBB}(M^F, M^F)$ and hence $\varphi^{\ast}_{\vvv}$  is a quasi-isomorphism.
\item In the homotopy category of $B_{\infty}$-algebras, we have $F^{\ast} = \varphi^{\ast}_{\uuu} {\varphi^{\ast}_{\vvv}}^{-1}$.
\end{enumerate}
\end{proposition}

Let $(\varphi,F): (\vvv, \BBB) \lra (\uuu, \AAA)$ be a subcartesian graded functor with associated $\uuu$-$\vvv$-bifunctor $S_{\varphi}$ and $\AAA$-$S_{\varphi}$-$\BBB$-bimodule $M_F$ as in Example \ref{exbimod}. This time we have the functor
$$F^{\ast}: \CC_{\uuu}(\AAA) \lra \CC_{\vvv}(\BBB).$$

\begin{proposition}\label{propcompat2}
\begin{enumerate}
\item $\omega: 1_{\BBB} \lra \RHom_{\AAA^{\op}}(M_F, M_F)$ and hence $\varphi^{\ast}_{\uuu}$ is a quasi-isomorphism.
\item In the homotopy category of $B_{\infty}$-algebras, we have $F^{\ast} = \varphi^{\ast}_{\vvv} {\varphi^{\ast}_{\uuu}}^{-1}$.
\end{enumerate}
\end{proposition}

\section{Grothendieck construction}\label{pargroth}

In this section, we present a unified framework for constructing map-graded categories and deconstructing their Hochschild complexes. The classical Grothendieck construction from \cite{SGA1} takes a pseudofunctor $\uuu \lra \Cat$ as input and turns it into a category fibered over $\uuu$. The main construction from \cite{lowenmap} is a $k$-linearized version of this construction. Now, we go a step further and start from a pseudofunctor
$$(\uuu, \AAA): \ccc \lra \underline{\Map}: C \longmapsto (\uuu_C, \AAA_C)$$
where $\ccc$ is a small category and $\underline{\Map}$ is the bicategory of map-graded categories and bimodules described in \S \ref{parbimod}. Allowing arbitrary bimodules rather than functors between map-graded categories allows us to capture the arrow category with respect to a bimodule from \S \ref{pararrow}.
In general, we can now deconstruct the Hochschild complex of the Grothendieck construction $(\tilde{\uuu}, \tilde{\AAA})$ of $(\uuu, \AAA)$ based upon the internal structure of $\ccc$. Here, the strategy is to cover $(\tilde{\uuu}, \tilde{\AAA})$ by other Grothendieck constructions, using base change for pseudofunctors from \S \ref{parbase}. For instance, in \S \ref{pargenarrow}, we consider ``generalized arrow categories'', and deconstruct them using iterated arrow category constructions.
In \S \ref{parcoverarrow}, we observe how, in the case where $\ccc$ is a poset, the sheaf property for Hochschild complexes on the one hand, and the arrow category construction on the other hand, can be seen as complementary tools for deconstructing Hochschild cohomology. 
In the final section \S \ref{parcomp}, we start from a pseudofunctor $(\uuu^{\star}, \AAA^{\star}): \ccc^{\ast} \lra \Map_{sc}$, which we compare to the natural pseudofunctor of Grothendieck constructions
$$(\uuu^{\ast}, \AAA^{\ast}): \ccc^{\ast} \lra \Map_{sc}: C \lra (\tilde{\uuu}|_{C}, \tilde{\AAA}|_{C})$$
built from the restriction $(\uuu, \AAA)$ of $(\uuu^{\star}, \AAA^{\star})$ to $\ccc$.
Our main Theorem \ref{thmmaincomp} is heavily based upon Keller's arrow category argument in the case of a fully faithful functor $\BBB \lra \AAA$, which is in fact a special case of our theorem. As an application, we recover the Mayer-Vietoris triangles for ringed spaces from \cite[\S 7.9]{lowenvandenberghhoch}.

\subsection{Diagrams in the bicategory of categories}\label{pargroth1}
Let $\underline{\Cat}$ be the bicategory of categories, bifunctors and natural transformations described in \S \ref{parbifun}. Let $\ccc$ be an arbitrary small category and consider a pseudofunctor
$$\uuu: \ccc \lra \underline{\Cat}: C \longmapsto \uuu_C.$$
The pseudofunctor maps a map $c: C \lra C'$ to a $\uuu_{C'}$-$\uuu_{C}$-bifunctor $\uuu(c) = S_c$ and for an additional $c': C' \lra C''$, there is an isomorphism
$$\phi_{c', c}: S_{c'} \circ S_c \lra S_{c'c}.$$
For a third map $c'': C'' \lra C'''$, we have a commutative diagram
\begin{equation}\label{coh}
\xymatrix{ {(S_{c''} \circ S_{c'}) \circ S_c} \ar[d]_{\alpha} \ar[r]^-{\phi_{c'. c} \circ 1} & {S_{c'' c'} \circ S_c} \ar[r]^-{\phi_{c'' c', c}} & {S_{c'' c' c}} \ar[d]^1 \\
{S_{c''} \circ (S_{c'} \circ S_c)} \ar[r]_-{1 \circ \phi_{c', c}} & {S_{c''} \circ S_{c' c}} \ar[r]_-{\phi_{c'', c' c}} & {S_{c'' c' c}} }
\end{equation}
where $\alpha$ is the isomorphism from the bicategory $\underline{\Cat}$. We suppose moreover that $S_{1_{C}} = 1_{\uuu_C}$.

We define the \emph{Grothendieck construction} of $\uuu$ to be the non-linear $\ccc$-graded category $\uuu$ with $\uuu_C$ as prescribed and, for $c: C \lra C'$, $U \in \uuu_C$, $U' \in \uuu_{C'}$:
$$\uuu_c(U,U') = S_c(U, U').$$
Recall that $$S_{c'} \circ S_c(U,U') = \coprod_{V \in \uuu_{C'}} S_{c'}(V,U'') \times S_{c}(U,V)/ \sim.$$
The composition on $\uuu$ is defined as the natural map
$$\xymatrix{ {S_{c'}(U', U'') \times S_{c}(U, U')} \ar[r] & {S_{c'} \circ S_{c}(U, U'')} \ar[r]_-{\phi_{c', c}} & {S_{c'c}(U, U'')}}$$
and we denote the image of $(s', s)$ under composition by $s's$.
There is a corresponding category $\tilde{\uuu}$ over $\ccc$ with $\Ob(\tilde{\uuu}) = \coprod_{C \in \ccc}\uuu_C$ (see Remark \ref{remgraded}).

\subsection{Diagrams in the bicategory of map-graded categories}\label{pargroth2}
The construction from \S \ref{pargroth1} can be extended in the following way. Let $\underline{\Map}$ be the bicategory of map-graded categories, bimodules and their morphisms described in \S \ref{parbimod}. Let $\ccc$ be an arbitrary small category and consider a pseudofunctor
$$(\uuu, \AAA): \ccc \lra \underline{\Map}: C \longmapsto (\uuu_C, \AAA_C).$$
The pseudofunctor maps $c: C \lra C'$ to an $\AAA_{C'}$-$\AAA_{C}$-bimodule $(\uuu(c) = S_c, \AAA(c) = M_c)$ and for an additional $c': C' \lra C''$ there is and isomorphism
$$(\phi_{c', c}, F_{c',c}): (S_{c'} \circ S_{c}, M_{c'} \otimes_{\AAA_{C'}} M_c) \lra (S_{c' c}, M_{c' c}).$$ 
These isomorphisms satisfy the natural coherence axiom similar to \eqref{coh} and $(S_{1_C}, M_{1_C}) = (1_{\uuu_C}, 1_{\AAA_C})$.

Clearly, there is an underlying pseudofunctor
$$\uuu: \ccc \lra \underline{\Cat}: C \longmapsto \uuu_C$$
which determines a non-linear $\ccc$-graded category $\uuu$ and associated category $\tilde{\uuu}$. We define the \emph{Grothendieck construction} of $(\uuu, \AAA)$ to be the following $\tilde{\uuu}$-graded category $\tilde{\AAA}$. For $U \in \uuu_C$, we define $$\tilde{\AAA}_U = (\AAA_C)_U.$$ For a morphism $(c: C \lra C', s \in S_c(U,U'))$ in $\tilde{\uuu}$, with $U \in \uuu_C$, $U' \in \uuu_{C'}$, and for $A \in (\AAA_C)_U$, $A' \in (\AAA_{C'})_{U'}$, we put
$$\tilde{\AAA}_{(c,s)}(A, A') = (M_c)_s(A,A').$$
Recall that $$(M_{c'} \otimes_{\AAA_{C'}} M_c)_r(A,A'') = \oplus_{r = [(s',s)], B} (M_{c'})_{s'}(B,A'') \otimes (M_c)_s(A,B)/\sim.$$
The composition on $\tilde{\AAA}$ is defined as the composition of the natural map
$$(M_{c'})_{s'}(A', A'') \otimes (M_c)_s(A,A') \lra (M_{c'} \otimes_{\AAA_{C'}} M_c)_{[(s', s)]}(A,A'')$$
followed by the map
$$\xymatrix{ {(M_{c'} \otimes_{\AAA_{C'}} M_c)_{[(s', s)]}(A,A'')} \ar[r]_-{F_{c', c}} & {(M_{c' c})_{s' s}(A,A'').} }$$

\subsection{Diagrams in the bicategory of linear categories}\label{pardiaglin}
In many applications, we are in a somewhat simplified situation from \S \ref{pargroth2}. Let $\underline{\Cat}(k)$ be the bicategory of $k$-linear functors and bimodules. There is a natural map $\underline{\Cat}(k) \lra \underline{\Map}: \AAA \longmapsto \AAA_{tr}$ where $\AAA_{tr}$ is the trivially graded category over $e$ from Example \ref{exstandgr} (2).
Consider a pseudofunctor
$$\AAA: \ccc \lra \underline{\Cat}(k): C \longmapsto \AAA_C.$$
If we consider the corresponding $(\eee, \AAA): \ccc \lra \underline{\Map}: C \longmapsto (e, \AAA_C)$ we see that there is a canonical isomorphism $\tilde{\eee} \cong \ccc$ and carrying out the construction from \S \ref{pargroth2}, we thus obtain a $\ccc$-graded category $\tilde{\AAA}$ with
$$\tilde{\AAA}_C = \AAA_C.$$
For a morphism $c: C \lra C'$ in $\ccc$ and for $A \in \AAA_C$, $A' \in \AAA_{C'}$, we have
$$\tilde{\AAA}_c(A,A') = M_c(A,A')$$
and composition is defined in the obvious way.

Conversely, let $\AAA$ be an arbitrary $\uuu$-graded category. We define the naturally associated pseudofunctor
$$\aaa: \uuu \lra \underline{\Cat}(k): U \lra \aaa_U$$
where $\aaa_U$ is the category with $\Ob(\aaa_U) = \AAA_U$ and $\aaa(A,A') = \AAA_{1_U}(A,A')$ and where for $u: U \lra U'$ in $\uuu$, $\aaa(u) = M_u$ is the natural $\aaa_{U'}$-$\aaa_U$-bimodule with
$$M_u(A,A') = \AAA_u(A,A').$$
Then $(\uuu,\AAA) \cong (\uuu, \tilde{\aaa})$ and hence the correspondence between \emph{fibered} $\uuu$-graded categories and pseudofunctors $\uuu \lra \Cat(k)$ naturally extends to a correspondence between \emph{all} $\uuu$-graded categories and pseudofunctors $\uuu \lra \underline{\Cat}(k)$.

\subsection{Diagrams in the category of map-graded categories}
Another specification of the setup from \S \ref{pargroth2} occurs if we consider a pseudofunctor
$$(\uuu, \AAA): \ccc^{\op} \lra \Map: C \longmapsto (\uuu_C, \AAA_C)$$
mapping a morphism $c: C \lra C'$ in $\ccc$ to a map-graded functor
$$(\varphi_c, F_c): (\uuu_{C'}, \AAA_{C'}) \lra (\uuu_C, \AAA_C).$$
By Example \ref{exbimod}(2), this gives rise to an $\uuu_{C'}$-$\uuu_C$-bimodule $S^{\varphi_c}$ and a $\AAA_{C'}$-$S^{\varphi_c}$-$\AAA_C$-bimodule $M^{F_c}$ with
$$\begin{aligned}
S^{\varphi_c}(U,U') = \uuu_{C}(U, \varphi_c(U')); &&& M^{F_c}_s(A,A') = {\AAA_C}_s(A, F_c(A')).
\end{aligned}$$
In fact, there are natural pseudofunctors $\Cat^{\op} \lra \underline{\Cat}$ and $\Map^{\op} \lra \underline{\Map}$ so that we obtain a composed pseudofunctor 
$$(\uuu, \AAA): \ccc \lra \Map \lra \underline{\Map}: C \longmapsto (\uuu_C, \AAA_C).$$
For $c: C \lra C'$, $c': C' \lra C''$ in $\ccc$, we are given natural isomorphisms
$$(\eta, \theta): (\varphi_c \varphi_{c'}, F_c F_{c'}) = (\varphi_c, F_c)(\varphi_{c'}, F_{c'}) \lra (\varphi_{c' c}, F_{c' c})$$
satisfying the natural coherence axiom similar to \eqref{coh}.
From this we obtain natural isomorphisms
$$(S^{\varphi_{c'}} \circ S^{\varphi_c}, M^{F_{c'}} \otimes_{\AAA_{C'}} M^{F_c}) \lra (S^{\varphi_{c' c}}, M^{F_{c' c}})$$
in the following way. 
On the level of bifunctors, we have
$$S^{\varphi_{c'}} \circ S^{\varphi_c}(U, U'') = \coprod_{V \in \uuu_{C'}}\uuu_{C'}(V, \varphi_{c'}(U'')) \times \uuu_{C}(U, \varphi_c(V))/ \sim$$
$$\xymatrix{ {} \ar[r]_{\varphi_c \times 1} & {}}$$
$$\coprod_{V \in \uuu_{C'}} \uuu_{C}(\varphi_c(V), \varphi_c \varphi_{c'}(U'')) \times \uuu_{C}(U, \varphi_c(V))/ \sim$$
$$\xymatrix{ {} \ar[r]_-{m_{\uuu_{C}}} && {\uuu_C(U, \varphi_c \varphi_{c'}(U''))}}$$
$$\xymatrix{ {} \ar[r]_-{m_{\uuu_C}(\eta_{U''}, -)} && {\uuu_C(U, \varphi_{c' c}(U'')) = S^{\varphi_{c' c}}(U, U'').}}$$
Similarly, on the level of bimodules we have
$${M^{F_{c'}} \otimes_{\AAA_{C'}} M^{F_c}_r(A, A'') = \oplus_{r = [(s', s)], B} (\AAA_{C'})_{s'}(B, F_{c'}(A'')) \otimes (\AAA_C)_s(A, F_c(B))/ \sim}$$
$$\xymatrix{ {} \ar[r]_{F_c \otimes 1} & {}}$$
$${\oplus_{r = [(s', s)], B} (\AAA_{C})_{\varphi_c(s')}(F_c(B), F_c F_{c'}(A'')) \otimes (\AAA_C)_s(A, F_c(B))/ \sim}$$
$$\xymatrix{ {} \ar[r]_-{m_{\AAA_C}} && {(\AAA_C)_{\varphi_c(s') s}(A, F_c F_{c'}(A''))}}$$
$$\xymatrix{ {} \ar[r]_-{m_{\AAA_C}(\theta_{A''}, -)} && {(\AAA_C)_{\varphi_C(s') s}(A, F_{c' c}(A'')) = M^{F_{c' c}}_{s' s}(A, A'')}}$$
These maps determine the composition on the Grothendieck construction $(\tilde{\uuu}, \tilde{\AAA})$.

\begin{remark}
Note that the non-linear $\ccc$-graded category $\tilde{\uuu}$ is the original Grothendieck construction of $\uuu: \ccc \lra \Cat$ in the sense of Grothendieck. 
\end{remark}

\begin{remark}\label{remcatk}
For a pseudofunctor $\AAA: \ccc \lra \Cat(k)$ landing in the $2$-category of $k$-linear categories, functors and natural transformations, composing with $\Cat(k) \lra \underline{\Cat}(k)$ brings us in the situation of \S \ref{pardiaglin}. The Grothendieck construction $\tilde{\AAA}$ is the one we used in \cite{lowenmap} in order to turn a presheaf of $k$-algebras, or more generally a pseudofunctor $\AAA: \ccc \lra \Cat(k)$, into a $\ccc$-graded category and define its structured Hochschild complex $\CC_{\ccc}(\tilde{\AAA})$. By \cite{lowenvandenberghCCT}, this complex computes the natural Hochschild cohomology of $\AAA$ generalized from Gerstenhaber and Schack's Hochschild cohomology of presheaves of algebras \cite{gerstenhaberschack2} \cite{gerstenhaberschack1}.
\end{remark}

\begin{remark}
Let $(\uuu, \AAA): \ccc \lra \Map$ be a pseudofunctor. In a completely similar fashion, the $\uuu_{C'}$-$\uuu_{C}$-bifunctors $S_{\varphi_c}$ and $\AAA_{C'}$-$S_{\varphi_c}$-$\AAA_{C}$-bimodules $M_{F_c}$ from Example \ref{exbimod}(1) give rise to a pseudofunctor
$$(\uuu', \AAA'): \ccc \lra \Map \lra \underline{\Map}.$$
\end{remark}

\subsection{Base change}\label{parbase}

Let $\uuu: \ccc \lra \underline{\Cat}$ be a pseudofunctor as in \ref{pargroth1} and let $\Phi: \ddd \lra \ccc$ be an arbitrary functor. There is a resulting composed pseudofunctor
$$\uuu^{\Phi}: \ddd \lra \ccc \lra \underline{\Cat}: D \longmapsto \uuu_{\Phi(D)}$$
with an associated category $\tilde{\uuu}^{\Phi}$.
The natural maps
$$\Ob(\tilde{\uuu}^{\phi}) = \coprod_{D \in \ddd} \uuu_{\Phi(D)} \lra \coprod_{C \in \ccc} \uuu_C = \Ob(\tilde{\uuu})$$
and
$$\tilde{\uuu}^{\Phi}((D,U), (D', U')) \lra  \tilde{\uuu}((\Phi(D), U), (\Phi(D'), U))$$ given by
$$\coprod_{d \in \ddd(D,D')} S_{\Phi(d)}(U,U') \lra \coprod_{c \in \ccc(\Phi(D), \Phi(D'))} S_c(U,U')$$
give rise to a functor \label{phi}
\begin{equation}
\varphi: \tilde{\uuu}^{\Phi} \lra \tilde{\uuu}.
\end{equation}

\begin{proposition}\label{proppseudoinj}
Let $\uuu$ and $\Phi$ be as above. If $\nnn_1(\Phi)$ is injective, then so is $\nnn_1(\varphi)$.
\end{proposition}

\begin{proposition}\label{proppseudocover}
Let $\uuu$ be as above and consider a collection $(\Phi_i: \ddd_i \lra \ccc)_{i \in I}$ of functors. If, for $n \in \N \cup \{ \infty \}$, the collection constitutes an $n$-cover in $\Cat$, then so does the collection $(\tilde{\Phi}_i: \tilde{\uuu}^{\Phi_i} \lra \tilde{\uuu})_{i \in I}$.
\end{proposition}

Now let $(\uuu, \AAA): \ccc \lra \underline{\Map}$ be a pseudofunctor as in \S \ref{pargroth2} and $\Phi: \ddd \lra \ccc$ a functor. We now obtain a composed pseudofunctor
$$(\uuu^{\Phi}, \AAA^{\Phi}): \ddd \lra \ccc \lra \underline{\Map}$$
with an associated $\tilde{\uuu}^{\Phi}$-graded category $\tilde{\AAA}^{\Phi}$.

\begin{proposition}
With the notations of Proposition \ref{propcart}, we have $$\tilde{\AAA}^{\Phi} = (\tilde{\AAA})^{\varphi}.$$
In particular there is a natural cartesian map-graded functor $$\tilde{\Phi} = (\varphi, \delta^{\varphi, \tilde{\AAA}}) : (\tilde{\uuu}^{\Phi}, \tilde{\AAA}^{\Phi}) \lra (\tilde{\uuu}, \tilde{\AAA}).$$
\end{proposition}

Next we use base changes to transform a pseudofunctor $(\uuu, \AAA): \ccc \lra \underline{\Map}$ into a functor
$$(\uuu^{\ast}, \AAA^{\ast}): \ccc^{\ast} \lra \Map$$
of Grothendieck constructions.
We first define the category $\ccc^{\ast}$ to be the arrow category $\ccc \rightarrow_S e$ for the unique $e$-$\ccc$-bifunctor $S$ with $S(C, \ast) = \{ \ast \}$. 
For every $C \in \ccc$, we obtain a composition
$$(\uuu|_C, \AAA|_C): \ccc/C \lra \ccc \lra \underline{\Map}$$
and a cartesian map-graded functor
$(\tilde{\uuu}|_C, \tilde{\AAA}|_C) \lra (\tilde{\uuu}, \tilde{\AAA})$. For every $c: C' \lra C$ in $\ccc$ we have a natural functor $\ccc/C' \lra \ccc/C$ and an induced cartesian map-graded functor
$(\tilde{\uuu}|_{C'}, \tilde{\AAA}|_{C'}) \lra (\tilde{\uuu}|_{C}, \tilde{\AAA}|_{C})$ and it is not hard to organize these data into a functor
$$(\uuu^{\ast}, \AAA^{\ast}): \ccc^{\ast} \lra \Map_{sc}: C \longmapsto (\tilde{\uuu}|_C, \tilde{\AAA}|_C)$$
where we define $(\tilde{\uuu}|_{\ast}, \tilde{\AAA}|_{\ast}) = (\tilde{\uuu}, \tilde{\AAA})$.

\begin{proposition} \label{propcstar} Suppose $\ccc$ has finite products.
Let $(C_i)_{i \in I}$ be a collection of objects in $\ccc$ such that for every $C \in \ccc$ there exists a map $C \lra C_i$ for some $i$. The composed functor
$$\CC: \ccc^{\ast} \lra \Map_{sc} \lra B_{\infty}: C \longmapsto \CC_{\tilde{\uuu}|_C}(\tilde{\AAA}|_C)$$
satisfies the sheaf property with respect to the collection of maps $(C_i \lra \ast)_{i \in I}$ in $\ccc^{\ast}$.
\end{proposition}

\begin{proof}
By Example \ref{excover}(2), the collection $(\ccc/C_i \lra \ccc)_{i \in I}$ constitutes an $\infty$-cover of $\ccc$ in $\Cat$. By proposition \ref{proppseudocover}, the induced restriction maps in $\Map_{sc}$ also form an $\infty$-cover. Further, it is readily seen that the pullback of $C_i \lra \ast$ and $C_j \lra \ast$ in $\ccc^{\ast}$ is given by $C_i \times C_j \lra \ast$ for the product $C_i \times C_j$ in $\ccc$. The pullback of $(\tilde{\uuu}|_{C_i}, \tilde{\AAA}|_{C_i}) \lra (\tilde{\uuu}, \tilde{\AAA})$ and $(\tilde{\uuu}|_{C_j}, \tilde{\AAA}|_{C_j}) \lra (\tilde{\uuu}, \tilde{\AAA})$ is given by
$(\tilde{\uuu}|_{C_i \times C_j}, \tilde{\AAA}|_{C_i \times C_j}) \lra (\tilde{\uuu}, \tilde{\AAA}).$ Thus, the result follows from Theorem \ref{mainsheaf}.
\end{proof}

\subsection{Generalized arrow categories}\label{pargenarrow}
We can cast the arrow category construction from \S \ref{pararrow} in the setup from \S \ref{pargroth2}. Consider the  path category 
$$\ccc' = \langle \xymatrix{ {0} \ar[r]_{\leq} & {1} } \rangle$$
With the notations of \S \ref{pararrow}, we obtain a pseudofunctor
$$(\www', \CCC'): \ccc' \lra \underline{\Map}$$
with 
$$\begin{aligned} \www'_0 = \vvv &&& \www'_1 = \uuu &&& \www'(0 \leq 1) = S \\ \CCC'_0 = \BBB &&& \CCC'_1 = \AAA &&& \CCC'(0 \leq 1) = M \end{aligned}$$
Then the Grothendieck constructions amounts to arrow categories:
$$\begin{aligned} \tilde{\www}' = (\vvv \rightarrow_{S} \uuu); &&& \tilde{\CCC}' = (\BBB \rightarrow_{M} \AAA). \end{aligned}$$

We can generalize the arrow category in the following way. Consider the path category
$$\ccc = \langle \xymatrix{ 0 \ar[r]_{\leq} & 1 \ar[r]_{\leq} & \dots \ar[r]_-{\leq} & {n-1} \ar[r]_-{\leq} & n} \rangle.$$
We will call a category isomorphic to a such a category $\ccc$ for some $n$ a \emph{chain category}.
Then a pseudofunctor
$$(\www, \CCC): \ccc \lra \underline{\Map}$$
takes values 
$$\begin{aligned} \www_i &&& \www(i \leq j) = S_{ij} \\ \CCC_i &&& \CCC(i \leq j) = M_{ij} \end{aligned}$$
and is equipped with isomorphisms
$$(\phi_{kji}, F_{kji}): (S_{jk} \circ S_{ij}, M_{jk} \otimes_{\AAA_j} M_{ij}) \lra (S_{ik}, M_{ik}).$$

\begin{example}
Choosing arbitrary graded categories $(\www_i, \CCC_i)$, bifunctors $S_{i, i+1}$ and bimodules $M_{i, i+1}$, it is possible to define the remaining data as generalized compositions and tensor products
$$S_{ij} = S_{j-1,j} \circ \dots \circ S_{i, i+1}$$and
$$M_{ij} = M_{j-1,j} \otimes \dots \otimes M_{i, i+1}.$$
\end{example}

Now let $\Phi^{ij}: \ccc^{ij} \subseteq \ccc$ be the subcategory 
$$\ccc^{ij} = \langle \xymatrix{ i \ar[r]_-{\leq} & i+1 \ar[r]_{\leq} & \dots \ar[r]_{\leq} & {j-1} \ar[r]_-{\leq} & j} \rangle$$
and put $(\www^{ij}, \CCC^{ij}) = (\www^{\Phi^{ij}}, \CCC^{\Phi^{ij}})$.
We thus obtain natural injective cartesian map-graded functors $\tilde{\Phi}^{ij}: (\tilde{\www}^{ij}, \tilde{\CCC}^{ij}) \lra (\tilde{\www}, \tilde{\CCC})$ as in \S \ref{parbase}.

On $\ccc^{01}$, we can define a new pseudofunctor 
$$(\www', \CCC'): \ccc^{01} \lra \underline{\Map}$$
with
$$\begin{aligned} \www'_0 =  \www_0 &&& \www'_1 = \www^{1n} &&& \www(0 \leq 1) = S \\ \CCC'_0 = \CCC_0 &&& \CCC'_1 = \CCC^{1n} &&& \CCC'(0 \leq 1) = M \end{aligned}$$
where $S$ is the natural $\tilde{\www}^{1n}$-$\www_0$-bimodule and $M$ the natural $\tilde{\CCC}^{ij}$-$S$-$\CCC_0$ bimodule determined by $\tilde{\Phi}^{1n}$, $\tilde{\Phi}^{00}$ and the identity bimodule on $(\tilde{\www}, \tilde{\CCC})$ (we have $(\tilde{\www}^{00}, \tilde{\CCC}^{00}) \cong (\www_0, \CCC_0)$). Thus, we have
$$S(0, i) = S_{0i}; \hspace{1,5cm} M(0,i) = M_{0i}.$$
It is readily seen, for instance using the analysis from \S \ref{pararrowthin}, that
$$(\tilde{\www}, \tilde{\CCC}') \cong (\tilde{\www}', \tilde{\CCC}').$$
By induction on $n$, we can thus apply the results of \S \ref{pararrow}.

Alternatively, we can take a global approach. Let $\Phi: \mathrm{ob}\ccc \subseteq \ccc$ be the subcategory consisting of the objects and identity morphisms of $\ccc$. 
There is a corresponding injective cartesian morphism $\tilde{\Phi}: (\tilde{\www}^{\Phi}, \tilde{\CCC}^{\Phi}) \lra (\tilde{\www}, \tilde{\CCC})$. For every individual object $i \in \ccc$ we have a futher subcategory $\Phi^i: i \subseteq \mathrm{ob}\ccc$ corresponding to $\tilde{\Phi}^i: (\www_i, \CCC_i) \lra (\tilde{\www}^{\Phi}, \tilde{\CCC}^{\Phi})$.
Clearly, the complement of $\tilde{\www}^{\Phi} \subseteq \tilde{\www}$ constitutes an ideal which we denote by $\sss$. 
From \eqref{eqhoch3} we thus obtain an exact sequence of Hochschild complexes:
$$\xymatrix{ 0 \ar[r] & {\CC_{\tilde{\www}}(\tilde{\CCC}, (1_{\tilde{\CCC}})_{\sss})} \ar[r] & {\CC_{\tilde{\www}}(\tilde{\CCC})} \ar[r] & {\oplus_{i = 0}^n \CC_{\www_i}(\CCC_i)} \ar[r] & 0}$$

Now consider a general pseudofunctor $(\uuu, \AAA): \ccc \lra \underline{\Map}$. We can use the internal structure of $\ccc$ in order to describe $(\tilde{\uuu}, \tilde{\AAA})$ as an arrow category. 
Let $\zzz$ be a thin ideal in $\ccc$ with associated subcategories $\ccc_0 \subseteq \ccc$ of ``objects below $\zzz$'' and $\ccc_1$ of ``objects above $\zzz$'' as described in \S \ref{pararrowthin} and suppose $\Ob(\ccc) = \Ob(\ccc_0) \cup \Ob(\ccc_1)$ such that $\ccc \cong (\ccc_0 \rightarrow_Z \ccc_1)$ for the natural restriction of $1_{\ccc}$ to a $\ccc_1$-$\ccc_0$-bifunctor $Z$.
The inclusions $\ccc_0 \subseteq \ccc$ and $\ccc_1 \subseteq \ccc$ give rise to injective cartesian graded functors
$$(\tilde{\uuu}_0, \tilde{\AAA}_0) \lra (\tilde{\uuu}, \tilde{\AAA}); \hspace{1,5cm} (\tilde{\uuu}_1, \tilde{\AAA}_1) \lra (\tilde{\uuu}, \tilde{\AAA}).$$
Consider the  natural restriction of $1_{\tilde{\uuu}}$ to a  $\tilde{\uuu}_1$-$\tilde{\uuu}_0$-bifunctor $T$ and the natural restriction of $1_{\tilde{\AAA}}$ to a $\tilde{\AAA}_1$-$T$-$\tilde{\AAA}_0$-bimodule $N$. 
Put $\tilde{\zzz} = \coprod S_z(U',U)$ running over $z: C' \lra C$ in $\zzz$, $U \in \uuu_C$, $U' \in \uuu_{C'}$.

\begin{proposition}\label{proparrowc}
The thin ideal $\tilde{\zzz}$ in $\tilde{\uuu}$ gives rise to isomorphisms
$$\tilde{\uuu} \cong (\tilde{\uuu}_0 \rightarrow_T \tilde{\uuu}_1); \hspace{1,5cm} \tilde{\AAA} \cong (\tilde{\AAA}_0 \rightarrow_N \tilde{\AAA_1}).$$
\end{proposition}

\begin{remark}
Note that a generalized arrow category over
$$\ccc = \langle \xymatrix{ 0 \ar[r]_{\leq} & 1 \ar[r]_{\leq} & 2} \rangle.$$
is used in the proof of \cite[\S 4.6, Theorem (d)]{kellerdih}.
\end{remark}

\subsection{Covers by arrow categories}\label{parcoverarrow}
The example from \S \ref{pargenarrow} suggests how we can view some map-graded categories as being assembled from the primary arrow category construction, which can be seen as a certain way of glueing categories by means of a bimodule. Another way to glue map-graded categories is along covers of of underlying grading categories as described in \S \ref{parstackmap}. More generally, we obtained a sheaf of Hochschild complexes on $\Map_{sc}$ in \S \ref{parsheafhoch} and in particular Mayer-Vietoris sequences in \S \ref{parMV}. In this section we explain how these different  tools can be combined.

Let $(\ccc, \leq)$ be (the category associated to) a finite poset. For elements $a, b \in \ccc$, we denote $a \sqsubset b$ if $a \leq b$ and if $a \leq c \leq b$ for $c \in \ccc$ implies $a = c$ of $c = b$. Let $t_0, \dots, t_n$ be the maximal elements of $\ccc$ and $s_0, \dots, s_m$ the minimal elements. For every composition chain 
$$s_i = c_0 \sqsubset c_1 \sqsubset \dots c_k \sqsubset c_{k+1} \dots \sqsubset c_p = t_j.$$
we obtain a generated subcategory which is a chain category in the sense of \S \ref{pargenarrow}. 
Every chain of elements
$$a_0 \leq a_1 \leq a_2 \leq \dots \leq a_l$$
can be refined and fitted into a composition chain
$$s_i = c_0 \sqsubset \dots \sqsubset c_n = a_0  \sqsubset  \dots  \sqsubset c_{m} = a_l \sqsubset c_{m+1} \dots \sqsubset c_p = t_j$$
containing all the elements $a_k$. Thus, the collection of subcategories $\ddd \subseteq \ccc$ generated by composition chains of $\ccc$ constitutes an $\infty$-cover of $\ccc$.
Clearly, the intersection of two such categories is itself a chain category, of strictly smaller length.

Let $(\Phi_i: \ddd_i \subseteq \ccc)_{i \in I}$ be the collection of chain categories generated by composition chains of $\ccc$. Put $\ddd_{ij} = \ddd_i \cap \ddd_j$ and $\Phi_{ij}: \ddd_{ij} \lra \ccc$, $\Phi_{ij}^i: \ddd_{ij} \lra \ddd_i$.

Now consider a pseudofunctor $$(\www, \CCC): \ccc \lra \underline{\Map}.$$
Put $(\www_{i}, \CCC_{i}) = (\www^{\Phi_i}, \CCC^{\Phi_i})$ and $(\www_{ij}, \CCC_{ij}) = (\www^{\Phi_{ij}}, \CCC^{\Phi_{ij}})$. We obtain a pullback diagram of injective cartesian map-graded functors 
\begin{equation}\label{MVsquare}
\xymatrix{ {(\tilde{\www}_{i}, \tilde{\CCC}_{i})} \ar[r]^{\tilde{\Phi}_{i}} & {(\tilde{\www}, \tilde{\CCC})} \\ {(\tilde{\www}_{ij}, \tilde{\CCC}_{ij})} \ar[u]^{\tilde{\Phi}^{i}_{ij}} \ar[r]_{\tilde{\Phi}^j_{ij}} & {(\tilde{\www}_{j}, \tilde{\CCC}_{j}).} \ar[u]_{\tilde{\Phi}_j} }
\end{equation}
By Proposition \ref{proppseudocover}, the $\tilde{\Phi}_i$ constitute an $\infty$-cover in $\Map$ and thus we can use Theorem \ref{mainsheaf} to relate the different Hochschild complexes. 
By the sheaf property, we obtain an exact sequence
$$0 \lra \CC_{\tilde{\www}}(\tilde{\CCC}) \lra \oplus_{i \in I} \CC_{\tilde{\www}_i}(\tilde{\CCC}_i) \lra \oplus_{\{i,j\}} \CC_{\tilde{\www}_{ij}}(\tilde{\CCC}_{ij}).$$
Alternatively, we can proceed inductively by first isolating one composition chain category $\ddd_0$ and defining $\ddd_1$ to be the subcategory generated by all the other composition chains. Both $\ddd_1$ and $\ddd_{01} = \ddd_0 \cap \ddd_1$ are covered by strictly fewer composition chain categories than $\ddd$, and the occuring chains do not increase in length. Proceeding in a similar fashion as before, we then obtain a Mayer-Vietoris sequence as in \S \ref{parMV}:
$$0 \lra \CC_{\tilde{\www}}(\tilde{\CCC}) \lra \CC_{\tilde{\www}_0}(\tilde{\CCC}_0) \oplus \CC_{\tilde{\www}_1}(\tilde{\CCC}_1) \lra \CC_{\tilde{\www}_{01}}(\tilde{\CCC}_{01}) \lra 0.$$

\begin{remark}
The approach discussed in this section can be extended from a poset to a \emph{delta}, i.e a category $\ccc$ in which the arrows go only one way: for $C, C' \in \ccc$, we have $\ccc(C, C') = \varnothing$ or $\ccc(C', C) = \varnothing$.
\end{remark}

\begin{example}\label{expreinj}
Put $\ccc = \{ s, t_0, t_1 \}$ with $s \leq t_0$ and $s \leq t_1$ and consider $(\www, \CCC): \ccc \lra \underline{\Map}$ with
$$\begin{aligned}
\www(s) = \www_s &&& \www(t_0) = \www_0 &&& \www(t_1) = \www_1 &&& \www(s \leq t_0) = S_0 &&& \www(s \leq t_1) = S_1 \\
\CCC(s) = \CCC_s &&& \CCC(t_0) = \CCC_0 &&& \CCC(t_1) = \CCC_1 &&& \CCC(s \leq t_0) = M_0 &&& \CCC(s \leq t_1) = M_1
\end{aligned}$$
The category $\ccc$ is covered by the two chain categories
$$\ddd_0 = \langle s \leq t_0 \rangle \hspace{1,5cm}\ddd_1 = \langle s \leq t_1 \rangle$$
and we have $\ddd_{01} = \ddd_0 \cap \ddd_1 = \langle s \rangle$. We denote the inclusions by $\Phi_i$, $\Phi_{ij}$, $\Phi_{ij}^i$ as before. The diagram \eqref{MVsquare} is given by
\begin{equation}
\xymatrix{ {(\www_s \rightarrow_{S_0} \www_0, \CCC_s \rightarrow_{M_0} \CCC_0)} \ar[r]^-{\tilde{\Phi}_{0}} & {(\tilde{\www}, \tilde{\CCC})} \\ {(\tilde{\www}_{s}, \tilde{\CCC}_{s})} \ar[u]^{\tilde{\Phi}^{0}_{01}} \ar[r]_-{\tilde{\Phi}^1_{01}} & {(\www_s \rightarrow_{S_1} \www_1, \CCC_s \rightarrow_{M_1} \CCC_1)} \ar[u]_{\tilde{\Phi}_1} }
\end{equation}
and we obtain a Mayer-Vietoris sequence
$$0 \lra \CC_{\tilde{\www}}(\tilde{\CCC}) \lra \oplus_{i = 0}^1\CC_{\www_s \rightarrow_{S_i} \www_i}(\CCC_s \rightarrow_{M_i} \CCC_i) \lra \CC_{\www_s}(\CCC_s) \lra 0.$$
Note that this example also fits into the setup from Proposition \ref{propcstar}, with $C_0 = t_0$, $C_1 = t_1$ and $C_0 \times C_1 = s$.
\end{example}

\subsection{Comparison of pseudofunctors}\label{parcomp}

Let $\ccc$ be a delta (see Definition \ref{defdelta}). Now let us start from a pseudofunctor
$$(\uuu, \AAA): \ccc \lra \Map_{sc}.$$
For every $c: C' \lra C$ we have an associated subcartesian graded functor $$(\varphi_c, F_c): (\uuu_{C'}, \AAA_{C'}) \lra (\uuu_C, \AAA_C)$$ with associated $\uuu_C$-$\uuu_{C'}$-bifunctor $S_c = S_{\varphi_c}$ and
$\AAA_C$-$S_{\varphi_c}$-$\AAA_{C'}$-bimodule $M_c = M_{\varphi_c}$ from Example \ref{exbimod}.
Let $\ccc^{\ast}$ be the category $\ccc$ with terminal object $\ast$ and morphisms $\ast_C: C \lra \ast$ attached as described in \S \ref{parbase} and let $(\uuu^{\star}, \AAA^{\star}): \ccc^{\ast} \lra \Map_{sc}$ be any pseudofunctor such that the restriction to $\ccc$ equals $(\uuu, \AAA)$.

\begin{example}\label{exstar}
We can consider $(\uuu^{\star}, \AAA^{\star}): \ccc^{\ast} \lra \Map_{sc}$ with $(\uuu^{\star}_C, : C \longmapsto (\uuu_C, \AAA_C)$
with $(\uuu^{\star}_{\ast}, \AAA^{\star}_{\ast}) = (\tilde{\uuu}, \tilde{\AAA})$ and the graded functors $(\uuu_C, \AAA_C) \lra (\tilde{\uuu}, \tilde{\AAA})$ induced by the natural functors $e \lra \ccc^{\ast}: \ast \lra C$.
\end{example}
As described in \S \ref{parbase}, we also obtain an associated pseudofunctor
$$(\uuu^{\ast}, \AAA^{\ast}): \ccc^{\ast} \lra \Map_{sc}$$
with subcartesian graded functors
$$(\tilde{\varphi}_c, \tilde{F}_c): (\tilde{\uuu}|_{C'}, \tilde{\AAA}|_{C'}) \lra (\tilde{\uuu}|_{C}, \tilde{\AAA}|_{C})$$
and
$$(\tilde{\varphi}_{\ast_C}, \tilde{F}_{\ast_C}): (\tilde{\uuu}|_{C}, \tilde{\AAA}|_{C}) \lra (\tilde{\uuu}, \tilde{\AAA}).$$
In this section we compare the Hochschild complexes associated to $(\uuu^{\star}, \AAA^{\star})$ and $(\uuu^{\ast}, \AAA^{\ast})$.
For every $C \in \ccc$, the functor $e \lra \ccc/C: \ast \longmapsto (1_C: C \rightarrow C)$ gives rise to a natural graded functor
$$(\varphi_C, F_C): (\uuu_C, \AAA_C) \lra (\tilde{\uuu}|_C, \tilde{\AAA}|_C).$$
Note that $(\tilde{\varphi}_c, \tilde{F}_c)(\varphi_{C'}, F_{C'}) \neq (\varphi_C, F_C)(\varphi_c, F_c)$.
Let $(\theta, H): (\tilde{\uuu}, \tilde{\AAA}) \lra (\uuu^{\star}_{\ast}, \AAA^{\star}_{\ast})$ be the natural subcartesian functor with
$$(\theta, H)( \tilde{\varphi}_{\ast_C}, \tilde{F}_{\ast_C})(\varphi_C, F_C)  = (\varphi^{\star}_{\ast_C},  F^{\star}_{\ast_C}).$$
By Example \ref{exdelta}, the Grothendieck construction of $(\uuu^{\star}, \AAA^{\star})$ can be described as an arrow category, yielding a commutative diagram
\begin{equation} \label{eqpar}
\xymatrix{ {\CC_{\tilde{\uuu}^{\star}}(\tilde{\AAA}^{\star})} \ar[r]^{\cong} \ar[d]_{\varphi_{\tilde{\uuu}}^{\ast}} & {\CC_{\uuu^{\star}_{\ast}}(\AAA^{\star}_{\ast})} \ar[ld]^{H^{\ast}} \ar[d]^{(F^{\star}_{\ast_C})^{\ast}} \\
{\CC_{\tilde{\uuu}}(\tilde{\AAA})} \ar[d]_{(\tilde{F}_{\ast_C})^{\ast}} & {\CC_{\uuu_C}(\AAA_C)} \\ {\CC_{\tilde{\uuu}|_C}(\tilde{\AAA}|_C)} \ar[ru]_{F^{\ast}_C} }
\end{equation}

We will prove the following:

\begin{theorem}\label{thmmaincomp}
\begin{enumerate}
\item For every $c: C' \lra C$ in $\ccc$, the natural diagram
\begin{equation} \label{eqcompat}
\xymatrix{ {\CC_{\tilde{\uuu}|_C}(\tilde{\AAA}|_C)} \ar[d]_{\tilde{F}_c^{\ast}} \ar[r]^{F_C^{\ast}} & {\CC_{\uuu_C}(\AAA_C)} \ar[d]^{F_c^{\ast}} \\
{\CC_{\tilde{\uuu}|_{C'}}(\tilde{\AAA}|_{C'})}  \ar[r]_{F_{C'}^{\ast}} & {\CC_{\uuu_{C'}}(\AAA_{C'})}}
\end{equation}
is commutative in the homotopy category of $B_{\infty}$-algebras, and the horizontal arrows are quasi-isomorphism.
\item The horizontal arrows naturally give rise to a morphism between pseudofunctors $\ccc \lra B_{\infty}$:
$$\CC_{\uuu^{\ast}|_{\ccc}}(\AAA^{\ast}|_{\ccc}) \lra \CC_{\uuu}(\AAA).$$
\item If the canonical morphism
$$\omega: 1_{\AAA^{\star}_{\ast}} \lra \RHom_{\tilde{\AAA}^{\op}}(M_H, M_H)$$ induces a quasi-isomorphism $\CC_{\uuu^{\star}_{\ast}}(\AAA^{\star}_{\ast}, \omega)$, then $H^{\ast}$ is a quasi-isomorphism and we obtain a morphism of pseudofunctors $\ccc^{\ast} \lra \mathsf{ho}(B_{\infty})$:
$$\CC_{\uuu^{\ast}}(\AAA^{\ast}) \lra \CC_{\uuu^{\star}}(\AAA^{\star})$$
in which the top horizontal arrow is given by $(H^{\ast})^{-1}$. This holds in particular is we choose $(\uuu^{\star}, \AAA^{\star})$ as in Example \ref{exstar}.
\end{enumerate}
\end{theorem}

As a consequence of the Theorem, suppose the conditions of Proposition \ref{propcstar} are fulfilled and $\CC_{\uuu^{\ast}}(\AAA^{\ast}): \ccc^{\ast} \lra B_{\infty}$ is a sheaf of Hochschild complexes. Then exact sequences of Hochschild complexes following from the sheaf property for $\CC^{\bullet}_{\uuu^{\ast}}(\AAA^{\ast})$ naturally translate to exact triangles in terms of  $\CC^{\bullet}_{\uuu^{\star}}(\AAA^{\star})$ in which we were originally interested. In particular, we naturally obtain Mayer-Vietoris exact triangles and their induced long exact cohomology sequences. Before proving the theorem, we start with a key example.

\begin{example}\label{excomp}
The first example is in the setup for the original arrow category, as described in the beginning of \S \ref{pargenarrow}. So $\ccc = \{ 0, 1\}$ with $0 \leq 1$. We consider $(\www, \ccc): \ccc \lra \Map_{sc}$ determined by $(\www_0, \CCC_0) = (\vvv, \BBB)$, $(\www_1, \ccc_1) = (\uuu, \AAA)$ and a single subcartesian graded functor
$$(\varphi, F): (\vvv, \BBB) \lra (\uuu, \AAA).$$
Using the bifunctor $S_{\varphi}$ and the bimodule $M_F$, we obtain the arrow category
$$(\tilde{\www}, \tilde{\CCC}) = (\vvv \rightarrow_{\varphi} \uuu, \BBB \rightarrow_F \AAA).$$
Obviously, $\ccc/1 \cong \ccc$, $\ccc/0 \cong e$,  and the diagram \eqref{eqcompat} reduces to
$$\xymatrix{ {\CC_{\vvv \rightarrow_{\varphi} \uuu}(\BBB \rightarrow_F \AAA)} \ar[r]^-{\varphi^{\ast}_{\uuu}} \ar[d]_{\varphi^{\ast}_{\vvv}} & {\CC_{\uuu}(\AAA)} \ar[d]^{F^{\ast}} \\
{\CC_{\vvv}(\BBB)} \ar[r]_1 & {\CC_{\vvv}(\BBB).}}$$
The fact that $\varphi^{\ast}_{\uuu}$ is a quasi-isomorphism and the diagram commutes in the homotopy category of $B_{\infty}$-algebras is the content of Proposition \ref{propcompat2}.
\end{example}

\begin{proof}
We will concentrate on the proof of (1). The proof of (3) then follows from diagram \eqref{eqpar} in which $\varphi_{\tilde{\uuu}}^{\ast}$ becomes a quasi-isomorphism by Theorem \ref{thmarrowqis}.

Let $C \in \ccc$ be fixed.
The category $\ccc/C$ is a delta with terminal object $1_C: C \lra C$. Let $e_C \subseteq \ccc/C$ be the category with single object $1_C$ and single morphism $1_C$. Let $\ddd_C \subseteq \ccc/C$ be the full subcategory consisting of all objects except $1_C$. For every $c: C' \lra C$, there is a unique morphism $c: c \lra 1_C$ in $\ccc/C$. According to Example \ref{exdelta}, we obtain a thin ideal
$\zzz = \{ c: c \lra 1_C \,\, |\,\, 1_C \neq c: C' \lra C \}$ and we have
$$\ccc/C \cong \ddd_C \rightarrow_Z e_C.$$
with $Z(c, 1_C) = \{c\}$.
The restriction of $(\uuu|_C, \AAA|_C)$ along $e_C \subseteq \ccc/C$ corresponds to the constant category $(\uuu_C, \AAA_C)$. Denote the restriction of $(\uuu|_C, \AAA|_C)$ along $\ddd_C \subseteq \ccc/C$ by $(\vvv_C, \BBB_C)$. 
Now consider $1_C \neq c: C' \lra C$. Since $\ccc$ is a delta,  there is a natural factorization $\ccc/C' \lra \ddd_C \lra \ccc/C$ from which we obtain a natural factorization of $(\tilde{\varphi}_c, \tilde{F}_c)$:
$$\xymatrix{ {(\tilde{\uuu}|_{C'}, \tilde{\AAA}|_{C'})}  \ar[r]_{(\tilde{\varphi}^0_c, \tilde{F}^0_c)} & {(\tilde{\vvv}_C, \tilde{\BBB}_C)} \ar[r]_{(\varphi_{\vvv}, F_{\BBB})} & {(\tilde{\uuu}|_{C}, \tilde{\AAA}|_{C}).}}$$
We further obtain a natural subcartesian graded functor $(\psi, G)$ fitting into commutative diagrams
\begin{equation} \label{eqdiagaux}
\xymatrix{ {(\tilde{\vvv}_C, \tilde{\BBB}_C)} \ar[r]^{(\psi, G)} & {(\uuu_C, \AAA_C)} \\ {(\uuu_{C'}, \AAA_{C'})} \ar[u]^{(\tilde{\varphi}_c^0, \tilde{F}_c^0)(\varphi_{C'}, F_{C'})} \ar[ru]_{(\varphi_c, F_c)} & }
\end{equation}
for $1_C \neq c: C' \lra C$. By Proposition \ref{proparrowc}, we have
$$\tilde{\uuu}|_C \cong (\tilde{\vvv}_C \rightarrow_{\psi} \uuu_C) \hspace{1,5cm} \tilde{\AAA}|_C \cong (\tilde{\BBB}_C \rightarrow_G \AAA_C)$$
where we use the bimodules $S_{\psi}$ and $M_{G}$. 
By Proposition \ref{propcompat2} (see also Example \ref{excomp}), we have a commutative diagram
$$\xymatrix{ {\CC_{\tilde{\uuu}|_C}(\tilde{\AAA}|_C)} \ar[r]^{F^{\ast}_C} \ar[d]_{F^{\ast}_{\BBB}} & {\CC_{\uuu_C}(\AAA_C)} \ar[dl]^{G^{\ast}} \\ {\CC_{\tilde{\vvv}_C}(\tilde{\BBB}_C)} & }$$
in which $F^{\ast}_C$ is a quasi-isomorphism. Composing this diagram with 
$$\xymatrix{ {\CC_{\tilde{\vvv}_C}(\tilde{\BBB}_C)} \ar[r]_-{(\tilde{F}^0_c)^{\ast}} & {\CC_{\tilde{\uuu}|_{C'}}(\tilde{\AAA}|_{C'})} \ar[r]_{F_{C'}^{\ast}} & {\CC_{\uuu_{C'}}(\AAA_{C'})}}$$
and making use of \eqref{eqdiagaux}, we obtain the commutative diagram \eqref{eqcompat} as desired.
\end{proof}

\begin{example}\label{exinj}
As an application, we look at the setup from \cite[\S 7.6]{lowenvandenberghhoch}. We consider the poset $\tilde{\Delta} = \{ I \,\, |\,\, I \subseteq \{ 1, 2, \dots, n\}$ ordered by reversed inclusion and the subposet $\Delta \subseteq \tilde{\Delta}$ of all $J \neq \varnothing$. We consider a stack $\tilde{\sss}$ of Grothendieck categories on $\tilde{\Delta}$ with exact restriction functors and fully faithful right adjoints, and denote $\sss$ for its restriction to $\Delta$. We further assume that the conditions (C1) and (C2) listed in \cite[\S 7.6]{lowenvandenberghhoch} are satisfied. Let $\Mod(\sss)$ denote the category of presheaf objects in $\sss$. The examples to have in mind are the stack of categories of sheaves of modules on a cover of a ringed space, and the stack of quasi-coherent sheaf categories on a finite affine cover of a separated scheme.

From $\tilde{\sss}$ we obtain a pseudofunctor of categories of injectives
$$\EEE: \tilde{\Delta} \longmapsto \Cat(k): I \longmapsto \EEE_I = \Inj(\tilde{\sss}(I)).$$
For $J \subseteq I$, we have a corresponding fully faithful functor
$$F_{IJ}: \EEE_I \lra \EEE_J.$$
Since these functors are fully faithful, there is a corresponding $(e, \EEE): \tilde{\Delta} \lra \Map_{sc}$ and the results from this section apply. 
The collection of objects $\{1\}$, $\{2\}$, $\dots$, $\{n\}$ in $\Delta$ satisfies the conditions in Proposition \ref{propcstar}, where further $\Delta^{\ast} \cong \tilde{\Delta}$ and the product of subsets $I$ and $J$ is given by $I \cap J$. We denote the pseudofunctor
$$\EEE^{\ast} = (\EEE|_{\Delta})^{\ast}: \Delta^{\ast} \lra \Cat(k): I \longmapsto \tilde{\EEE}|_I$$
where $\tilde{\EEE}|_I$ is the Grothendieck construction of the restriction of $\EEE$ to $\Delta/I = \{ J \in \Delta \,\, |\,\, I \subseteq J\}$ and $\tilde{\EEE}|_{\ast} = \tilde{\EEE}|_{\Delta}$ is the Grothendieck construction of $\EEE|_{\Delta}$. 
Thus, we conclude that $\CC(\EEE^{\ast})$ satisfies the sheaf property with respect to the collection of maps $\{i\} \lra \ast$ for $i \in \{ 1, \dots, n\}$. 
The first part of the proof of \cite[Theorem 7.7.1]{lowenvandenberghhoch} amounts to the verification of the condition in Theorem \ref{thmmaincomp} (3). 
We thus obtain from Theorem \ref{thmmaincomp} a morphism of pseudofunctors $\Delta^{\ast} \lra \mathsf{ho}(B_{\infty})$:
$$\CC(\EEE^{\ast}) \lra \CC(\EEE)$$
in which all component maps are quasi-isomorphisms.  The Mayer-Vietoris exact triangle for ringed spaces proved in \cite[Theorem 7.9.1]{lowenvandenberghhoch} becomes an immediate corollary of our theorem, and the proof given in \cite{lowenvandenberghhoch} is in fact a special case of the proof of Theorem \ref{thmmaincomp}.
\end{example}

\begin{remark}
The sheaves of Hochschild complexes we obtain in this paper naturally give rise to hypercohomology spectral sequences. However, if we start for instance from a ringed space $(X, \ooo_X)$, the site on which an associated sheaf of Hochschild complexes of Grothendieck constructions of categories of injectives lives is fundamentally different from the standard site associated to $X$. In Example \ref{exinj} this difference is ``brigded'' by a bimodule between the categories associated to the different suprema associated to a cover of a ringed space (on the one hand, the ringed space and on the other hand, the downset of the cover). A global approach along these lines (possibly combined with techniques from \cite{lowensheafhoch}) should lead to new Hochschild cohomology spectral sequences.

The construction of a Hochschild cohomology local-to-global spectral sequence for general ringed spaces based upon map-graded Hochschild cohomology and hypercoverings remains work in progress \cite{lowenvandenberghlocglob}.
\end{remark}

\def\cprime{$'$} \def\cprime{$'$}
\providecommand{\bysame}{\leavevmode\hbox to3em{\hrulefill}\thinspace}
\providecommand{\MR}{\relax\ifhmode\unskip\space\fi MR }
\providecommand{\MRhref}[2]{%
  \href{http://www.ams.org/mathscinet-getitem?mr=#1}{#2}
}
\providecommand{\href}[2]{#2}

\end{document}